\documentclass[11pt, noinfoline,final,a4paper]{imsart}

\arxiv{1706.00679}

\usepackage[utf8]{inputenc}
\usepackage[T1]{fontenc}
\usepackage[english]{babel}
\usepackage{textcomp}
\usepackage{dsfont}
\usepackage{hyperref}
\usepackage{color}
\usepackage{graphics}
\usepackage{epstopdf}
\usepackage{amsmath,amsthm,amssymb}
\usepackage{enumitem}
\usepackage{enumerate}
\usepackage{tikz,pgfplots}
\usepackage[norelsize,ruled,vlined,commentsnumbered,linesnumbered]{algorithm2e}
\usepackage{a4wide}

\makeatletter
\newcommand{\nosemic}{\renewcommand{\@endalgocfline}{\relax}}
\newcommand{\dosemic}{\renewcommand{\@endalgocfline}{\algocf@endline}}
\let\oldnl\nl
\newcommand{\nonl}{\renewcommand{\nl}{\let\nl\oldnl}}
\makeatother

\SetAlFnt{\footnotesize	\sffamily}

\newtheorem{theorem}{Theorem}
\newtheorem{lemma}[theorem]{Lemma}

\newtheorem{proposition}[theorem]{Proposition}
\newtheorem{definition}[theorem]{Definition}
\newtheorem{remark}{Remark}

\def\beq{\begin{equation}}
\def\eq{\begin{equation}}
\def\eeq{\end{equation}}
\def\qe{\end{equation}}
\def\beqn{\begin{eqnarray*}}
\def\eeqn{\end{eqnarray*}}
\def\bitem{\begin{itemize}}
\def\eitem{\end{itemize}}
\def\benum{\begin{enumerate}}
\def\eenum{\end{enumerate}}
\def\bmult{\begin{multline*}}
\def\emult{\end{multline*}}
\def\bcenter{\begin{center}}
\def\ecenter{\end{center}}

\DeclareMathOperator*{\argmax}{arg\, max}
\DeclareMathOperator*{\argmin}{arg\, min}

\DeclareMathOperator{\rank}{rank}

\def\cA{\mathcal{A}}

\def\cH{\mathcal{H}}

\def\cU{\mathcal{U}}
\def\cV{\mathcal{V}}

\def\bA{\boldsymbol{A}}

\def\bU{\boldsymbol{U}}

\def\bb{\mathbf{b}}

\def\bbC{\mathds{C}}

\def\bbE{\mathds{E}}
\def\bbF{\mathds{F}}

\def\bbH{\mathds{H}}

\def\bbK{\mathds{K}}
\def\bbL{\mathds{L}}

\def\bbP{\mathds{P}}

\def\bbR{\mathds{R}}
\def\bbS{\mathds{S}}
\def\bbT{\mathds{T}}
\def\bb\bU{\mathds{\bU}}
\def\bbV{\mathds{V}}

\def\bbZ{\mathds{Z}}

\newcommand{\E}{\operatorname{\mathds{E}}}
\renewcommand{\P}{\operatorname{\mathds{P}}}
\renewcommand{\bar}{\overline}
\renewcommand{\hat}{\widehat}
\renewcommand{\tilde}{\widetilde}

\newcommand{\Cov}{\operatorname{Cov}}

\def\\bUnif{\text{\bUnif}}

\newcommand{\1}{\mathds{1}}

\usepackage[textwidth=2.8cm, textsize=scriptsize]{todonotes} 
\date{\today}

\begin{document}

\begin{frontmatter}

\title{Testing Gaussian Process with Applications to Super-Resolution}
\runtitle{Spacing Test for Super-Resolution}

\begin{aug}
\author{\fnms{Jean-Marc} \snm{Aza\"is${}^\bullet$}\ead[label=e1]{jean-marc.azais@math.univ-toulouse.fr}},
\author{\fnms{Yohann} \snm{De Castro${}^{\dag{},\star{}}$}\ead[label=e2]{yohann.decastro@math.u-psud.fr}\ead[label=e4]{yohann.de-castro@inria.fr}}
\author{\and\fnms{St\'ephane} \snm{Mourareau${}^\circ$}\ead[label=e3]{stephane.mourareau@u-pem.fr}}

\affiliation{Universit\'e Paul Sabatier}
\address{${}^\bullet$Institut de Math\'ematiques de Toulouse, Universit\'e Paul Sabatier\\ 118 route de Narbonne, 31062 Toulouse, France.}

\affiliation{Universit\'e Paris-Sud}
\address{${}^\dag$Laboratoire de Math\'ematiques d'Orsay\\  Univ. Paris-Sud, CNRS,  Universit\'e Paris-Saclay, 91405 Orsay, France.
}  
\affiliation{INRIA}
\address{${}^\star$INRIA, Centre de Recherche de Paris, Équipe MoKaPlan\\  2 rue Simone Iff, 75012 Paris, France.
}  
\affiliation{Universit\'e Paris-Est Marne-la-Vall\'ee}
\address{${}^\circ$Laboratoire d'Analyse et de Math\'ematiques Appliqu\'ees, Univ. Paris-Est\\ Champs sur Marne, 77454 Marne La Vall\'ee, France.
}

\runauthor{Aza\"is, De Castro and Mourareau}
\end{aug}

\begin{abstract}
This article introduces exact testing procedures on the mean of a Gaussian process $X$ derived from the outcomes of $\ell_1$-minimization over the space of complex valued measures. The process $X$ can be thought as the sum of two terms: first, the convolution between some kernel and a target atomic measure (mean of the process); second, a random perturbation by an additive centered Gaussian process. The first testing procedure considered is based on a dense sequence of grids on the index set of~$X$ and we establish that it converges (as the grid step tends to zero) to a randomized testing procedure: the decision of the test depends on the observation $X$ and also on an independent random variable. The second testing procedure is based on the maxima and the Hessian of $X$ in a grid-less manner. We show that both testing procedures can be performed when the variance is unknown (and the correlation function of $X$ is known). These testing procedures can be used for the problem of deconvolution over the space of complex valued measures, and applications in frame of the Super-Resolution theory are presented. As a byproduct, numerical investigations may demonstrate that our grid-less method is more powerful (it~detects sparse alternatives) than tests based on very thin~grids. 

\end{abstract}

\begin{keyword}[class=MSC]
\kwd[Primary ]{62E15}
\kwd{62F03}
\kwd{60G15}
\kwd{62H10}
\kwd{62H15} 
\kwd[; secondary ]{60E05}
\kwd{60G10}
\kwd{62J05}
\kwd{94A08}
\end{keyword}

\begin{keyword}
\kwd{Hypothesis Testing}
\kwd{Gaussian Process}
\kwd{Kac-Rice formula}
\kwd{Super-Resolution}
\end{keyword}

\end{frontmatter}

\maketitle

{
\it
\footnotesize
\begin{center}
Preprint of \today
\end{center}
\vspace*{-0.5cm}
}

\section{Introduction}
\subsection{Grid-less spike detection through the ‘‘continuous'' LARS}

New testing procedures based on the outcomes of $\ell_1$ minimization methods have attracted a lot of attention in the statistical community. Of particular interest is the so-called ‘‘{\it Spacing test}\,'', that we referred to as $S^{\mathrm{ST}}$, based on the Least-Angle Regression Selection (LARS), that measures the significance of the addition of a new active variable along the LARS path, see \cite[Chapter 6]{hastie2015statistical} for further details. Specifically, one is testing the relative distance between consecutive ‘‘{\it knots}\,'' of the LARS, for instance~$\lambda_{1,P}$ and $\lambda_{2,P}$. The first knot $\lambda_{1,P}$ is the maximal correlation between a response variable $y\in\bbC^N$ and~$P$ predictors. The second knot $\lambda_{2,P}$ is then the correlation between some residuals and $P-1$ predictors, and so on. This approach is now well referenced among the regularized methods of high-dimensional statistics and it can be linked to minimizing the $\ell_1$-norm over $P$ coordinates, see for instance \cite[Chapter~6]{hastie2015statistical}. 

In this paper, we focus on $\ell_1$-minimization over the space of signed measures and we ask for testing procedures based on these solutions. Indeed, in deconvolution problems over the space of measures \cite{Bredis_Pikkarainen_13}\textemdash \textit{e.g.}, Super-Resolution or line spectral estimation \cite{Candes_FernandezGranda_14,fernandez2016super,Duval_Peyre_JFOCM_15,Tang_Bhaskar_Shah_Recht_13,DeCastro_Gamboa_12,Azais_DeCastro_Gamboa_15}\textemdash one may observe a noisy version of a convolution of a target discrete measure by some known kernel $K$ and one may be willing to infer on the target discrete measure. In this case, testing a particular measure is encompassed by testing the mean of some ‘‘correlation'' process~$Z$, see Section~\ref{sec:SR} for further details. 

\medskip

\begin{figure}[!h]
\includegraphics[width=0.45\textwidth]{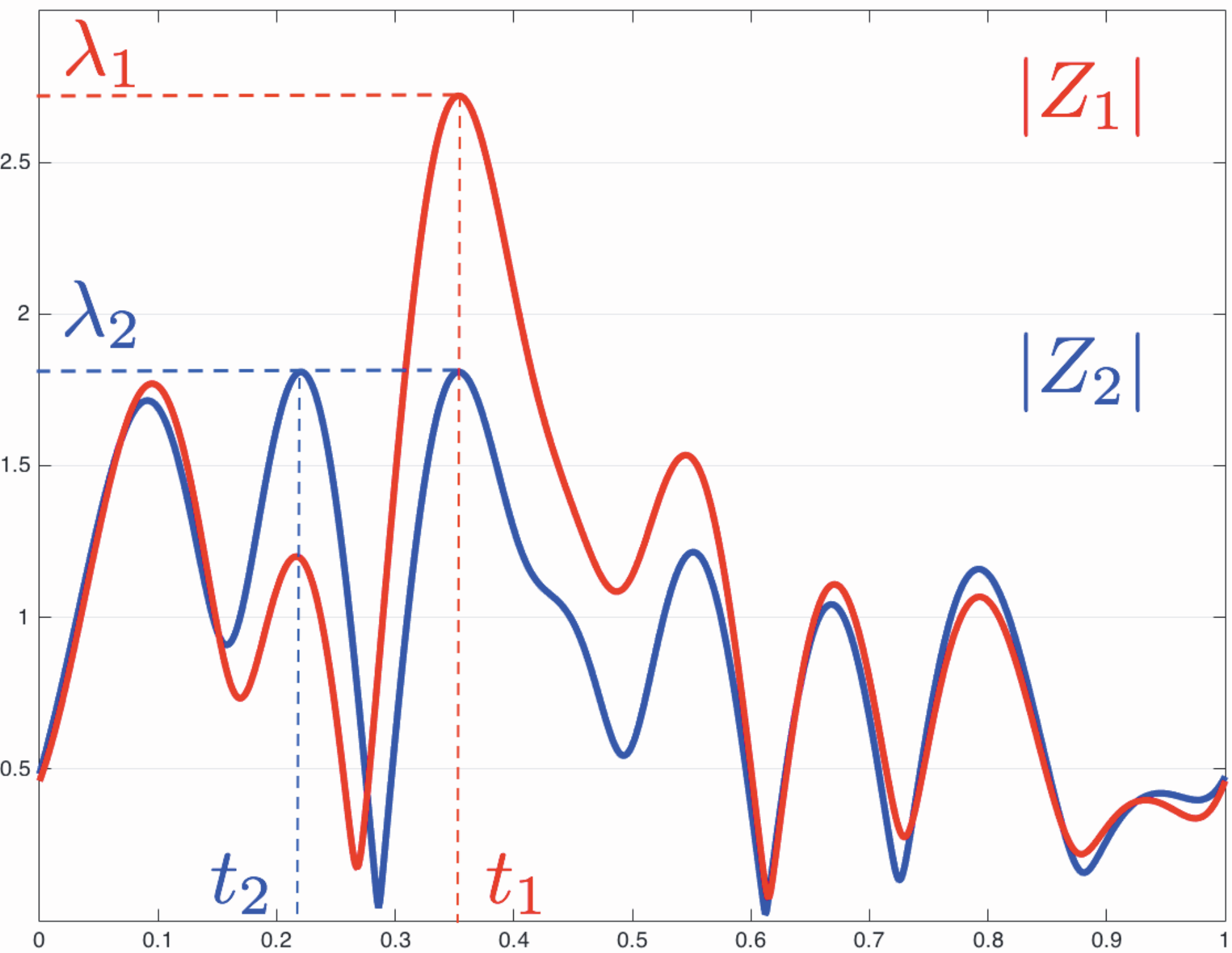}
\caption{ LARS for Super-Resolution:  we fit a Dirichlet kernel (which is the Point Spread Function of Super-Resolution) at the maximal correlation point $t_1$ until the maximal correlation in the residual is matched at a second point $t_2\neq t_1$.}
\label{fig:LARS_SR}
\end{figure}
In general deconvolution problems, remark that there is an uncountable number of predictors with valued in a hilbert space (not necessarily finite)\textemdash while there were~$P$ predictors previously when inferring on vectors of $\bbR^N$ in the high-dimensional statistics framework. Indeed, we are looking at correlations $Z(t)=\langle y,k(t)\rangle$ between a response variable $y$ and a ‘‘feature map'' $k(t)$ indexed by a continuum, say for instance~$t\in\bbK=[0,2\pi)$. In this case, the set of predictors is uncountable and given by $\{k(t)\,;\ t\in\bbK\}$. Furthermore, $k(t)$ is an element of the Reproducing Kernel Hilbert Space~$\cH$ (RKHS) defined by the convolution kernel $K$\textemdash assumed to be symmetric positive definite. In particular, the hilbert space $\cH$ can be infinite dimensional. As an example, assume that one observes some Fourier coefficients of some discrete measure on the torus~$[0,2\pi)$ and one is willing to infer on its support. A strategy would be to look at correlations between the response variable $y\in\bbC^{N}$ and the Fourier curve $k(t)=(\cos(kt)\pm\imath\sin(kt))_{-f_c\leq k\leq f_c}\in\bbC^N$ for some frequency cut-off~$f_c\geq1$ so that $N=2f_c+1$. It results in a complex valued correlation process $Z(t):=\langle y,k(t)\rangle=\sum y_k e^{\imath k t}$ indexed by $t\in[0,2\pi)$. In this case, the RKHS~$\cH$ has dimension $N$, the number of observed Fourier coefficients, and the convolution kernel is given by the Dirichlet kernel, see Section~\ref{sec:SR}. As an illustration, we present Figure~\ref{fig:LARS_SR} where we take $Z_1=Z$ and the red curve displays the absolute value of the correlation process $Z$. One can standardly show that $|Z|(t)$ is the likelihood of the model that consists in one spike at point~$t$. Therefore, its maximal value~$\lambda_1$ can be interpreted as the Maximum Likelihood for models with one spike. Its argument maximal point~$t_1$ is then the Maximum Likelihood Estimator and one may be willing to consider it as a first estimation of the target discrete measure's support. Then one can consider the residuals $Z_2=Z_1-a\langle y,k(t_1)\rangle$ where $a\in\bbC$ is the weight of the estimated signal chosen so that we get the blue curve of Figure~\ref{fig:LARS_SR}, namely a second support point $t_2$ should enter the model since the residuals $|Z_2|$ achieve their maximal absolute value at two locations, $t_1$ and~$t_2$. More details can be found in Section~\ref{sec:SRLARS}.

In this framework, the LARS algorithm does not return a sequence of entries (among $P$ possible coordinates) and phases as in high-dimensional statistics but rather a sequence of locations $t_1,t_2,\ldots$ (among the continuum~$\bbK=[0,2\pi)$) and phases. In this paper, we invoke the LARS to this framework\textemdash we referred to it as ‘‘continuous'' LARS\textemdash for which an uncountable number of active variables may enter the model. We present this extension in Section~\ref{sec:LARS} defining consecutive knots $(\lambda_1,\lambda_2)$. One can wonder:
\begin{itemize}
\item 
{\it 
Can the Spacing test be used in the frame of Super-Resolution? 
}
\item 
{\it 
Is there a grid-less procedure more powerful, in the sense of detecting spikes, than the Spacing tests constructed on thin grids?
}
\end{itemize}

\medskip

{
\begin{figure}[!h]
\includegraphics[width=0.85\textwidth]{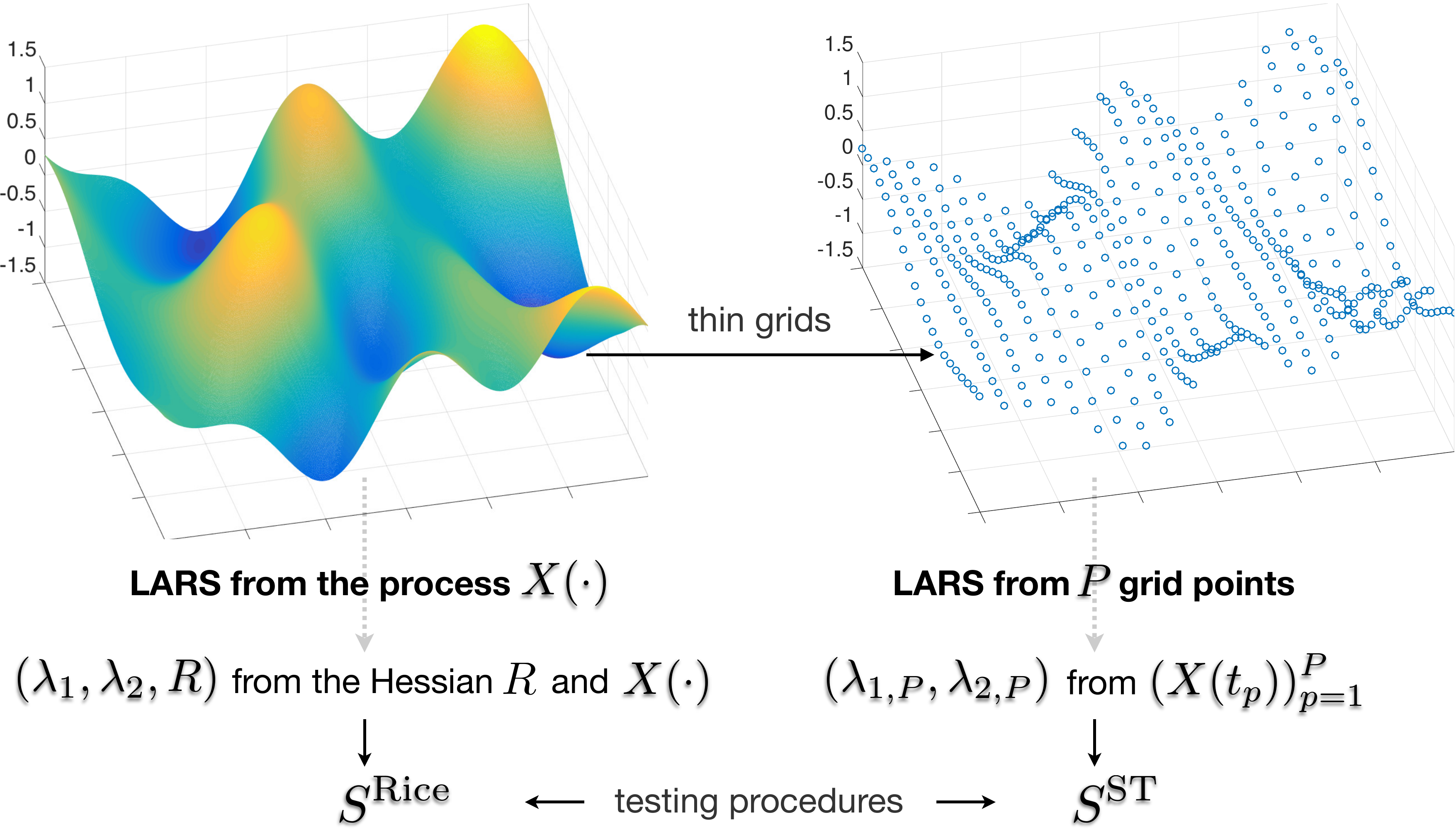}
\caption{The grid-less approach uses the Hessian and the first two ‘‘knots'' $(\lambda_1,\lambda_2)$ of the ‘‘continuous'' LARS  to build the test statistics $S^{\mathrm{Rice}}$. We compare it to the grid approach that builds a test statistics~$S^{\mathrm{ST}}$ using the knots $(\lambda_{1,P},\lambda_{2,P})$ computed from a $P$ points grid discretization $(X(t_p))_{p=1}^P$ of the continuous process $X$.}
\label{fig:grids_vs_continuum}
\end{figure}
}
Interestingly, as we will prove, the answer is no to the first question if no modifications of the test statistic is done. Furthermore, the way that the Spacing test can be fixed to be extended to a ‘‘grid-less'' frame gives a new testing procedure~$S^{\mathrm{Rice}}$ that accounts for the distance between consecutive knots $(\lambda_1,\lambda_2)$ with respect to value of the Hessian $R$ at some maximal point, see Figure~\ref{fig:grids_vs_continuum} for a global view on our approach. 

\subsection{A comparative study}
When the predictors are normalized, the Spacing test (ST) statistics is given by the expression
\[
S^{\text{ST}}(\lambda_{1,P},\lambda_{2,P}):=\frac{\overline\Phi(\lambda_{1,P})}{\overline\Phi(\lambda_{2,P})}
\] 
where $\overline\Phi=1-\Phi$ is the Gaussian survival function and~$\Phi$ the standard normal cumulative distribution function. In the framework of high-dimensional statistics, this statistics is exactly distributed w.r.t. a uniform law on~$[0,1]$ under the global null, namely $S^{\text{ST}}$ can be considered as the observed significance~\cite{taylor2014exact,ADCM16}. It is clear that one should not use this testing procedure in the Super-Resolution framework since there is no theoretical guarantees in this case. Yet the practitioner may be tempted to replace $(\lambda_{1,P},\lambda_{2,P})$ by $(\lambda_{1},\lambda_{2})$ given by the ‘‘continuous'' LARS. Unfortunately, this paper shows that the resulting test statistics~$S^{\text{ST}}$ is non conservative in this frame, \textit{i.e.}, it makes too many false rejections and one should avoid using it in practice, see the green line in Figure~\ref{unbiaised_naive}.

\begin{figure}[!h]
\includegraphics[width=4.5cm,height=4.5cm]{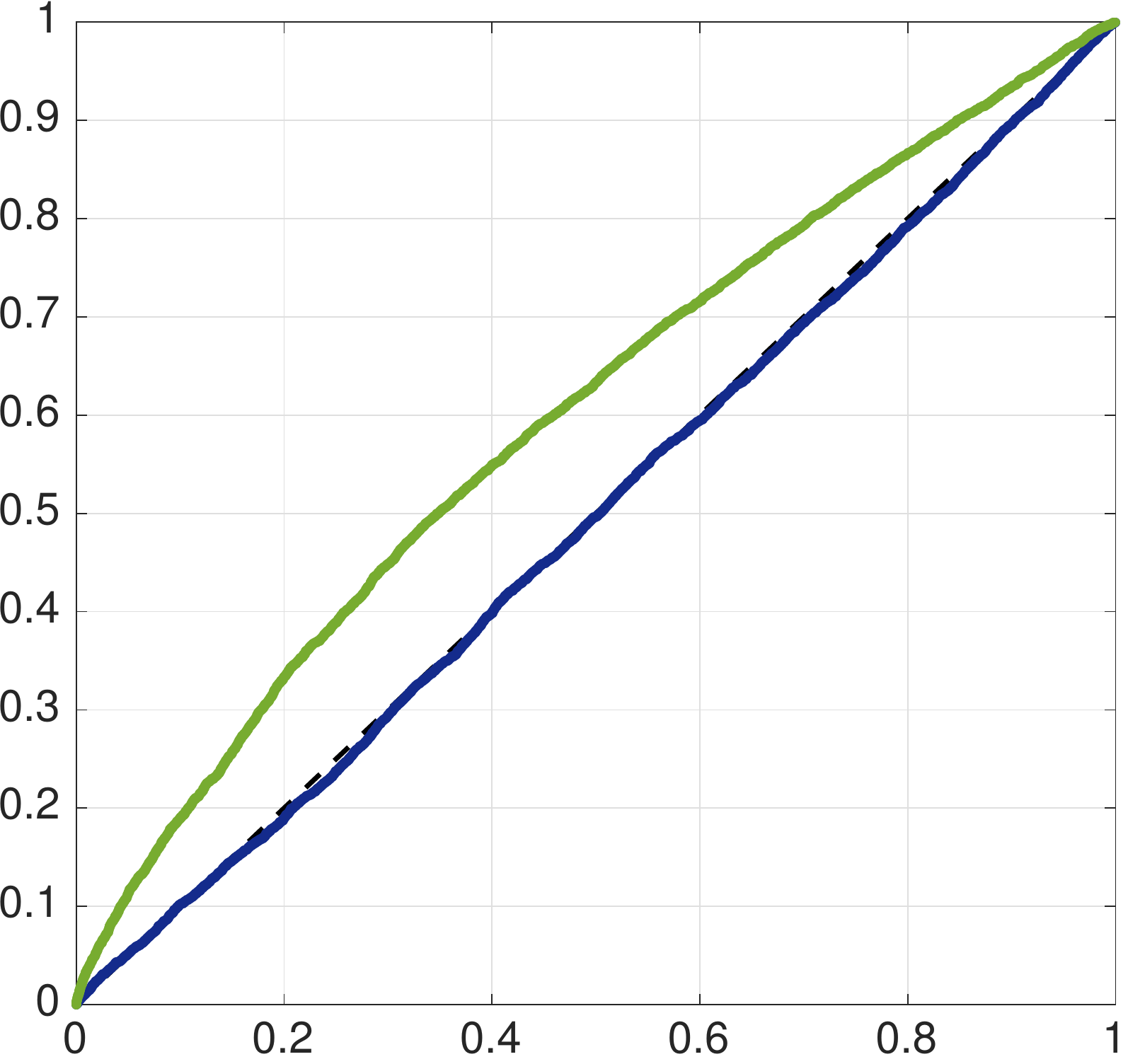}
\ \ \ \ \ 
\includegraphics[width=4.5cm,height=4.5cm]{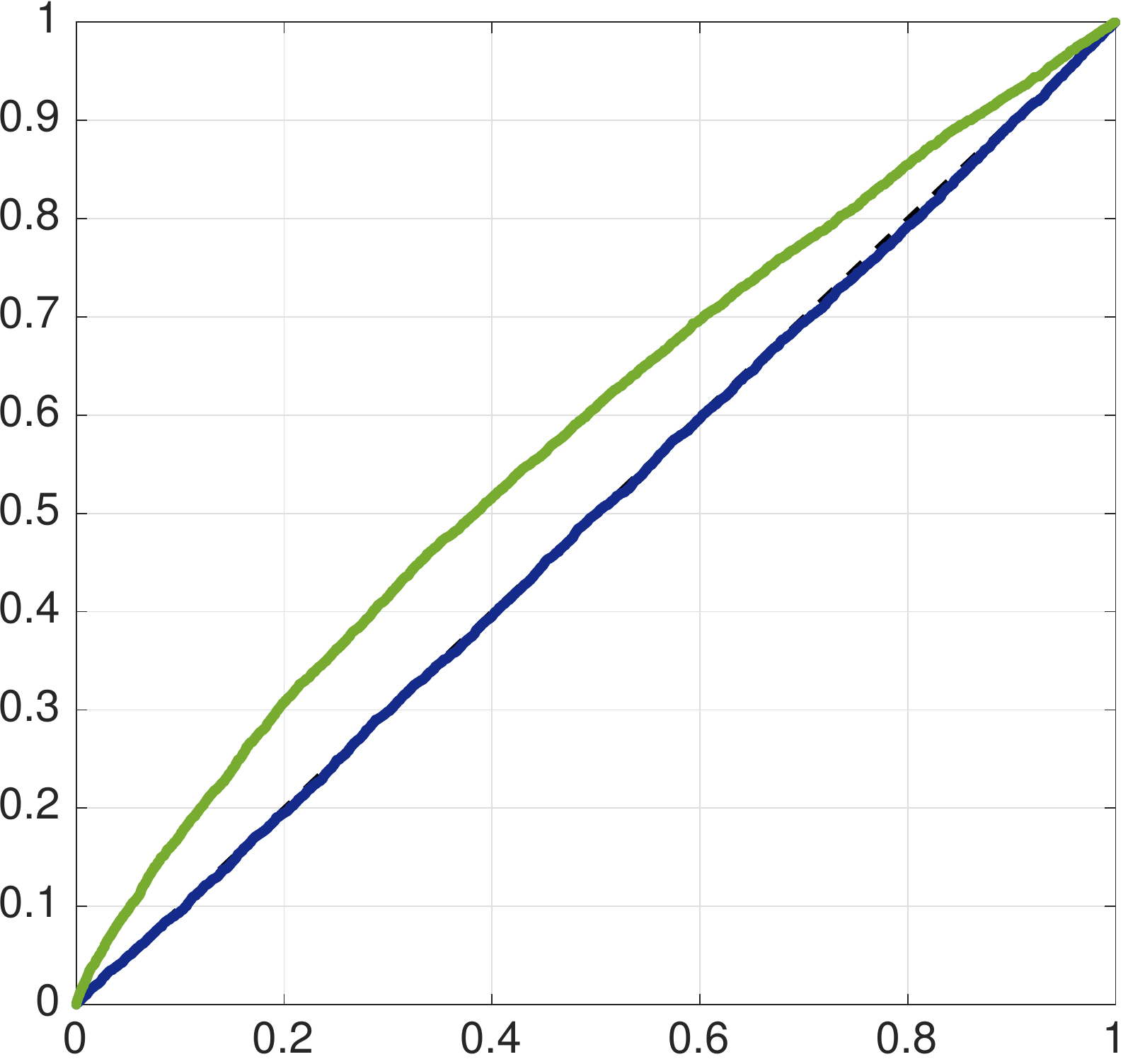}
\ \ \ \ \ 
\includegraphics[width=4.5cm,height=4.5cm]{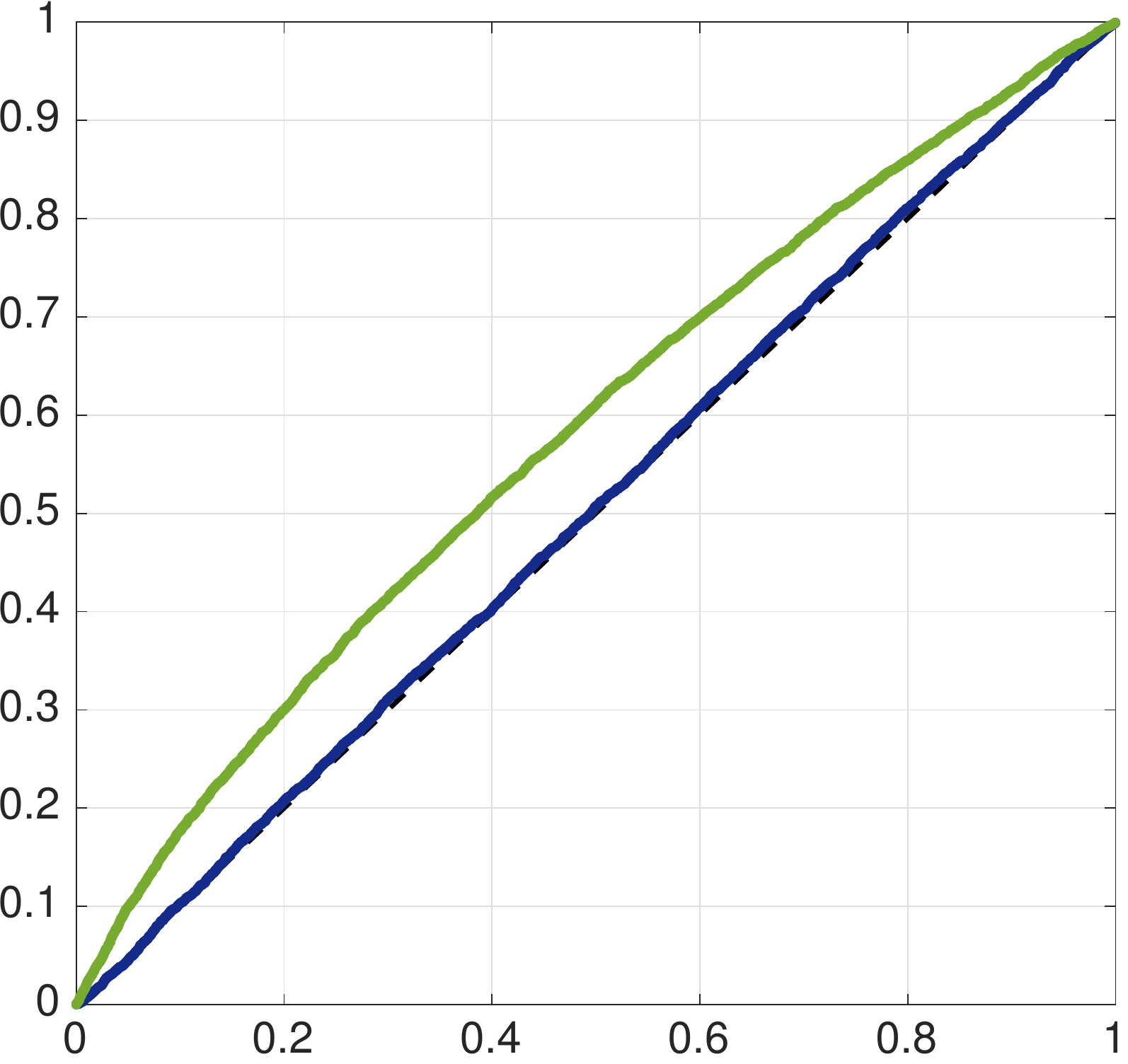}
\caption{{\bf [Under the null]} Comparison  of the  empirical cumulative distribution of the two statistics $S^{\rm Rice}$ $($blue line, see Theorem~\ref{thm:rice_known_variance}$)$  and $ S^{\rm ST}  $ $($green line$)$ {under the null hypothesis} when applied to the consecutive knots~$(\lambda_1,\lambda_2)$ given by the ‘‘continuous'' LARS in both cases. The diagonal $($cdf of the uniform$)$ is represented in dotted black line. The model is described by the Super-Resolution framework $($see Section~\ref{sec:SR}$)$ with cutoff frequencies $f_c = 3,5,7$ from left to right. The new test statistic $S^{\rm Rice}$ is exactly distributed w.r.t.\! the uniform law on $[0,1]$ under the null hypothesis.}
\label{unbiaised_naive}
\end{figure}

\noindent
To overcome this disappointing feature, one may be willing to consider thinner and thinner grids and look at the limit as $P$ tends to infinity. In this case, one can show that $\lambda_{1,P}$ tends to the~$\lambda_1$ of ‘‘continuous'' LARS, but $\lambda_{2,P}$ does not converge to $\lambda_2$, it converges to $\overline{\lambda}_2$ as shown in \eqref{eq:limit_l2}. This results in a limit test that is a randomized version of the Spacing test that we referred to as $S^{\mathrm{Grid}}$ and presented in Theorem~\ref{thm:grid_known_variance}. 

The second approach is to take a thin grid and to use $S^{\mathrm{ST}}$. This approach is perfectly valid, this test statistics follows a uniform distribution under the null and it should be compared to our new testing procedure $S^{\mathrm{Rice}}$. This numerical investigation has been performed in the frame of Super-Resolution and it is presented in Figure~\ref{Compa_1mean}, more details can be found in Section~\ref{sec:num}. Figure~\ref{Compa_1mean} gives the cumulative distribution functions of the test statistics under ‘‘sparse'' alternatives, {\it i.e.,} when true spikes are to be detected. The larger the power, the better the test detects spikes (abscissa represents the level of the test and the ordinate the probability to detect the spike). In these sets of experiments, we can note that 
\begin{itemize}
\item
{\it 
The testing procedure~$S^{\mathrm{Rice}}$ based on some Hessian and the whole process $X(\cdot)$ is uniformly better than the spacing test even if one takes very thin grids.}
\end{itemize}
 One can see that the power (the ability to detect sparse objects, Dirac masses here) of the grid methods seems to present a limit that is always improved by the continuous approach.

\begin{figure}[!h]
\includegraphics[width=4.5cm,height=4.5cm]{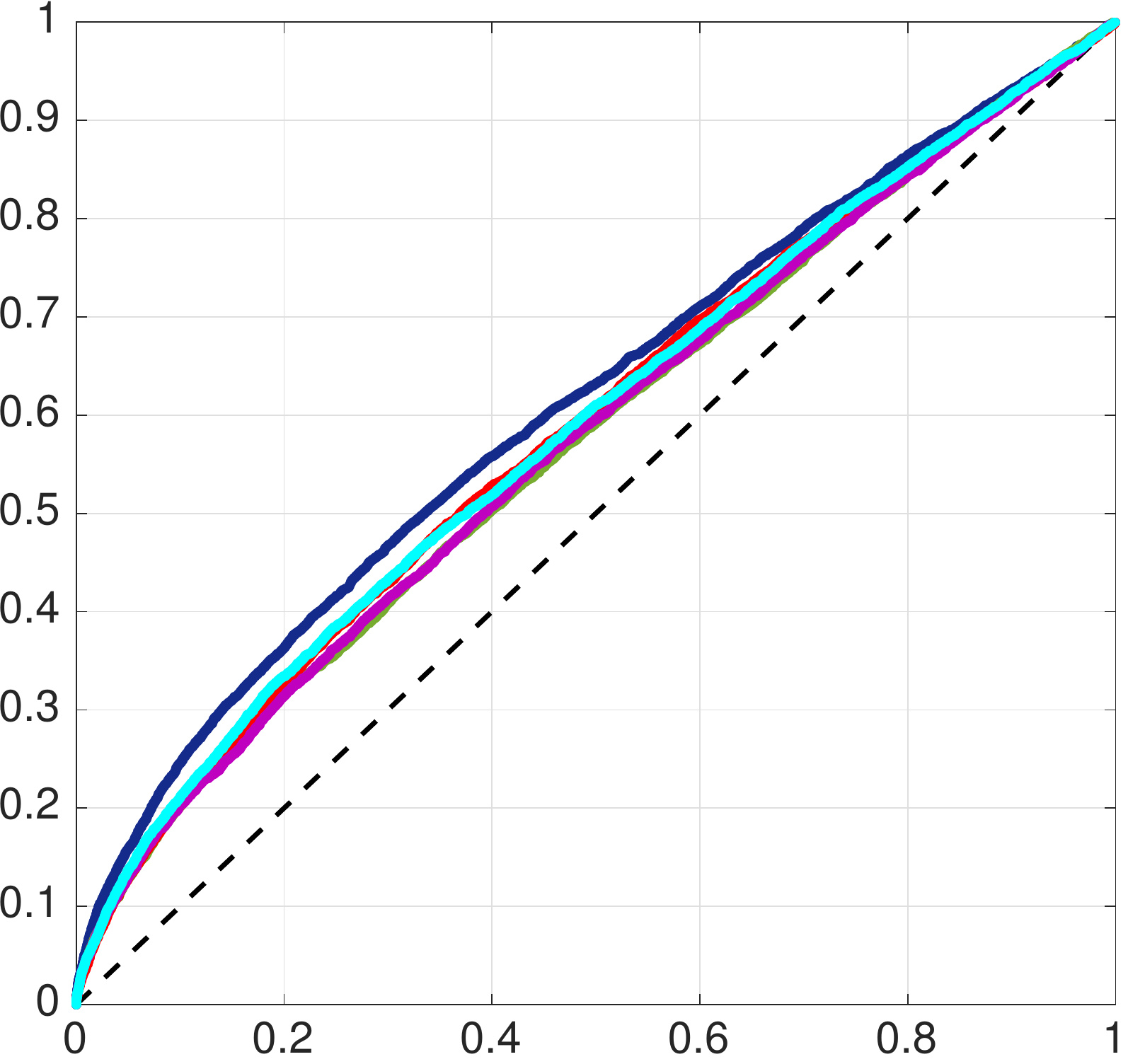}
\ \ \ \ \ 
\includegraphics[width=4.5cm,height=4.5cm]{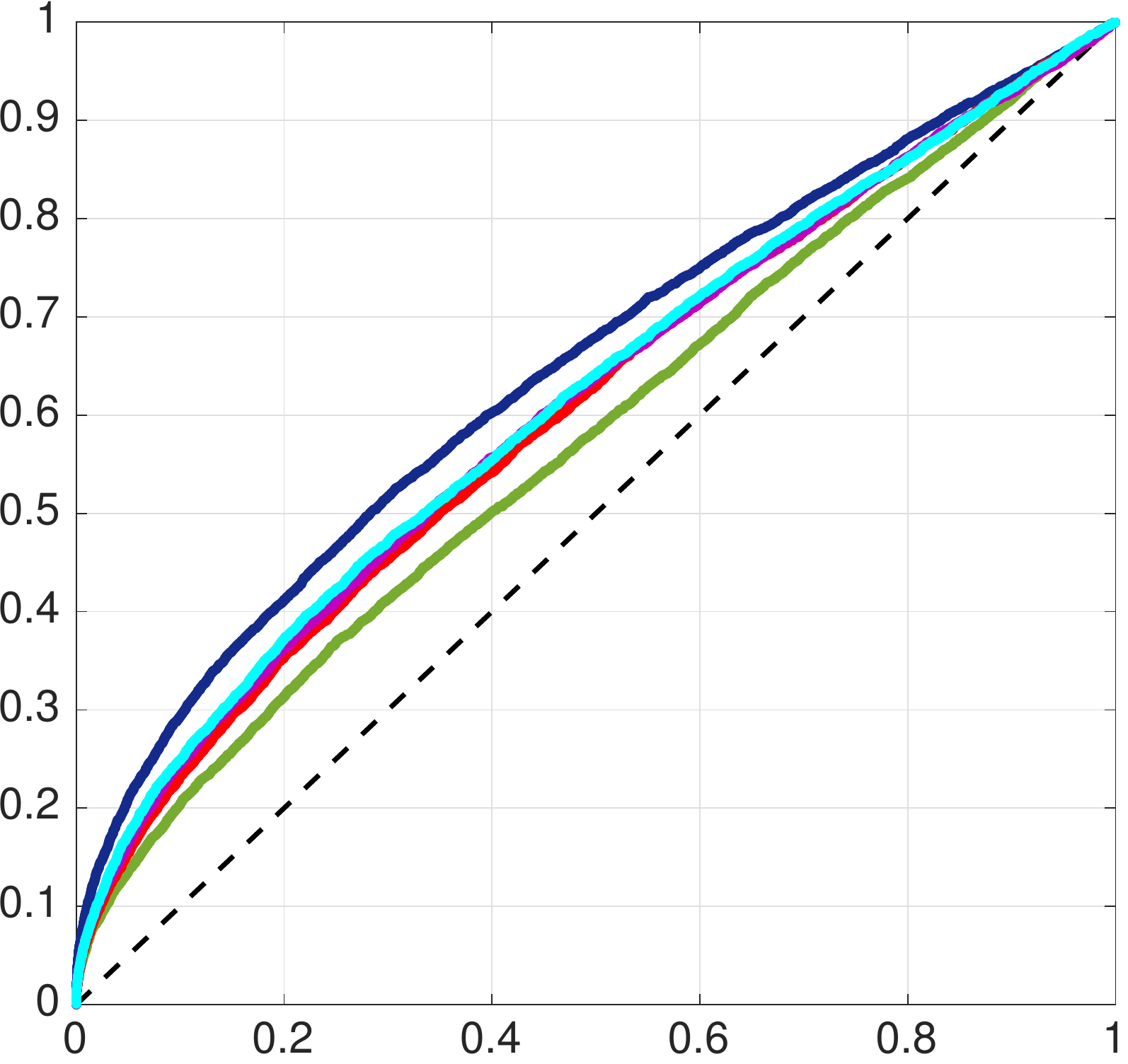}
\ \ \ \ \ 
\includegraphics[width=4.5cm,height=4.5cm]{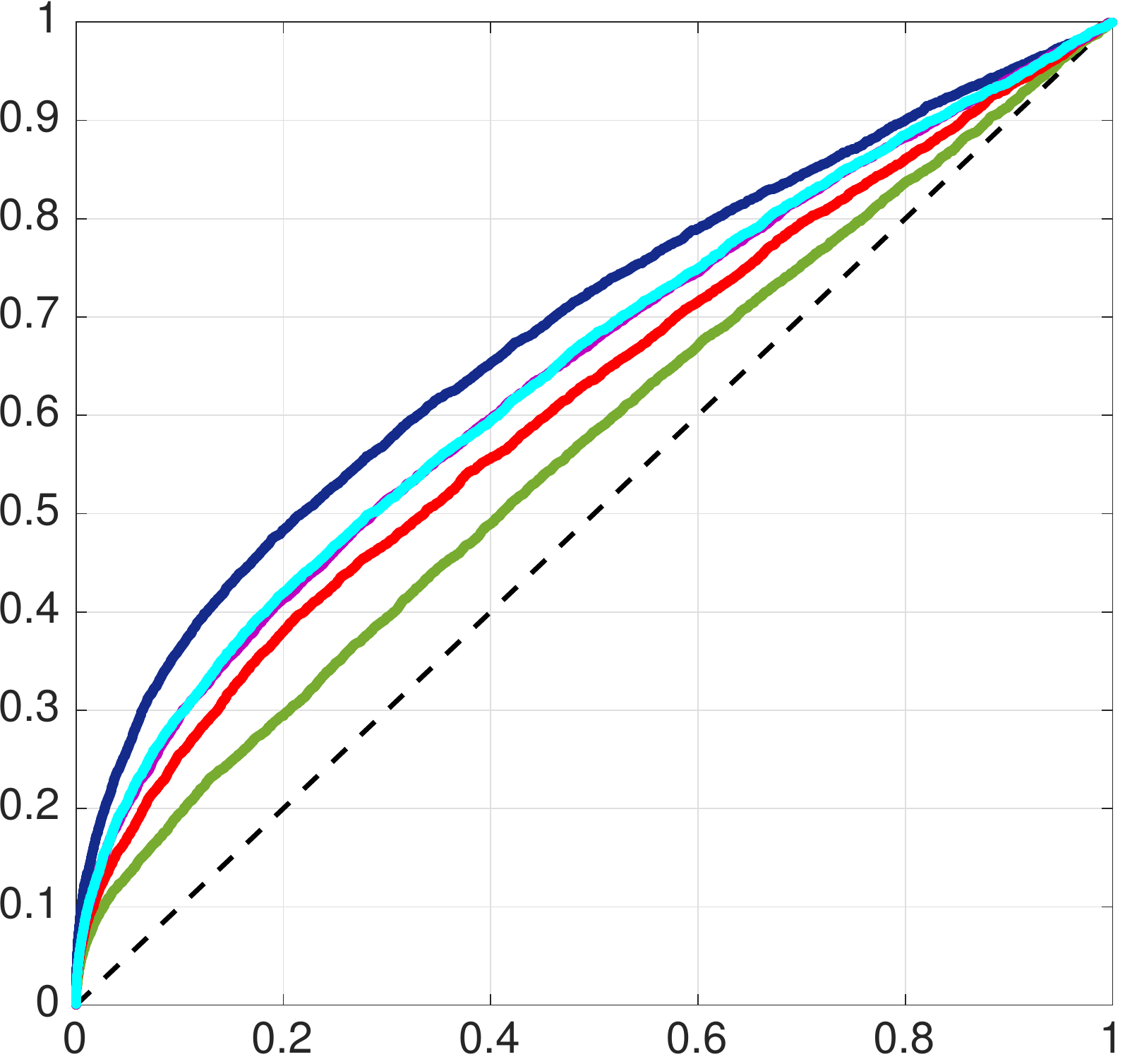} \\
\includegraphics[width=4.5cm,height=4.5cm]{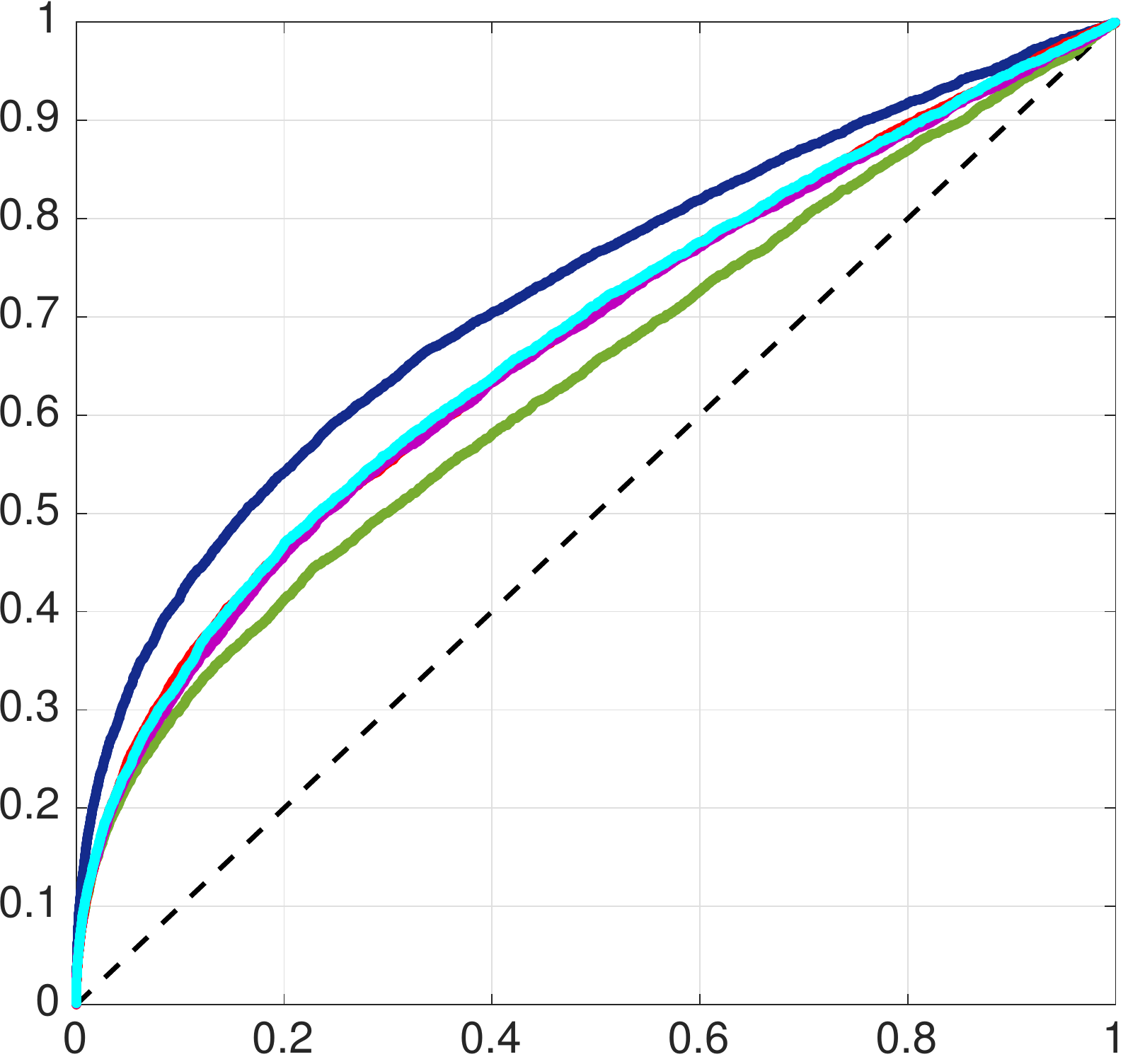}
\ \ \ \ \ 
\includegraphics[width=4.5cm,height=4.5cm]{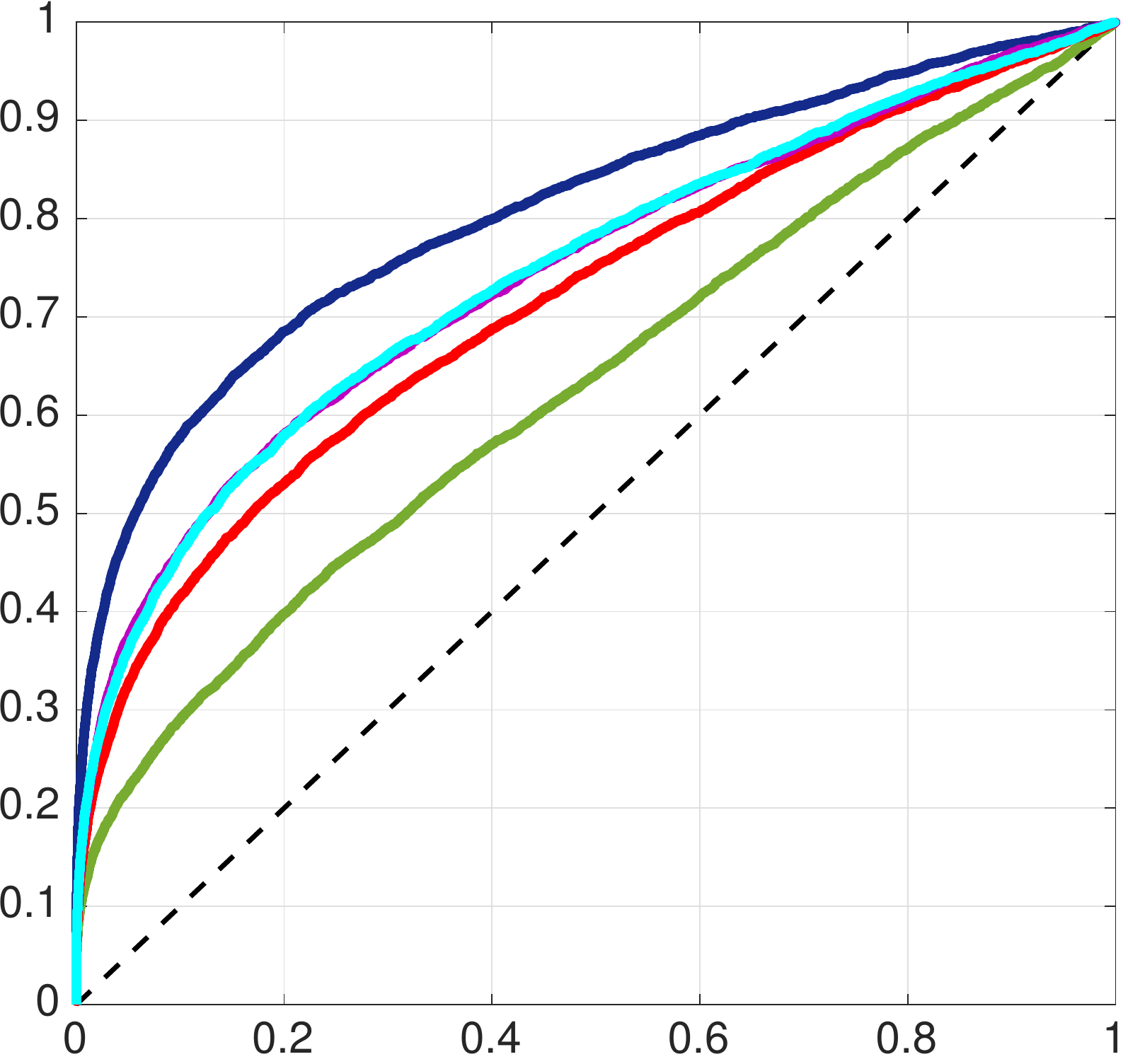}
\ \ \ \ \ 
\includegraphics[width=4.5cm,height=4.5cm]{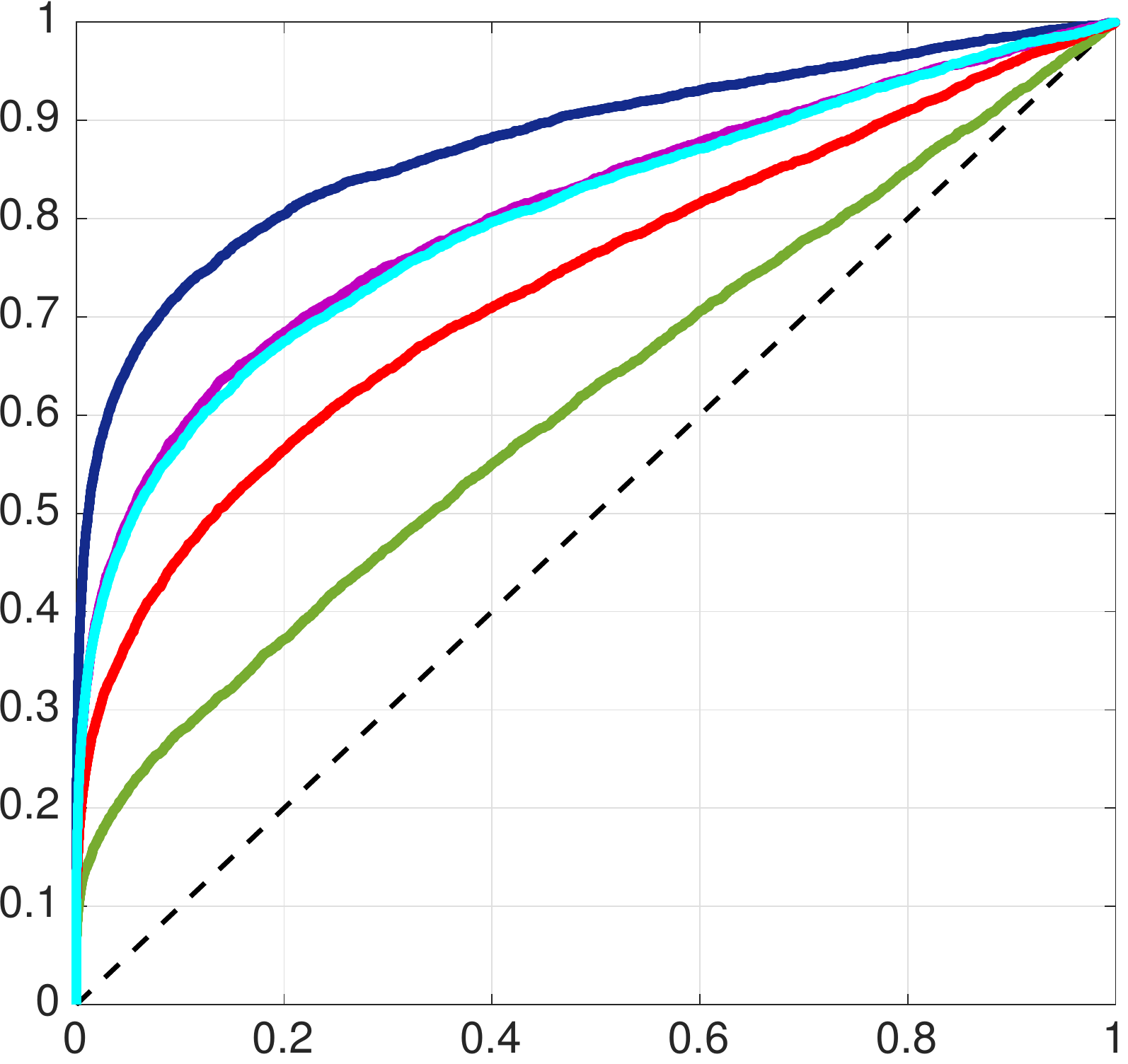}
\caption{{\bf [Under the alternative]}  Empirical   cumulative distribution {under the alternative} of the  Rice test (blue)  and the discrete grid tests  with size $3^2$  (green) , $10^2$(red),  $32^2$ (purple) and $50^2$ (cyan) . The alternative  is defined by a single atom  at a random  location  with a weight $ \log(N) = \log(2 f_c +1)$ (first row) or $\sqrt{N}$ (second row). In columns :  $f_c = 3,5,7$.}
\label{Compa_1mean}
\end{figure}

\subsection{Contribution}
For the first time, this paper paves the way to build new testing procedures in the framework of Super-Resolution theory and line spectral estimation. In particular, we prove that we can rightfully construct global null exact testing procedures on the first two {\it knots}~$\lambda_1$ and $\lambda_2$ of the ‘‘continuous'' LARS when one has a continuum of predictors, see Theorems \ref{thm:grid_known_variance} and~\ref{thm:rice_known_variance} and Figure~\ref{unbiaised_naive}. These two new procedures offer the ability to test the mean of any stationary Gaussian process with known correlation function~$\Gamma$ and $\mathcal C^2$-paths. Furthermore, one of these tests is unbiased, see Theorem \ref{thm:grid_known_variance} and they can be both studentized, see Theorems~\ref{t:t} and \ref{thm:rice_unknown_variance} and Figure~\ref{fig:student}, when variance~$\sigma^2$ is unknown. 

\subsection*{Outline}

\medskip

\hrule
\medskip

\contentsline {section}{\numberline {1}Introduction}{1}{section.1}
\contentsline {subsection}{\numberline {1.1}Grid-less spike detection through the \IeC {\textquoteleft }\IeC {\textquoteleft }continuous'' LARS}{1}{subsection.1.1}
\contentsline {subsection}{\numberline {1.2}A comparative study}{3}{subsection.1.2}
\contentsline {subsection}{\numberline {1.3}Contribution}{5}{subsection.1.3}
\newpage
\contentsline {section}{\numberline {2}The \IeC {\textquoteleft }\IeC {\textquoteleft }continuous'' LARS}{6}{section.2}
\contentsline {subsection}{\numberline {2.1}Cameron-Martin type Assumption on the mean}{6}{subsection.2.1}
\contentsline {subsection}{\numberline {2.2}Description of the \IeC {\textquoteleft }\IeC {\textquoteleft }continuous'' LARS}{8}{subsection.2.2}
\contentsline {subsection}{\numberline {2.3}Equivalent expression of the second knot}{12}{subsection.2.3}
\contentsline {subsection}{\numberline {2.4}Illustration: The two first knots of Super-Resolution}{12}{subsection.2.4}
\contentsline {section}{\numberline {3}Notations and problem formulation}{13}{section.3}
\contentsline {subsection}{\numberline {3.1}Hypothesis testing problem}{13}{subsection.3.1}
\contentsline {subsection}{\numberline {3.2}The first and second knots of a Gaussian process}{14}{subsection.3.2}
\contentsline {section}{\numberline {4}Passing to the limit, the grid approach}{14}{section.4}
\contentsline {section}{\numberline {5}The Rice method: a grid-less approach}{17}{section.5}
\contentsline {subsection}{\numberline {5.1}The known variance case}{17}{subsection.5.1}
\contentsline {subsection}{\numberline {5.2}The unknown variance case}{20}{subsection.5.2}
\contentsline {section}{\numberline {6}Applications to the Super-Resolution Theory }{25}{section.6}
\contentsline {subsection}{\numberline {6.1}Framework and results}{25}{subsection.6.1}
\contentsline {subsection}{\numberline {6.2}A numerical study}{27}{subsection.6.2}
\contentsline {section}{\numberline {A}Proofs}{30}{appendix.A}
\contentsline {subsection}{\numberline {A.1}Proof of Theorem~\ref {thm:grid_known_variance}}{31}{subsection.A.1}
\contentsline {subsection}{\numberline {A.2}Proof of Theorem \ref {t:t}}{34}{subsection.A.2}
\contentsline {subsection}{\numberline {A.3}Proof of Proposition~\ref {l:jm:l1}}{35}{subsection.A.3}
\contentsline {subsection}{\numberline {A.4}Proof of Proposition \ref {cor:rice_known_variance_blasso} and Proposition \ref {cor:rice_unknown_variance_blasso}}{36}{subsection.A.4}
\contentsline {section}{\numberline {B}Auxiliary results}{37}{appendix.B}
\contentsline {subsection}{\numberline {B.1}Regularity of $X^{|z}$ and new expression of $R(z)$}{37}{subsection.B.1}
\contentsline {subsection}{\numberline {B.2}Maximum of a continuous process}{37}{subsection.B.2}
\contentsline {section}{References}{38}{section*.7}

\medskip
\hrule
\bigskip
Notations and the formal problem formulation is described in Section~\ref{sec:prob_formulation}. In Section~\ref{s:asador}, we present the test statistic $S^{\mathrm{Grid}}$ which is constructed taking the limit of consecutive LARS knots~$(\lambda_{1,P},\lambda_{2,P})$ on thinner and thinner grids (namely the number of predictors $P$ tends to infinity). Section~\ref{s:riz} is the theoretical construction of our grid-less test based on consecutive knots~$(\lambda_{1},\lambda_{2})$ of the ‘‘continuous'' LARS. The main result concerning the test statistic~$S^{\mathrm{Rice}}$ is presented in this section. Applications to spike detection in Super-Resolution are developed in Section~\ref{sec:SR}. An appendix with the proofs can be found at the end of the paper.

The general construction of the ‘‘continuous'' LARS is given in Section~\ref{sec:LARS}. This section is independent from the rest of the paper. 

\section{The ‘‘continuous'' LARS}
\label{sec:LARS}

\subsection{Cameron-Martin type Assumption on the mean}
The algorithm presented here can be used for a large class of complex processes $Z$. We consider a complex-valued Gaussian process~$Z$ indexed on a compact metric space $\bbK$ with covariance function~$K$. 

\begin{remark}
Note that this model encompasses our to-be-announced-framework (see Section~\ref{sec:prob_formulation}) setting $\bbK=[0,2\pi)$ and $K=2\sigma^2\Gamma$ with~$\Gamma$ the correlation of $A_1$ defined in Section~\ref{sec:HypTestProb}. We do not assume that the process is stationary in this section.
\end{remark}

\noindent
We assume that its covariance $K$ is such that there exists $\sigma>0$ such that 
\eq
\label{eq:normK}
\forall s\neq t\in\bbK,\quad K(t,t)=2\sigma^2\ \mathrm{and}\ K(s,t)<2\sigma^2\,.
\qe
The scalar $2$ accounts for the contribution of the real and the imaginary part of $Z$ and $\sigma^2$ is the variance of the real part of $Z$. We assume that~$Z$ has continuous sample paths.

We present here the underlying hypothesis on the mean of the Gaussian processes under consideration when using the LARS algorithm. This hypothesis is of Cameron-Martin type. Indeed, the main drawback that should be avoided is when the mean cannot be represented in the RKHS of the Gaussian process $Z$. We recall that we can define a reproducing Hilbert space of the covariance~$K$, see \cite[Chapter 2.6]{gine2015mathematical} for instance. Denote $(\mathcal H,\langle\cdot,\cdot\rangle_{\mathcal H})$ this complex Hilbert space. Also, we can invoke a Karhunen-Loève expansion of the process $Z$. Namely, there exist i.i.d. complex standard normal variables~$(g_j)_{j\geq 1}$, a real orthonormal system $(e_j)_{j\geq 1}$ on $L^2(\bbK)$ and $\sigma_j>0$ such that
\[
Z-\bbE Z=\sum_j \sigma_j g_j e_j\quad\mathrm{and}\quad\sum_j\sigma_j^2=2\sigma^2<\infty,
\]
where the identity holds almost surely in the Banach space of continuous functions on $\bbK$ equipped with the $L^\infty$-norm. By Mercer's theorem, we know that
\[
\forall s,t\in\bbK,\quad K(s,t)=\sum_j\sigma_j^2\,e_j(s)e_j(t)\,,
\]
where the identity holds almost surely in the Banach space of continuous functions on $\bbK\times\bbK$ equipped with the $L^\infty$-norm. We recall also that the Hilbert space $\mathcal H$ can be defined as
\[
\mathcal H:=\Big\{\sum_j a_je_j\ |\ \sum_j\frac{|a_j|^2}{\sigma_j^2}<\infty\Big\}
\]
with the inner product
\[
 \Big\langle\sum_j a_je_j,\sum_j b_je_j\Big\rangle_{\mathcal H}=\sum_{j}\frac{a_j\bar b_j }{\sigma_j^2}\,.
\]
We observe $Z$ and we want to estimate its mean $\E Z$. Remark that almost surely it holds $Z-\E Z\in\overline{\mathcal H}$, where $\overline{\mathcal H}$ is the closure of ${\mathcal H}$ in the space of continuous functions equipped with the infinity norm, see {\it e.g.}  \cite[Corollary 2.6.11]{gine2015mathematical}. Remark that $\overline{\mathcal H}$ is also closed in $L^2(\bbK)$. Denoting by $E$ the $L^2$ orthogonal space of $\overline{\mathcal H}$, one has $L^2(\bbK)=\overline{\mathcal H}\oplus E$ where the sum is orthogonal. We denote by $\mathcal P$ (resp. $\mathcal P^\perp$) the orthogonal projection onto $\overline{\mathcal H}$ (resp.~$E$). Since almost surely $Z-\E Z\in\overline{\mathcal H}$, remark that almost surely $\mathcal P^\perp(Z)=\mathcal P^\perp(\bbE Z)$ and this process can be observed and is deterministic. Without loss of generality, we assume that $\mathcal P^\perp(\bbE Z)=0$ subtracting~$\mathcal P^\perp(Z)$ to $Z$. Also, we assume that 
\eq
\label{eq:hypLARS}
\mathcal P(\bbE Z)\in{\mathcal H}\,.
\qe
Recall that $\mathcal P(\bbE Z)=\bbE Z$ and Assumption~\eqref{eq:hypLARS} gives that $Z\in\overline{\mathcal H}$ using $Z-\E Z\in\overline{\mathcal H}$.

\subsection{Description of the ‘‘continuous'' LARS}
We assume that $Z\in\mathcal H$ and, as mentioned above, this assumption is equivalent to Assumption~\eqref{eq:hypLARS}. Following standard references, {\it e.g.,} \cite[Chapter 5.6]{hastie2015statistical}, the Least-Angle Regression Selection (LARS) algorithm can be extended to Gaussian processes. To the best of our knowledge, the LARS for complex Gaussian processes has never been introduced and we present its formulation here for the first time. Actually, the presentation given in this section can be applied to any RKHS setting. It results in a description of the LARS in infinite dimensional feature spaces and this framework has been dealt in \cite{rosset2007l1}. However, note that the paper  \cite{rosset2007l1} only concerns real signed measures and their ‘‘doubling'' dimension trick \cite[page 546]{rosset2007l1} cannot be used when dealing with complex measures. In particular, their result cannot be invoked in Super-Resolution where it is of utmost importance to deal with complex measures. This section presents the ‘‘continuous'' LARS for Super-Resolution.

The LARS is a variable selection algorithm giving a sequence $((\lambda_k,\mu_k))_{k\geq1}$ where the knots are ordered such that $\lambda_1\geq\lambda_2\geq\ldots>0$ and $\mu_k\in(\mathcal M(\bbK,\mathds C),\|\cdot\|_{1})$ is a complex-valued measure. {We recall that the space $(\mathcal M(\bbK,\mathds C),\|\cdot\|_{1})$ is defined as the dual space of the space of continuous functions on~$\bbK$ equipped with the $L^\infty$-norm. A pseudo-code is presented in Algorithm~\ref{alg:ContinuousLARS} and the technical details are presented below. When defining the ‘‘continuous LARS'', we assume that 
\eq
\label{ass:LARS}
\tag{$\bA_K$}
K \text{ is (at least) four times differentiable.}
\qe
Under this assumption, the process $Z$ is twice differentiable in quadratic mean and once differentiable almost surely.
{\tiny
\begin{algorithm}[t]
  \SetAlgoLined 
  \KwData{A correlation process $Z$ indexed by $\bbK$ and its variance-covariance function~$K$.} 
  \KwResult{A sequence $((\lambda_k,\mu_k))_{k\geq1}$ where the knots are ordered such that $\lambda_1\geq\lambda_2\geq\ldots>0$ and $\mu_k\in(\mathcal M(\bbK,\mathds C),\|\cdot\|_{1})$ is a complex-valued measure.}
    \BlankLine
    {\nonl \tcc{We initialize this Forward procedure computing $\lambda_1$ and $\mu_1$}}
    Set $k=1$,  $\displaystyle \lambda_1:=\max_{\bbK}|Z|$ and $\mu_1=0$.\\
       {\nonl \tcc{We use an ‘‘active set'' $\cA_k$ giving the support of the next solution $\mu_{k+1}$}}
      Set $\displaystyle t_1:=\arg\max_{\bbK}|Z|$ and $\cA_1=(t_1)$.\\
        {\nonl   \tcc{We use a ‘‘residual'' $Z_k$ initialized with}}
      Set $Z_1:=Z$.
          \BlankLine
      {\nonl \tcc{We iterate the next commands until a stopping criterion is met}}
      Set $k\leftarrow k+1$ {\tcc{$\cA_{k-1}=:(t_1,\ldots,t_{k-1})$ and $\lambda_{k-1}$ has been defined at the previous step.}}
      For $\lambda>0$ and $x=(x_1,\ldots,x_{k-1})\in\bbR^{k-1}$ define 
\begin{align*}
a(\lambda,x)&:=(K(x_i,x_j))^{-1}
\left(\begin{array}{c}Z(x_1)-(\lambda/ \lambda_{k-1})Z_{k-1}(t_1)
 \\\vdots \\
 Z(x_{k-1})-(\lambda/ \lambda_{k-1})Z_{k-1}(t_{k-1})\end{array}\right)\\
 h_j(\lambda,x)&:=\frac{\partial}{\partial t}\Big[\big|Z(t)-\sum_{i=1}^{k-1}a_i(\lambda,x)K(x_i,t)\big|^2\Big](x_j)
\end{align*}    
and solve $(h_1(\lambda,x),\ldots,h_{k-1}(\lambda,x))=0$ starting from $(\lambda,x)=(\lambda_{k-1},\cA_{k-1})$ for $0<\lambda\leq\lambda_{k-1}$. The solution path is denoted by $x(\lambda):=(t_1(\lambda),\ldots,t_{k-1}(\lambda))$.

          Set  $\displaystyle Z^{(\lambda)}(\cdot):=Z(\cdot)-\sum_{i=1}^{k-1}a_i(\lambda,x(\lambda))K(t_i(\lambda),\cdot)$ and  pick
\[
\lambda_{k}:=\max\big\{\beta>0\, ;\  \exists\, t\notin x(\beta),\ \mathrm{s.t.}\ |Z^{(\beta)}|(t)=\beta\big\}
\quad\mathrm{and}\quad t_k:=\argmax_{s\notin x(\lambda_k)} |Z^{(\lambda_k)}|(s)\,.
\]
   \\
Set $\cA_k=(t_1(\lambda_k),\ldots,t_{k-1}(\lambda_k),t_k)$ and 
\[
\mu_k=\sum_{i=1}^{k-1}a_i(\lambda_k,x(\lambda_k))\delta_{t_i(\lambda_k)}\text{ and }
Z_k(\cdot)=Z^{(\lambda_k)}(\cdot)=Z(\cdot)-\sum_{i=1}^{k-1}a_i(\lambda_k,x(\lambda_k))K(t_i(\lambda_k),\cdot)\,.
\]
\\
          Iterate from {\bf 4}.
          \BlankLine
\caption{Continous LARS}
\label{alg:ContinuousLARS}
\end{algorithm}
}

\subsubsection{The first knot}
Inspired by the Super-Resolution framework\textemdash presented in Section~\ref{sec:SR}, we consider $Z$ as some ‘‘correlation process'' in the spirit of \eqref{eq:CorrZSR}. In particular, the most correlated point can be defined by~\eqref{def:hatz}, namely
\[
\lambda_1:=\max_{t\in\bbK}|Z(t)|\,.
\]
Under Assumption~\eqref{eq:normK}, Proposition~\ref{prop:Unicity} shows that almost surely there exists a unique point~$t_1$ such that $\lambda_1=|Z(t_1)|$. Define the ‘‘active set'' function $\lambda\mapsto\mathcal A(\lambda)$ as 
\[
\mathcal A(\lambda_1)=\cA_1:=(t_1)\,,
\] 
and $\mathcal A(\lambda)=\emptyset$ for $\lambda>\lambda_1$. 
The path $\lambda\mapsto\cA(\lambda)$ for $\lambda\leq\lambda_1$ will be defined in the sequel. It is a piecewise continuously differentiable path representing the support of a discrete measure $\mu^{(\lambda)}$ such that
\[
||Z-\int_\bbK K(u,\cdot)\mathrm d\mu^{(\lambda)}(u)||_\infty\leq\lambda\,,
\]
namely the residual has $L^\infty$-norm less than $\lambda$. Set the first fitted solution to $\mu_1=0$ and the first residual to~$Z_1=Z$ for initialization purposes. Observe that 
\begin{align}
Z_1(t_1)&=\lambda_1\,e^{\imath \theta_1},\label{eq:t1}\\
|Z_1(t_1)|&=\lambda_1\,,\notag\\
\forall t\in\bbK,\quad Z_1(t)&=Z(t)-\int_\bbK K(u,t)\mu_1(\mathrm du)\,,\notag\\
\forall t\neq t_1,\quad |Z_1(t)|&<\lambda_1\,.\notag
\end{align}

\subsubsection{The second knot}

We want to add an other point $t_2$ to the active set and define a discrete measure $\mu_2$ supported on~$\mathcal A_1$ while keeping the above inequalities true. First, we solve the least-squares fit given by
\eq
\notag
a=\argmin_{c\in\bbC}\big\|Z_{1}(\cdot)-cK(t_1,\cdot)\big\|_{\mathcal H}^2\,.
\qe
This program can be solved in closed form and it holds that $a=Z(t_1)/(2\sigma^2)$. Then, for any $0<\lambda\leq\lambda_{1}$, define $Z^{(\lambda)}$ by
\begin{align*}
Z^{(\lambda)}(t)&=Z(t)+(\frac{\lambda}{\lambda_1}-1)Z(t_1)\frac{K(t_1,t)}{2\sigma^2}\,,
\end{align*}
and observe that $|Z^{(\lambda)}(t_1)|=\lambda$. Remark also that $|Z^{(\lambda)}|$ has a local maxima at point  $t=t_1$. Indeed, under \eqref{ass:LARS}, the function $X:(t,\theta)\mapsto\mathrm{Re}(e^{-\imath \theta}Z(t))$ is continuously differentiable and it has $t=t_1$ as global maximum by definition of $t_1$ and $\lambda_1$. Il follows from \eqref{eq:X_decompo} that $\hat z:=(t_1,\theta_1)$ is a local maxima of $X$ and therefore $t_1$ is a local maxima of $|Z^{(\lambda)}|$.

Now, we keep track of the largest value of the ‘‘correlation'' process $|Z^{(\lambda)}|$ on the complementary set of $\mathcal A_{1}$ while moving $\lambda$ from $\lambda_{1}$ toward zero. We define~$\lambda_{2}$ as the largest value for which there exists a point $t\notin\mathcal A_{1}$ such that $|Z^{(\lambda)}|(t)=\lambda$. Set 
\begin{align}
\lambda_{2}&:=\max\big\{\beta>0\, ;\  \exists\, t\notin\mathcal A_{1},\ \mathrm{s.t.}\ |Z^{(\beta)}|(t)=\beta\big\}\,,\notag\\
\mathrm{and}\quad t_2&:=\argmax_{s\notin\mathcal A_{1}} |Z^{(\lambda_2)}|(s)\,.\label{eq:t2}
\end{align}
If $t_2$ is not unique, we add all the solutions of \eqref{eq:t2} to the active set $\mathcal A_2$. For sake of readability, we assume that $t_2$ is the only solution to \eqref{eq:t2}. Then, update 
\begin{align*}
\cA(\lambda)&=\cA_1\quad \lambda_2<\lambda\leq\lambda_1\,,\\
\cA(\lambda_2)&=\mathcal A_2:=(t_1,t_2)\,,\\
\mu_2&=(1-\lambda_2/\lambda_{1})a\delta_{t_1}\,,\\
Z_2(\cdot)&=Z^{(\lambda_2)}(\cdot)=Z_1(\cdot)+(\lambda_2/\lambda_{1}-1)aK(t_1,\cdot)\,,
\end{align*}
where, for all $t\in\bbK$,
\begin{align*}
 Z_2(t)&=Z_{1}(t)+(\lambda_1/\lambda_{2}-1)aK(t_1,t)\,\\
  &=Z(t)-\int_\bbK K(u,t)\mu_{2}(\mathrm du)\,,
\end{align*}
is the second residual associated to the second fitted solution $\mu_2$. Remark also that
\begin{align*}
\forall t\in\{t_1,t_2\},\quad|Z_2(t)|&=\lambda_2\,,\\
\forall t\neq\{t_1,t_2\},\quad |Z_2(t)|&<\lambda_2\,.
\end{align*}

\subsubsection{The other Knots: Moving the Active Set between Knots}
From this point we proceed iteratively. For $k\geq3$, we assume that we have found $(\lambda_{k-1},\mu_{k-1})$ and $\mathcal A_{k-1}=(t_1,\ldots,t_{k-1})$ such that 
\begin{align}
\cA(\lambda_{k-1})&=\cA_{k-1}\,,\notag\\
\forall t\in\bbK,\quad Z_{k-1}(t)&:=Z(t)-\int_\bbK K(u,t)\mu_{k-1}(\mathrm du)\,,\notag\\
\forall t\in\mathcal  A_{k-1},\quad |Z_{k-1}(t)|&=\lambda_{k-1}\,,\notag\\
\forall t\notin\mathcal  A_{k-1},\quad |Z_{k-1}(t)|&<\lambda_{k-1}\,.\notag
\end{align}
We want to define the path $\lambda\mapsto\cA(\lambda)$ for values $\lambda\leq\lambda_{k-1}$ starting from $\cA(\lambda_{k-1})=\cA_{k-1}$. We look for a path $\cA(\lambda)=(t_1(\lambda),\ldots,t_{k-1}(\lambda))$ such that $t_i(\lambda)$ are continuously differentiable and there exists $\mu^{(\lambda)}$ supported on  $\cA(\lambda)$ such that the above inequalities hold true. This path will be defined on $(\lambda_k,\lambda_{k-1}]$ for a value $\lambda_k$ defined later.

\medskip

\noindent
$\circ\quad$ Consider $0<\lambda\leq\lambda_{k-1}$ and define
\[
a(\lambda)=M_{k-1}(\lambda)^{-1}
\left(\begin{array}{c}Z(t_1(\lambda))-(\lambda/ \lambda_{k-1})Z_{k-1}(t_1(\lambda_{k-1}))
 \\\vdots \\Z(t_{k-1}(\lambda))-(\lambda/ \lambda_{k-1})Z_{k-1}(t_{k-1}(\lambda_{k-1}))\end{array}\right)
 \]
where we denote $M_{k-1}(\lambda)=(K(t_i(\lambda),t_j(\lambda)))_{1\leq i,j\leq k-1}$ and we assume that $M_{k-1}(\lambda)$ is invertible. If $M_{k-1}(\lambda)$ is not invertible then we stop. The path $\cA(\lambda)=(t_1(\lambda),\ldots,t_{k-1}(\lambda))$ will be defined later on. Note that $\cA(\lambda_{k-1})=(t_1(\lambda_{k-1}),\ldots,t_{k-1}(\lambda_{k-1}))=\cA_{k-1}$ for $\lambda=\lambda_{k-1}$.

\begin{remark}
Note that the function $\sum_{i=1}^{k-1}a_i(\lambda_{k-1})K(t_i(\lambda_{k-1}),\cdot)$ is the regression of $Z$ onto the finite dimensional space $\mathrm{Span}\{K(t_i(\lambda_{k-1}),\cdot)\, ;\ i=1,\ldots,k-1\}$.
\end{remark}

\noindent
$\circ\quad$ Then, for any $0<\lambda\leq\lambda_{k-1}$, define 
\begin{align*}
\mu^{(\lambda)}&:=\sum_{i=1}^{k-1}a_i(\lambda)\delta_{t_i(\lambda)}\,,\\
Z^{(\lambda)}(\cdot)&:=Z(\cdot)-\sum_{i=1}^{k-1}a_i(\lambda)K(t_i(\lambda),\cdot)\,.
\end{align*}
and observe that $|Z^{(\lambda)}(t)|=\lambda$ for all $t\in\{t_1(\lambda),\ldots,t_{k-1}(\lambda)\}$. Indeed, it holds
\begin{align*}
Z^{(\lambda)}(t_j(\lambda))&=Z(t_j(\lambda))-\sum_{i=1}^{k-1}a_i(\lambda)K(t_i(\lambda),t_j(\lambda))\,,\\
&=Z(t_j(\lambda))-a(\lambda)^\top (M_{k-1}(\lambda))(0,\ldots,0,\underbrace{1}_{j\mathrm{th}},0,\ldots,0)\,,\\
&=Z(t_j(\lambda))-Z(t_j(\lambda))+(\lambda/ \lambda_{k-1})Z_{k-1}(t_j(\lambda_{k-1}))\,,\\
&=\lambda Z_{k-1}(t_j(\lambda_{k-1}))/\lambda_{k-1}\,,
\end{align*}
and recall that it holds $|Z_{k-1}(t_j(\lambda_{k-1}))|=\lambda_{k-1}$.

We will enforce that $t_j(\lambda)$ is a local maximum of $|Z^{(\lambda)}|$ imposing that its derivative is zero along the path $\cA(\lambda)$ for $\lambda_k<\lambda\leq\lambda_{k-1}$. This can be done invoking the implicit function theorem as follows. Define for $\lambda>0$ and $x=(x_1,\ldots,x_{k-1})\in\bbR^{k-1}$
\[
F(\lambda,x):=(h_1(\lambda,x),\ldots,h_{k-1}(\lambda,x))
\]
where
\begin{align}
a(\lambda,x)&:=(K(x_i,x_j))^{-1}
\left(\begin{array}{c}Z(x_1)-(\lambda/ \lambda_{k-1})Z_{k-1}(t_1(\lambda_{k-1}))
 \\\vdots \\
 Z(x_{k-1})-(\lambda/ \lambda_{k-1})Z_{k-1}(t_{k-1}(\lambda_{k-1}))\end{array}\right)\notag\\
 h_j(\lambda,x)&:=\frac{\partial}{\partial t}\Big[\big|Z(t)-\sum_{i=1}^{k-1}a_i(\lambda,x)K(x_i,t)\big|^2\Big](x_j)
 \label{eq:zero_derivative}
\end{align}
Assume that the jacobian $\frac{\partial F}{\partial x}$ is invertible. If $\frac{\partial F}{\partial x}$ is not invertible then we stop. Therefore, the implicit function theorem implies that there exists a continuously differentiable path $x(\lambda):=(t_1(\lambda),\ldots,t_{k_1}(\lambda))$ such that $F(\lambda,x)=0$ is equivalent to $x=(t_1(\lambda),\ldots,t_{k_1}(\lambda))$ on a neighborhood of $\lambda=\lambda_{k-1}$. On this path, the derivative of $t\mapsto|Z^{(\lambda)}|^2(t)$ at points $t=t_j(\lambda)$ is zero (thanks to \eqref{eq:zero_derivative}) while $|Z^{(\lambda)}|(t_j(\lambda))=\lambda$. We deduce that there exists a neighborhood of~$\lambda_{k-1}$ on which each point $t_j(\lambda)$ is a local maximum of $|Z^{(\lambda)}|$. 

Now, we keep track of the largest value of the ‘‘correlation'' process $|Z^{(\lambda)}|$ on the complementary set of $\mathcal A(\lambda)$ while moving $\lambda$ from $\lambda_{k-1}$ toward zero. We define~$\lambda_{k}$ as the largest value for which there exists a point $t\notin\mathcal A(\lambda)$ such that $|Z^{(\lambda)}|(t)=\lambda$. Set 
\begin{align}
\lambda_{k}&:=\max\big\{\beta>0\, ;\  \exists\, t\notin\mathcal A(\beta),\ \mathrm{s.t.}\ |Z^{(\beta)}|(t)=\beta\big\}\,,\notag\\
\mathrm{and}\quad t_k&:=\argmax_{s\notin\{t_1(\lambda_k),\ldots,t_{k-1}(\lambda_k)\}} |Z^{(\lambda_k)}|(s)\,.\label{eq:tk}
\end{align}
If $t_k$ is not unique, we add all the solutions of \eqref{eq:tk} to the active set $\mathcal A_k$. For sake of readability, we assume that $t_k$ is the only solution to \eqref{eq:tk}.

\medskip

\noindent
$\circ\quad$ Update 
\begin{align*}
\cA(\lambda)&=(t_1(\lambda_k),\ldots,t_{k-1}(\lambda_k))\quad \lambda_k<\lambda\leq\lambda_{k-1}\,,\\
\cA(\lambda_k)&=\mathcal A_k:=(t_1(\lambda_k),\ldots,t_{k-1}(\lambda_k),t_k)\,,\\
\mu_k&=\mu^{(\lambda_k)}=\sum_{i=1}^{k-1}a_i(\lambda_k)\delta_{t_i(\lambda_k)}\,,\\
Z_k(\cdot)&=Z^{(\lambda_k)}(\cdot)=Z(\cdot)-\sum_{i=1}^{k-1}a_i(\lambda_k)K(t_i(\lambda_k),\cdot)\,,
\end{align*}
where, for all $t\in\bbK$,
\begin{align*}
 Z_k(t)
  &=Z(t)-\int_\bbK K(u,t)\mathrm d\mu_{k}(u)\,,
\end{align*}
is the $k$th residual associated to the $k$th fitted solution $\mu_k$. Remark also that
\begin{align*}
\forall t\in\{t_1,\ldots,t_{k}\},\quad|Z_k(t)|&=\lambda_k\,,\\
\forall t\neq\{t_1,\ldots,t_{k}\},\quad |Z_k(t)|&<\lambda_k\,,
\end{align*}
and update $k$ to $k+1$ to iterate the procedure.

\subsection{Equivalent expression of the second knot}
\label{sec:second_knot}
First, observe that $\lambda_1$ is defined as in \eqref{def:hatz} and that the two definitions agree. Indeed, recall that $X(t,\theta)=\mathrm{Re}\,(e^{-\imath\theta}Z(t))$ so that $\max X=\max|Z|$ at point $\hat z=(t_1,\theta_1)$ with $t_1$ as in \eqref{eq:t1}. By optimality, it holds that  $\lambda_1=e^{-\imath\theta_1}Z(t_1)$.

Then, the case $k=2$ is interesting since $\lambda_2$ is a statistic used in the test statistics described in the sequel. We will see that the two definitions agree here again, please refer to Section~\ref{sec:prob_formulation} for notations. For $k=2$, it holds $Z_1=Z$ and the least squares direction $a$ is given by $a=Z(t_1)/(2\sigma^2)$ and $Z^{(\lambda)}$ by
\begin{align*}
Z^{(\lambda)}(t)&=Z(t)+(\frac{\lambda}{\lambda_1}-1)Z(t_1)\frac{K(t_1,t)}{2\sigma^2}\,,\\
&=Z(t)+e^{\imath\theta_1}({\lambda}-\lambda_1)\frac{K(t_1,t)}{2\sigma^2}
\end{align*}
Multiplying by $e^{-\imath\theta}$ and taking the real part, this latter can be equivalently written as
\[
\mathrm{Re}\,(e^{-\imath\theta}Z^{(\lambda)}(t))=X(z)+({\lambda}-{\lambda_1})\cos(\theta_1-\theta)\frac{K(t_1,t)}{2\sigma^2}\,,
\]
where $z=(t,\theta)\in\bbT$. Now, recall that $\rho(t,\theta):=\Gamma(t)\cos\theta=\cos(\theta){K(0,t)}/({2\sigma^2})$ to compute
\eq
\label{eq:X_decompo}
\mathrm{Re}\,(e^{-\imath\theta}Z^{(\lambda)}(t))=X(z)+({\lambda}-{\lambda_1})\rho(z-\hat z)\,.
\qe
We deduce that
\begin{align*}
\mathrm{Re}\,(e^{-\imath\theta}Z^{(\lambda)}(t))\leq\lambda&\Leftrightarrow X(z)-{\lambda_1}\rho(z-\hat z)\leq\lambda(1-\rho(z-\hat z))\\
&\Leftrightarrow \frac{X(z)-X(\hat z)\rho(z-\hat z)}{1-\rho(z-\hat z)}\leq\lambda\\
&\Leftrightarrow X^{\hat z}(z)\leq\lambda
\end{align*}
showing that the second knot $\lambda_2$ is exactly the quantity defined in~\eqref{eq:second_knot}.

\subsection{Illustration: The two first knots of Super-Resolution}
\label{sec:SRLARS}

The Super-Resolution process is defined in \eqref{eq:CorrZSR}. It satisfies Condition~\eqref{a:KLN} and Condition~\eqref{a:nonDegeneratefiniteComplex} of Section~\ref{sec:VarianceEstimation} with $N= 2f_c +1$. The first point is given by the maximum of the modulus of $Z$, see the red curve in Figure~\ref{fig:LARS_SR}. Observe that $Z_1=Z$ and the maximum satisfies $Z_1(t_1)=\lambda_1\,e^{\imath \theta_1}$. Then, we compute 
\[
Z^{(\lambda)}(t)=Z_1(t)+(\frac{\lambda}{\lambda_1}-1)Z_1(t_1)\frac{{\mathbf D_{N}(t_1-t)}}{2N\sigma^2}\,,
\]
where $\mathbf D_{N}$ denotes the Dirichlet kernel. For $\lambda>\lambda_2$, the maximum of $|Z^{(\lambda)}|$ is achieved at a unique point, namely $t_1$. For $\lambda=\lambda_2$, a second point achieves the maximum. This transition defines $Z_2:=Z^{(\lambda_2)}$, see Figure~\ref{fig:LARS_SR}.

\noindent
From this point, we can iterate fitting the least squares direction on the support $\{t_1,t_2\}$ and decreasing~$|Z_2|$ while a third point achieves the maximum. Given the red curve in Figure~\ref{fig:LARS_SR}, it was not obvious that the second knot would have been~$t_2$ since other local maxima seemed more significant than $t_2$ on the red curve.

\section{Notations and problem formulation}
\label{sec:prob_formulation}
\subsection{Hypothesis testing problem}
\label{sec:HypTestProb}
In this paper, our purpose is to test the mean value of a stationary complex-valued Gaussian process $Z$  with $\mathcal C^2$-paths indexed by~$[0,2\pi)$. We assume that $Z=A_1+\imath A_2$ where~$A_1$ and $A_2 $ are two independent and identically distributed real-valued processes  with $\mathcal C^2$-paths. Assume that the correlation function~$\Gamma$ of~$A_1$ (and $A_2$) satisfies
\eq
\label{e:Normalization}
\tag{$\mathbf A_{\mathrm{norm}}$} 
\forall t\in(0,2\pi),\quad 
|\Gamma(t)|<1 \,
\qe
and let $\sigma^2 := \bbV\!\mathrm{ar}(A_1(\cdot))$ so that 
\eq
\label{eq:variance}
\bbC\mathrm{ov}(A_1(s),A_1(t)) = \sigma^2 \Gamma(t-s)\,.
\qe 
We denote by $\bbT:=[0,2\pi)^2$ the $2$-dimensional torus. Assume that we observe a real-valued process~$(X(z))_{z\in\bbT}$ indexed by $\bbT$ such that 
\[
\forall z \in\bbT,\quad X(z):=A_1(t)\cos\theta+A_2(t)\sin\theta=\mathrm{Re}\big(e^{-\imath\theta}Z(t)\big)\,,
\]
where $z = (t,\theta)$ and $\mathrm{Re}(\cdot)$ denotes the real part of a complex number. Remark that observing~$X$ is equivalent to observe $Z$ since we can recover $Z$ from $X$ and conversely. Furthermore, we may assume that the process $(X(z))_{z\in\bbT}$ satisfies 
\eq
\label{e:NonDegenerated}
\tag{$\mathbf  A_{\mathrm{degen}}$} 
  \mbox{a.s.  there is no point }z\in\bbT \mbox{ s.t.  }X'(z) =0\ \mbox{and}\ \det(X''(z)) =0,
\qe
where $X'(z)$ and $X''(z)$ denote the gradient and the Hessian of $X$ at point $z$. Note that sufficient conditions for \eqref{e:NonDegenerated} are given by \cite[Proposition 6.5]{Azais_Wschebor_09} applied to $(X(z))_{z\in\bbT}$. In particular if the distribution of $X''(t)$ is non degenerated, using  \cite[Condition (b) of Proposition~6.5]{Azais_Wschebor_09}, it implies that Assumption~\eqref{e:NonDegenerated}  is met. Note also that Assumption~\eqref{e:NonDegenerated} is referred to as ‘‘Morse'' process in \cite{adler2009random}. Remark that \eqref{e:Normalization} and \eqref{e:NonDegenerated} are mild assumptions ensuring that $Z$ is a non-pathological process with $\mathcal C^2$-paths.

\medskip

\noindent
This paper aims at testing the following hypotheses:
\eq
\notag
\mathds H_0: \mbox{``}Z\text{\ is centered}\,\mbox{''} \quad\text{against}\quad  \mathds H_1: \mbox{``}Z\text{\ is not centered}\,\mbox{''}\,.
\qe
Remark that this framework encompasses any testing problem whose null hypothesis is a single hypothesis on the mean of~$Z$, subtracting the mean tested by the null hypothesis. Indeed, remark that $Z$ can always be decomposed into 
\[
Z = Z^0 + \eta\,,
\]
where $Z^0=\bbE Z$ is the deterministic noiseless response and $\eta$ is some centered random additive perturbation of $Z^0$. Given any function $f^0$, one might be interested in testing wether $Z^0=f^0$ or equivalently $Z-f^0$ is centered. Not rejecting this hypothesis means that there is no evidence that the residual $Z-f^0$ is not centered. On the other hand, rejecting the null means that the testing procedure have found some evidence that one should not consider that the residual $Z-f^0$ is centered. Now, the same discussion can be made for $X$ remarking that
\[
X(z)=\mathrm{Re}\big(e^{-\imath\theta}Z(t)\big)=X^0(z)+N(z)
\]
where we denote by $X^0(z):=\mathrm{Re}\big(e^{-\imath\theta}Z^0(t)\big)$ the deterministic noiseless response part and by $N(z):=\mathrm{Re}\big(e^{-\imath\theta}\eta(t)\big)$  some centered random additive perturbation of $X^0$.

 \subsection{The first and second knots of a Gaussian process} 
As in high-dimensional statistics, we can define the first and second knots~$(\lambda_1,\lambda_2)$ as follows. If we model some spatial correlation by means of the process $X$, the most correlated point $\widehat z\in\bbT$ and the maximal correlation $\lambda_1$ are respectively the argument maximum and the maximum of $X$ defined by 
\eq
\label{def:hatz}
\widehat z:=\argmax_{z\in\bbT}X(z)\quad\mathrm{and}\quad\lambda_1:=X(\widehat z)\,.
\qe
Under Assumption \eqref{e:Normalization}, one can check that the argument maximum is almost surely a singleton, see~Proposition~\ref{prop:Unicity}.

To construct the second knot, given a fixed $z\in\bbT$, one can equivalently consider two regressions of~$X(y )$, as follows.

 $\bullet$ On the one hand, the regression on $X(z)$ that will appear in the grid method of Section~\ref{s:asador}. Using a  convenient normalisation related to the definition  of the LARS knots, we set 
 
 \[
\forall y\in\bbT\setminus\{z\},\quad X^{z}(y):=\frac{X(y)-X(z)\rho( z-y)}{1-\rho(z-y)}=X(z)+\frac{X(y)-X(z)}{1-\rho(z-y)}\,,
\]
where 
\[
\forall z\in\bbT,\quad
\rho(z):=\Gamma(t)\cos\theta\,,
\] is the correlation function of the stationary Gaussian process~$X$. One can check that $X^{z}$ is a Gaussian process indexed by $\bbT\setminus\{z\}$ and independent of $X(z)$.  
 
  $\bullet$ On the other hand, the regression on $(X(z),X'(z))$  will be needed for convergence purposes in Section~\ref{s:riz}. With the convenient normalization, we set 
  \[
\forall y\in\bbT\setminus\{z\},\quad  X^{|z}(y) :=\frac{ X(y) -\rho(z-y) X(z) +  \langle \rho'(z-y) , \widetilde \Lambda^{-1} X'(z) \rangle }{1-\rho(z-y) }.
\] 
where $\rho'$ is the gradient of the correlation function $\rho$ and $\widetilde{\Lambda} := -\rho''(0)$ is the variance-covariance matrix of the derivative process of $X$, namely $X'$. \bigskip

Since the derivative at $\widehat z$ is zero, note that $X^{\hat{z}}(\cdot) = X^{|\hat{z}}(\cdot)$ and we define the second knot~$\lambda_2$~as 
\eq
\label{eq:second_knot}
\widehat y:=\argmax_{y\in\bbT\setminus\{\widehat z\}}X^{\widehat z}(y)\quad\mathrm{and}\quad\lambda_2:=X^{\widehat z}(\widehat y) = X^{| \widehat z}(\widehat{y})\,,
\qe
where we prove that $(\widehat y,\lambda_2)$ are well defined and that $\widehat y$ is almost surely unique, see Proposition~\ref{prop:Unicity} and Remark~\ref{rem:Pumping}. Furthermore, the couple $(\widehat y,\lambda_2)$ can be equivalently defined using the extension of the LARS to our framework, the interested reader may consult Section~\ref{sec:second_knot}.

\section{Passing to the limit, the grid approach} \label{s:asador}

The main idea of this section is to define a sequence of grids $(G_n)_{n\geq1}$ on $\bbT$,  to construct a sequence of test statistics $(S_n)_{n\geq1}$ from the values of the process $X$ on $G_n$  as in \cite{ADCM16} and to pass to the limit as $n \to \infty$. More precisely, we consider~$G_n $ to be the grid with  mesh $\Delta_n:=(2\pi) 2^{-n}$ on $ \bbT$ (corresponding to $P=2^{2n}$ grid points so that $n=(\log_2 P)/2$),
 \[
 \widehat z_n := \argmax _{z \in G_n}  X(z)\quad \mathrm{and}\quad \lambda_{1,n}  :=\max _{z \in G_n}  X(z)\,.
 \]
 It is the maximum of the process $X$ when indexing by the grid. We can also define the maximum of the regression when indexing by the grid, namely
 \[
 \lambda_{2,n}  :=\max _{y\in G_n\setminus\{\widehat z_n\}}  X^{\widehat z_n}(y)\,.
 \]
The Hessian at the maximum \eqref{def:hatz} on $\bbT$ is denoted by $X'':=X''(\widehat z)$ (in particular it does not depend on the grid but on the maximum $\hat z$ of $X$). By Assumption~\eqref{e:NonDegenerated}, it is a random variable with values in the set of non degenerated negative definite matrices of size $2\times 2$.  We can define a non degenerated positive quadratic form (\textit{i.e.,} a metric) on~$\bbR^2$ by $\|v\|_{X''} = -v^\top X''v$, for $v\in\bbR^2$. Using this metric, we can consider the corresponding Voronoi tessellation of~$\bbZ ^2$. It is a regular partition of $\bbR^2$ by parallelograms, invariant by translations~$(1,0)$ and $(0,1)$. Denote by $V_\mathrm{o}\subset[-1,1]^2$  the Voronoi cell of the origin in this partition and by $\mathcal U:=\mathcal U(V_\mathrm o)$ the uniform distribution on this cell. We understand the law $\mathcal U$ as a conditional law with respect to $X''$ and, conditionally to $X''$, this law is taken independent of $(\lambda_1,\lambda_2)$, see Lemma~\ref{lem:z1}. Conditionally to $X''$, define the randomized statistics
\eq
\label{eq:limit_l2}
  \overline \lambda_2 :=
  \lambda_2 \vee \Bigg\{\lambda_1+\sup _{k \in\bbZ^2\setminus\{0\}} \frac{k^\top}{ \| {\widetilde \Lambda}^\frac12 k\|} X'' \Big(\frac{ k-2 \mathcal{U }}{\|{\widetilde \Lambda}^\frac12 k\|}\Big)\Bigg\}\,,
\qe
where ${\widetilde \Lambda}^\frac12$ is the square root of $\widetilde \Lambda=-\rho''(0)$ and $a\vee b=\max(a,b)$. A proof of the following result is given in Appendix~\ref{proof:grid_known_variance}.

\begin{remark}
Remark that we have taken dyadic grids here. Following the proof in Appendix~\ref{proof:grid_known_variance}, one can exhibit how $\bar\lambda_2$ depend on the sequence of grids. The key result is Lemma~\ref{lem:z1} and we borrow its notation in this remark. In the general case where one consider a different type of sequence of grids, one still have independence between $( \widehat z- \overline{z}_n)$ and $(\lambda_1,\lambda_2)$ but the law of the limit of $\Delta_n^{-1}( \widehat z- \overline{z}_n)$ (for some $\Delta_n$ that may depend on the grid sequence) may differ from $\cU$. We refer to this law (if it exists) as $\cV_k$ where $k\in\bbZ^2$. The dependence in $k$ depicts the fact cells defined by joining adjacent points of the grid might be topologically different (which is not the case in the dyadic case). It results that the definition of $\bar\lambda_2$ should be modified changing $\cU$ by $\cV_k$. It does not change the main message here: the resulting test is randomized and \eqref{e:naive} is non-conservative and should be avoided in pratice.
\end{remark}

\begin{theorem}
   \label{thm:grid_known_variance}
   
 Under~$\bbH_0$, Assumptions \eqref{e:Normalization} and \eqref{e:NonDegenerated}, the randomized test statistics 
 \[
 S^{\mathrm{Grid}}:= \frac{\overline \Phi(\lambda_1/\sigma)}{\overline \Phi ( \overline\lambda_2/\sigma)} \sim\mathcal U([0,1])\,,
 \] 
 where $\overline\Phi$ denotes the standard Gaussian survival function. Moreover, the test with $p$-value $S^{\mathrm{Grid}}$ is unbiased: under the alternative $\bbH_1$, it holds  $ \bbP \{S^{\mathrm{Grid}} \leqslant  \alpha\} \geqslant \alpha$ for all $\alpha\in(0,1)$.
   \end{theorem}
   
   Theorem \ref{thm:grid_known_variance} shows in particular  that  the  statistics\textemdash referred to as the Spacing test statistics in the introduction\textemdash given by
 \begin{equation} \label{e:naive}
 S^{\mathrm{ST} }= \frac{\overline \Phi(\lambda_1/\sigma)}{\overline \Phi ( \lambda_2/\sigma)} 
  \end{equation}
   does not follows a $\mathcal U([0,1])$ distribution under~$\bbH_0$ ans leads to a non-conservative test. Indeed, observe that almost surely $\lambda_2\leq\bar\lambda_2$ so that $S^{\mathrm{ST} }\geq S^{\mathrm{Grid}}$ almost surely. Note that the two test statistics differ on the event $\{\lambda_2\neq\bar\lambda_2\}=\{\lambda_2<\bar\lambda_2\}$.

Now, when the variance $\sigma^2$ is unknown, we can build an estimator $\widehat\sigma^2$ defined in~\eqref{eq:sighatz} and obtain a studentized version of the previous theorem. Please consult Section~\ref{sec:VarianceEstimation} for further details on the construction of the estimator $\widehat\sigma$ and on Conditions \eqref{a:KLN} and \eqref{a:nonDegeneratefiniteComplex}.

\begin{theorem}  \label{t:t}  
Assume \eqref{e:Normalization}, \eqref{e:NonDegenerated}, \eqref{a:KLN} and \eqref{a:nonDegeneratefiniteComplex} where $2\leqslant N<\infty$, then the following test statistics $T^{\mathrm{Grid}}$ satisfies 
	 $$
	  T^{\mathrm{Grid}}:= \frac{ \overline F _{m-1} \left({\lambda_1}/{\widehat \sigma} \right) }{\overline F _{m-1} \big( {\overline{\lambda}_2}/{\widehat \sigma}  \big) } \sim\mathcal U([0,1])\,$$
under~$\bbH_0$ where $m = 2 N$, $  F_{m-1}$ is the Student cumulative distribution function with $m-1$ degrees of freedom, $ \overline  F _{m-1}  = 1-F_{m-1} $ its survival function and $ \widehat \sigma^2$ is  defined by~\eqref{eq:sighatz}. 
\end{theorem}
A proof can be found in Appendix~\ref{proof:t:t}.
	  
	  
\begin{remark}
Only the first point of \eqref{a:nonDegeneratefiniteComplex} is required for the proof. Moreover, if  $ ~m= +\infty$, the Student distribution is to be replaced by a standard normal distribution. 
\end{remark} 

\section[The Rice method: a grid-less approach]{The Rice method: a grid-less approach} \label{s:riz}
In this section, we build our test statistic directly on the entire path of the process $X$ in a grid-less manner. We assume that the process $X$ satisfies  Assumptions \eqref{e:Normalization} and \eqref{e:NonDegenerated}, and is {centered} (namely~$\bbH_0$).
  As in the preceding section,  we consider $\lambda_1$  and $\lambda_2$ defined by~\eqref{def:hatz} and \eqref{eq:second_knot} respectively.
  
We denote $X = \sigma   \widetilde X$ so that the covariance function of $\widetilde X$ is the correlation function $\rho$ of $X$, namely~$\widetilde X$ is the standardized version of $X$. Note that, by regression formulas and stationarity, it holds
  $$
  \forall z\in\bbT,\quad
  \E\big[\widetilde X''(z)\big|(\widetilde X(z), \widetilde X'(z))\big] = -\widetilde\Lambda  \widetilde X(z)\,,
  $$
  so that  we can define the process $\widetilde R$ by the decomposition 
  $$ \widetilde X''(z) = -\widetilde\Lambda  \widetilde X( z) +    \widetilde R(z)$$
 where~$ \widetilde R(z)$ and~$\widetilde X(z)$ are independent for any $z\in\bbT$ and $\widetilde\Lambda=-\rho''(0)$ is the variance-covariance matrix of~$\widetilde X'(t)$. In particular, observe that
 \beq
 \notag
X''(\widehat z) = -\widetilde\Lambda X(  \widehat z) +   R(\widehat z)\,,
  \eeq
  where $R(\widehat z)=\sigma\widetilde R(\widehat z)$. Using the  Rice method of \cite[Theorem 7.2]{Azais_Wschebor_09} (see also \cite{lockhart2014significance}),  it follows that the maximum $  \lambda_1$ has for density  w.r.t the Lebesgue measure on~$\bbR^+$ at point $\ell>0$
\eq
\notag
\mathtt{(cst)}
(-1)^d\E\big[ \det(-\widetilde\Lambda X( 0) +   R(0)) \1_{A_{\ell}} \big| X(0) = \ell, X'(0) = 0\big]  \sigma^{-1}\phi( \sigma^{-1}\ell), 
\qe
where $\phi$~denotes the standard Gaussian density, $A_\ell$  is the event $\{X(y) \leqslant  \ell,\  \forall y\in\bbT\} $ and~$\mathtt{(cst)}$, as in the following, denotes a positive constant. The numerical values $\mathtt{(cst)}$ may vary from an occurence to another and it may depend  on $m$ and $\sigma$ which are assumed fixed in our framework.

  \subsection{The known variance case}
  We begin by the known variance case. The main observation is that the method of  \cite[Theorem~7.2]{Azais_Wschebor_09} can be extended to compute the joint distribution of $ (\lambda_1, \lambda_2,  R(\widehat z)) $ as follows.     
  \begin{itemize}  
 \item Denote $\bbS\!$ the set of symmetric matrices and pick a Borel set $B$ on $\mathcal D :=\bbR^2\!\times\bbS$.

\item For every $z \in \bbT$, recall that
$$
\forall y\in\bbT\setminus\{z\},\quad 
 X^{|z}(y) :=\frac{ X(y) -\rho(z-y) X(z) +  \langle \rho'(z-y) , \widetilde \Lambda^{-1} X'(z) \rangle }{1-\rho(z-y) }
$$
and define
\eq
\label{eq:lambdaT}
\forall z\in\bbT,\quad
\lambda_2^z:=\sup_{y\in\bbT\setminus\{ z\}}X^{|z}(y)\,.
\qe
Remark that, for fixed $z\in\bbT$, $\lambda_2^z$ is a.s. finite by Lemma \ref{l:jm:l1}, $X^{|z}(\cdot)$ is independent of~$(X(z),X'(z))$ and, by way of consequence, $\lambda_2^z$ is independent of $(X(z),X'(z))$. Furthermore, note that since~$\bbT$ is without boundary, for $z=\widehat z$, one has $X'(z)=0$ and $\lambda_2^z=\lambda_2$ as defined by~\eqref{eq:second_knot}.
\item Observe that on the event $\{ \forall y \neq z,\ X(y) < X(z)\}$ one has almost surely that  $z=\widehat z$, $X(z)=\lambda_1$, $ \lambda^z_2 =\lambda_2$ and $R(z)=R(\widehat z)$. Also, a simple computation shows that
  \[
  \forall z\in\bbT\ \mathrm{s.t.}\ X'(z)=0,\quad
  \1_{\{z=\widehat z\}}=
  \1_{\{ \forall y \neq z,\ X(y) < X(z)\}}=\1_{\{0<\lambda_2^z< X(z)\}},
  \] 
  almost surely. Hence, by unicity of $\widehat z$ and recalling that the set $\{{z\, ;\  X'(z) =0}\}$ is finite under~\eqref{e:NonDegenerated}, we deduce that
  \[
  \sum_{z:X'(z) =0}\1_{\{( X(z),  \lambda^z_2 ,  R(z))\in B\}\cap\{0<\lambda_2^z< X(z)\}}=\1_{\{(\lambda_1 , \lambda_2 ,  R(\widehat z)) \in B\}}.
  \]
 \item   On $\mathcal D$ define smooth lower approximations $ \varphi^{(n)}_B$  of the indicator function of $B$  that converge when~$n$ goes to infinity i.e. 
 \[
 \forall (\ell_1,\ell_2,r)\in\Omega,\quad\varphi^{(n)}_B(\ell_1,\ell_2,r)\longrightarrow \1_{\{(\ell_1,\ell_2,r)\in B\}\cap\{0<\ell_2<\ell_1\}}\,.
 \]

 \item   Apply Rice formula with weights  \cite[Theorem 6.4]{Azais_Wschebor_09} (see also the proof of  \cite[Theorem~7.2]{Azais_Wschebor_09}) to compute
\begin{align*}
\E&\Big[
\sum_{z:X'(z) =0}  \varphi^{(n)}_B \big( X(z),  \lambda^z_2 ,  R(z) \big )\Big]
\\&= \mathtt{(cst)} \int_{\bbT}    \E\Big[  |\det( -  \widetilde \Lambda  X(z) + R(z))|\,  \varphi^{(n)}_B(X(z), \lambda_2^z, R(z))\ \Big|\ X'(z) =0\Big] \mathrm dz  \,
\\&= \mathtt{(cst)} \int_{\bbT}    \E\Big[  |\det( -  \widetilde \Lambda  X(z) + R(z))|\,  \varphi^{(n)}_B(X(z), \lambda_2^z, R(z))\ \Big]   \,
 \mathrm dz 
\end{align*}
where the last equality relies on the fact that  $(X(z),\lambda_2^z,R(z))$ is independent of $X'(z)$.
  \item  Combining the previous observations and passing to the monotone  limit  as $n$ tends to~$\infty$  in the aforementioned Rice formula with weights, we get that 
\begin{align} 
\P  &\big\{  \big(\lambda_1 , \lambda_2 ,  R(\widehat z) \big) \in B \big\} \notag
\\& =\E \Big[\sum_{z:X'(z) =0}\1_{\{( X(z),  \lambda^z_2 ,  R(z))\in B\}\cap\{0<\lambda^z_2<X(z)\}}\Big]\notag
\\& = \mathtt{(cst)} \int_{\bbT}    \E\Big[  |\det( -  \widetilde \Lambda  X(z) + R(z)) | 
 \1_{ \{(X(z), \lambda_2^z, R(z))\in B\}\cap\{0<\lambda_2^z< X(z)\}} \Big]  \ 
 \mathrm dz  \notag
\\&
 =\mathtt{(cst)} \,  \E\Big[ | \det( -  \widetilde \Lambda
  X(0)   +R(0)) |
   \1_{ \{(X(0), \lambda_2^0, R(0))\in B\}\cap\{0<\lambda_2^0< X(0)\}}  \Big]   \,,\notag\\
   &=\mathtt{(cst)} \,   \E\Big[  \det( -  \widetilde \Lambda
  X(0)   +R(0)) 
 \1_{ \{0< \lambda_2^0<X(0)\} } \1_{\{ (X(0), \lambda_2^0, R(0))\in B\}}   \Big]   \,,\label{eq:BigStepRice}
\end{align}
by stationarity and using that, on the event $\{0< \lambda_2^0<X(0)\}$, the matrix $-X''(0)=   \widetilde \Lambda
  X(0)   -R(0)$ belongs to the set of positive definite symmetric matrices, namely $\bbS^+$.
\end{itemize}

Before stating the key result on the joint density of $ (\lambda_1, \lambda_2, R(\widehat z)) $ we need to introduce a dominating measure. First, recall that $X(0)$ is independent of the pair $( \lambda_2^0, R(0)) $. Then, observe that $( \lambda_2^0, R(0))=\sigma\times( \widetilde\lambda_2^0, \widetilde R(0)) $ where $ \widetilde\lambda_2^0$ is defined as in \eqref{eq:lambdaT} for the process~$\widetilde X$. Denote~$\mu_1$ the law of $( \widetilde\lambda_2^0, \widetilde R(0))$ and note that it does not depend on $\sigma$. Denote $\mu_\sigma$ the law of~$( \lambda_2^0, R(0))$ and remark that for any Borel set~$B$ of $\bbR\times\bbS$, it holds $\mu_\sigma(\sigma B)=\mu_1(B)$. Eventually, remark that
\eq
\label{eq:domin}
\mbox{The law of } (X(0),\lambda_2^0, R(0)) \mbox{ is dominated by } \mathrm{Leb}(\bbR)\otimes\mu_\sigma\,.
\qe 
where $\mathrm{Leb}(\bbR)$ denotes the Lebesgue measure on~$\bbR$. As a consequence we can prove the following proposition.

  \begin{proposition}
\label{prop:Rice_density_known_variance}
Under~$\bbH_0$, the joint law $\mathcal L((\lambda_1, \lambda_2, R(\widehat z)))$ of  $ (\lambda_1, \lambda_2, R(\widehat z)) $ satisfies for all $(\ell_1,\ell_2,r)\in\bbR^2\times\bbS$,
 \begin{align*}
 \frac{\mathrm d\mathcal L((\lambda_1, \lambda_2, R(\widehat z)))}{\mathrm d \mathrm{Leb}(\bbR)\otimes\mu_\sigma}&(\ell_1,\ell_2,r)= \mathtt{(cst)} \det( -\widetilde \Lambda
  \ell_1  +r) \1_{\{0<\ell_2<\ell_1\}}  \sigma^{-1}\phi( \sigma^{-1}\ell_1)\,,
 \end{align*}  
 where $\mathrm{Leb}(\bbR)\otimes\mu_\sigma$ is defined by~\eqref{eq:domin} and $\bbS\!$ denotes the set of symmetric matrices.
\end{proposition}

\begin{proof}
Observe that the density at point $\ell_1$ of $X(0)$ with respect to the Lebesgue measure is~$ \sigma^{-1}\phi( \sigma^{-1}\ell_1)$ and recall \eqref{eq:domin}. Now, for any Borel set $B$ of $\bbR^2\times\bbS$, note that
\begin{align*}
\E&\Big[  \det( -  \widetilde \Lambda  X(0)   +R(0))  \1_{ \{(0<\lambda_2^0<X(0)\}}   \1_{ \{(X(0), \lambda_2^0, R(0))\in B\}}  \Big] \\
&=\int_B  \det( -  \widetilde \Lambda  \ell_1   +r) \1_{ \{0<\ell_2<\ell_1\}}\sigma^{-1}\phi( \sigma^{-1}\ell_1)\mathrm{d}\ell_1\mu_\sigma(\mathrm{d}(\ell_2,r))
\end{align*}
thanks to \eqref{eq:BigStepRice}, which prove the result.
\end{proof}

\noindent
We can now state our result when the variance is known.
\begin{theorem}
 \label{thm:rice_known_variance}
Set
  $$
  \forall r\in\bbS^+,\ 
  \forall \ell>0,\quad 
  \overline G_r(\ell) :=\int_\ell ^{+\infty}  \det (-\widetilde\Lambda u +r)    \phi(u \sigma^{-1}) \mathrm{d}u\,,
  $$
where  $\widetilde \Lambda$ denotes the Hessian of the correlation function $\rho$ of $X$ at the origin.
  Under Assumptions \eqref{e:Normalization} and~\eqref{e:NonDegenerated}, the test statistic  
  $$
 S^{\mathrm{Rice}} :=\frac{\overline  G_{R(\widehat z)} (\lambda_1)}{  \overline  G_{R(\widehat z)} (\lambda_2)} \sim \mathcal U([0,1])
  $$
 under~$\bbH_0$.

    \end{theorem}

\noindent
\begin{proof}

Using Proposition \ref{prop:Rice_density_known_variance}, we know that  the density of $\lambda_1$ at $\ell_1$  and conditional to $ (\lambda_2,R( \widehat z))=(\ell_2,r)$ is equal to 
$$
\mathtt{(cst)}\, \det(-  \widetilde \Lambda \ell_1 +r)  \phi(\sigma^{-1}\ell_1)  \1_{\{0<\ell_2 <\ell_1\} } .
$$
It is well known  that, if a random variable $Z$ has for cumulative density function $\bbF$ then $\bbF(Z)$ follows an uniform distribution  on~$[0,1]$. This implies  that, conditionally to  $ (\lambda_2,R( \widehat z))=(\ell_2,r)$, 
$$
\frac{\bar  G_r (\lambda_1)}{\bar  G_r (\ell_2)} \sim \mathcal U([0,1]).
$$
Since the conditional distribution does not depend on $(\ell_2,r)$, it is also the non conditional distribution and it yields
$$
\frac{\bar  G_{R(\widehat z )}(\lambda_1)}{\bar  G_{R(\widehat z )} (\lambda_2)} \sim \mathcal U([0,1])\,,
$$
as claimed.
\end{proof}

\subsection{The unknown variance case}

\subsubsection{Estimating the variance}
\label{sec:VarianceEstimation}
When the variance $\sigma^2$ is unknown in~\eqref{eq:variance}, we precise here the assumptions and the estimator we use to estimate the variance. In this section, except for explicit examples, we consider a real valued centered Gaussian process $Y$ {not necessarily stationary} defined on the $2$-dimensional torus~$\bbT$. Let $m\geq2$ (possibly infinite) and assume that $Y$ admits an order $m$ Karhunen-Lo\`eve expansion in the sense that
\begin{align*}
\label{a:Klm}
 Y = \sum_{i=1}^{m} \zeta_i f_i \mbox{ with } \bbV\!\mathrm{ar}(\zeta_i) = \sigma^2 \mbox{ and}~\forall t\in\bbT,\ \sum_{i=1}^{m}| f_i(t)|^2 = 1\,,
 \tag{${\mathrm{KL}(m)}$}
\end{align*}
where the equality holds in $\bbL^2(\Omega)$ and $(f_1,\ldots,f_m)$ is a system of non-zero functions orthogonal on~$\bbL^2(\bbT)$. Through our analysis, we need to consider one of the following assumptions.
\begin{itemize}
\item If $m$ is finite, 
\begin{align*}
&\exists (z_1,\dots,z_m) \in\bbT^m ~\text{ pairwise distincts s.t.} 
\\
&(Y(z_1),\dots,Y(z_m))\ \mathrm{is\ non\ degenerated}.
\label{a:nonDegeneratem}
\tag{${\mathrm{ND}(m)}$}
\end{align*}
\item If $m = \infty$,
\begin{align*}
\notag
&\forall p \in \mathbb{N}^\star,\ \exists (z_1,\ldots,z_{p})\in{\bbT^{p}} ~\text{ pairwise distincts s.t.}\\
&(Y(z_1),\ldots,Y(z_{p}))\ \mathrm{is\ non\ degenerated}.
\label{a:nonDegenerateInfinite}
\tag{${\mathrm{\,ND}(\infty)}$}
\end{align*}
\end{itemize}
Recall that a Gaussian vector is called non-degenerated if its variance-covariance matrix is non-degenerated, \textit{i.e.,} it has full rank.

 \medskip
 
Some examples  of process $Y$ satisfying \eqref{a:Klm} and \eqref{a:nonDegeneratem} with $m=\infty$ are given by
 the normalized Brownian motion and any Gaussian stationary process  with a spectrum that admits an accumulation point, see~\cite[Page 203]{cramerstationary}. For instance, the process corresponding to the Super-Resolution problem satisfies \eqref{a:Klm} and~\eqref{a:nonDegeneratem} with $m$ finite, namely $m$ is  twice the number of observed frequencies, see Section~\ref{sec:SR}.
  
    \begin{definition} 
    \label{def:kl}
Let $Y$ be a Gaussian process with constant variance $\sigma^2  =\mathds V \mathrm{ar}(Y(\cdot))$  and  satisfying Assumptions~\eqref{a:Klm} and~\eqref{a:nonDegeneratem} with $m$ finite. The quantity 
	 \beq \notag
	 \widehat \sigma^2_{\mathrm{KL}}(Y):= \frac1m \sum_{i=1}^m \zeta_i^2,
	 \eeq
	 is called the   Karhunen-Lo\`eve estimator of $\sigma^2$.
	 \end{definition}
	 
\begin{remark}
An  explicit expression  of the estimator $\widehat \sigma ^2_{\mathrm{KL}}$ is always  possible  from {some} set of pairwise disjoint points $z_1,\ldots,z_{m'}$ with $m'\geqslant m$. We only need to check that the variance-covariance matrix of the $(Y(z_1),\ldots,Y(z_{m'}))$ has rank $m$.
\end{remark}
	 
\begin{remark}
Sufficiency considerations  imply that  $\widehat \sigma ^2_{\mathrm{KL}}$  is an optimal unbiased estimator for the mean-squared error by Rao–Blackwell theorem.
\end{remark} 

Given the aforementioned definition, we are now able to construct variance estimators for the process~$X$. We assume that the complex Gaussian process $Z$ that define $X$ satisfies the following hypotheses for some $N \in \mathbb{N}$.
\begin{equation*}
Z ~\text{admits a complex Karhunen-Lo\`eve expansion of order } N
\label{a:KLN}
\tag{$\mathrm{KL}_Z(N)$}
\end{equation*}
and satisfies the following non-degeneracy conditions:
\begin{align*}
&\forall (t_1,\ldots,t_{N})\in{[0,2\pi)^{N}} ~\text{ pairwise distincts,}\\
&(Z(t_1),Z(t_2),\ldots,Z(t_{N}))\ \mathrm{is\ non\ degenerated\ and}\tag{$\mathrm{ND}_Z(N)$} 
\label{a:nonDegeneratefiniteComplex}
\\
&(Z(t_1),Z'(t_1),Z(t_3),\ldots,Z(t_{N}))\ \mathrm{is\ non\ degenerated}.
\notag
\end{align*}

Our aim is to build, for each $z \in \bbT$, two estimators of the variance $\sigma^2$ independently from~$X(z)$ or $(X(z),X'(z))$. Indeed, in the following, we will distinguish two kind of statistics. The first one is the limit of the finite dimensional statistic $S^{\mathrm{Grid}}$, see Section \ref{s:asador}. The second one is the case of the maximum over $\bbT$, see Section \ref{s:riz}. Both cases won't use the same estimation of $\sigma^2$.
\begin{itemize}
\item In the grid situation, we define
                \beq
	    	    X_{\mathrm{norm}}^{z}(y) := \frac{X(y) - \rho(z - y) X(z)}{\sqrt{1 - \rho^2(z-y)}}\,,\notag
	    	    \eeq
	    	    where $y$ belongs to $\bbT\setminus\{z\}$, $\rho(\cdot)$ denotes the  correlation function of the process $X$ and~set
         \beq\notag
	     \widehat {\sigma}^2_{ z} := \widehat \sigma ^2_{\mathrm{KL}} \big(X_{\mathrm{norm}}^{z}(\cdot) \big) \eeq
	     which is well defined, independent of $X(z)$ and with constant variance $\sigma^2$. 
	     Finally, we consider the variance estimator
	     \beq \label{eq:sighatz} \widehat {\sigma}^2= \widehat {\sigma}^2_{\widehat z}, \eeq
	     defined at point $\hat z$ given by \eqref{def:hatz}.
\item In the continuous case, we define
	     \begin{equation}
	    	    X_{\mathrm{norm}}^{|z}(y) := \frac{X(y) - \rho(z - y) X(z) + \langle \rho'(z-y) , \widetilde \Lambda^{-1} X'(z) \rangle}{\sqrt{1 - \rho^2(z-y) + \langle \rho'(z-y), \widetilde \Lambda^{-1}\rho'(z-y) \rangle}}\,,\notag
	    	    \end{equation}
	    	    where $ y$ belongs to $\bbT\setminus\{z\}$ and set
		    	    	             \beq\notag
	    	    	     \widehat {\sigma}^2_{| z} := \widehat \sigma ^2_{\mathrm{KL}} \big(X_{\mathrm{norm}}^{|z}(\cdot) \big) \eeq
	    	    	     which is well defined, independent of  $(X(z),X'(z))$ and with constant variance $\sigma^2$. 
			     Finally, we consider the variance estimator
	    	    	     \beq \label{eq:sighatz2} \widehat {\sigma}_{|}^2= \widehat {\sigma}^2_{| \widehat z}\,, \eeq
			      defined at point $\hat z$ given by \eqref{def:hatz}.
\end{itemize}	    
		    
   \begin{proposition} 
   \label{prop:am}  Let $Z$ satisfy~\eqref{a:KLN} and \eqref{a:nonDegeneratefiniteComplex} and set $z \in \bbT$ then the following claims are true under~$\bbH_0$.
   \begin{itemize}
   \item[$(a)$] $\widehat {\sigma}^2_{z}$ is well defined and follows a  $\frac{\sigma^2 \chi^2_{2N-1}}{2N-1}$ distribution.
    \item[$(b)$] $\widehat {\sigma}^2_{|z}$ is well defined and follows a  $\frac{\sigma^2 \chi^2_{2N-3}}{2N-3}$ distribution.
     \item[$(c)$] The process $X_{\mathrm{norm}}^{z}(\cdot)/\widehat\sigma_z$ is independent of $\widehat\sigma_z$, and the process $X_{\mathrm{norm}}^{|z}(\cdot)/\widehat\sigma_{|z}$ is independent of the random variable $\widehat\sigma_{|z}$.
   \end{itemize}
   \end{proposition}

	      \begin{proof} $(a)$ Fix $z = (t_1,\theta_1) \in \bbT$. Since $Z$ satisfies \eqref{a:nonDegeneratefiniteComplex}, there exists $(t_2,\ldots ,t_{N})\in[0,2\pi)^{N-1}$ pairwise different such that $(Z(t_1),Z(t_2),\ldots,Z(t_{N}))$ is non degenerated. Then, considering $z_1 = z$, $z_{N+1} = z + (0,\pi/2)$ and
	      $$
	      \forall i \in \{2,\dots,N\}, ~ z_i = (t_i,\theta_1) ~\text{ and }~ z_{N+i} = (t_i,\theta_1 + \pi/2),
	      $$
	      the vector $V_1 := (X(z_1),\dots,X(z_{2N}))$ satisfies
	      \begin{align*}
	      2N = & \rank(X(z_1),X(z_2),\dots,X(z_{2N})) \\
	      =& \rank(X(z_1),X_{\mathrm{norm}}^{z}(z_2),\ldots,X_{\mathrm{norm}}^{z}(z_{2N})) \\
	      =& ~ 1 + \rank(X_{\mathrm{norm}}^{z}(z_2),\ldots,X_{\mathrm{norm}}^{z}(z_{2N}))
	      \end{align*}
	      where $\rank$ denotes the rank of the covariance matrix of a random vector. Deduce that~$X_{\mathrm{norm}}^{z}(\cdot)$ satisfies $({\mathrm{ND}(2N-1)})$. This, in turn, implies that the $2N$  functions 
	      $$
	      g_i(\cdot) = f_i(\cdot) -\rho(z-\cdot) f_i(z)
	      $$
	        are in fact in a space of dimension $2N-1$ and a Gram-Schmidt orthogonalization in $\bbL^2(\bbT)$ gives  $({\mathrm{KL}(2N-1)})$ for the process $X_{\mathrm{norm}}^{z}(\cdot)$. Finally, from $(X_{\mathrm{norm}}^{z}(z_2),\ldots,X_{\mathrm{norm}}^{z}(z_{2N}))$, 	        we compute $\widehat {\sigma}_{z}^2$ that follows the desired distribution.
	        
\medskip 

	        $(b)$ In the case of the regression over $(X(z),X'(z))$, remark that
	        $$
	        \partial_{\theta} X(z) = X(t_1,\theta_1+\pi/2) = X(z_2)
	        $$
	        and $\partial_{t} X(z) = \mathrm{Re}(e^{-\imath\theta_1} Z'(t_1))$ where $\partial_{\theta}$ (resp.~$\partial_{t}$) denote the partial derivative with respect to $\theta$ (resp.~$t$). Because of hypothesis \eqref{a:nonDegeneratefiniteComplex}, the two vectors $V_1$ and 
	        $$
	        V_2 := (X(z_1),X(z_2),\mathrm{Re}(e^{-\imath\theta_1} Z'(t_1)),\mathrm{Im}(e^{-\imath\theta_1} Z'(t_1)),X(z_5),\dots,X(z_{2N}))
	        $$
	        have rank $2N$ so both are invertible functions of $(\mathrm{Re}(\zeta_1),\mathrm{Im}(\zeta_1),\dots,\mathrm{Re}(\zeta_N),\mathrm{Im}(\zeta_N))$. In particular, $\mathrm{Im}(e^{-\imath \theta_1} Z'(t_1))$ is a linear combination of $V_2$. Let $\gamma_1$ and $\gamma_2$ be the coefficients associated to $X(z_3)$ and~$X(z_4)$. By triangular combination, we deduce that the distribution~of
	        $$
	        (X(z_1),\partial_\theta X(z_1),\partial_t X(z_1), \gamma_1 X(z_3) + \gamma_2 X(z_4),X(z_5),\dots,X(z_{2N}))
	        $$
	        is non-degenerated and so that $(\gamma_1,\gamma_2) \neq (0,0)$. Setting $\psi$ such that
	        $$
	        \cos(\psi) = \frac{\gamma_1}{\sqrt{\gamma_1^2 + \gamma_2^2}} ~~ \text{and} ~~ \sin(\psi) = \frac{\gamma_2}{\sqrt{\gamma_1^2 + \gamma_2^2}}
	        $$
	        we get the non-degeneracy of
	        $$
	        (X(z_1),\partial_\theta X(z_1),\partial_t X(z_1), X(z_{2N+1}),X(z_5),\dots,X(z_{2N}))
	        $$
	        where $z_{2N+1} = (t_2,\theta_1 + \psi)$. Finally, similarly to the proof of the previous point, regression, scaling and independence prove that the rank of $(X_{\mathrm{norm}}^{|z}(z_5),\dots,X_{\mathrm{norm}}^{|z}(z_{2N+1}))$ is $2N-3$ so that $X_{\mathrm{norm}}^{|z}(\cdot)$ satisfies $\mathrm{KL}(2N-3)$ and $\mathrm{ND}(2N-3)$ and that $\widehat {\sigma}_{|z}^2$ is well defined and distributed as $\frac{\sigma^2 \chi^2_{2N-3}}{2N-3}$.

\medskip

	       $(c)$ This is a direct consequence of the independence of the angle and the norm for each marginal Gaussian vector build from $X_{\mathrm{norm}}^z$ or $X_{\mathrm{norm}}^{|z}$.
	     \end{proof}

\begin{remark}
When the complex process $Z$ admits an infinite Karhunen-Lo\`eve decomposition, we need the following modified hypothesis
\begin{align*}
&\forall p \in \mathbb{N}^\star,\ \forall (t_1,\ldots,t_{p})\in{[0,2\pi)^{p}} ~\text{ pairwise distincts,}\\
&(Z(t_1),Z(t_2),\ldots,Z(t_{p}))\ \mathrm{is\ non\ degenerated\ and} \label{a:nonDegenerateinfiniteComplex}
\tag{$\mathrm{ND}_Z(\infty)$}
\\
&(Z(t_1),Z'(t_1),Z(t_3),\ldots,Z(t_{p}))\ \mathrm{is\ non\ degenerated}.
\end{align*}
Indeed, for every enter $p\geq1$, note that from  the observation  of the vector $(Z(t_1), Z(t_2),\ldots, Z(t_p))$ $($resp. $(Z(t_1),Z'(t_1),\dots,Z(t_p)))$ for pairwise disjoint points $t_1,\dots,t_p$, we can construct an estimator, say~$ \widehat \sigma^2_{2p}$ $($resp. $ \widehat \sigma^2_{|2p})$, of $\sigma^2$ with  distribution $ \sigma^2 \chi^2_{2p-1}/(2p-1)$ $($resp.~$ \sigma^2 \chi^2_{2p-3}/(2p-3))$ under~$\bbH_0$. Making~$p$ tend  to infinity, classical concentration inequalities and Borel-Cantelli lemma  prove that~$ \widehat \sigma^2_{2p}$ $($resp.~$ \widehat \sigma^2_{|2p})$  converges almost surely to $\sigma^2$ under~$\bbH_0$. Thus the variance $\sigma^2$  is directly observable from the entire path of~$X$. We still denote  $\widehat \sigma ^2_{z}$ $($resp.~$ \widehat \sigma^2_{|z})$ this observation, where $z= z_1 = (t_1,\theta_1)$. 
\end{remark}

\subsubsection{Computing the Joint Law}

Hence, suppose that we observe $X = \sigma   \widetilde X$ where $\sigma >0 $ is unknown. Assume that $Z$ satisfies~\eqref{a:KLN} and \eqref{a:nonDegeneratefiniteComplex}, and set $m=2N$. The regression of the Hessian on $(X(z),X'(z))$ reads now 
$$
\forall z\in\bbT,\quad
X''(z) = - \widetilde \Lambda X(z)  + \sigma  \widetilde R(z).
$$
because $X'(z)$ is independent of $(X(z),X''(z))$ by stationarity. The variance being unknown, we estimate it using $\widehat{\sigma}_{|}^2$ which is defined by~\eqref{eq:sighatz2}. For fixed $z\in\bbT$, by Claims~$(b)$ and $(c)$ of Proposition \ref{prop:am}, we know that the following random variables or random processes 
 $$
 X(z) \,,  X'(z) \,, \ \frac{X_{\mathrm{norm}}^{|z}(\cdot)}{\widehat \sigma_{|z}} \mbox{ and } \ \widehat \sigma_{|z}
 $$
 are mutually independent. As $X_{\mathrm{norm}}^{|z}(\cdot) = h_z(\cdot)\, X^{|z}(\cdot)$ where $h_z(\cdot)$ is a deterministic function and as Lemma \ref{l:jm1} shows that~$R(z)$ can be expressed as radial limits of  $X^{|z}(\cdot)$ at point $z$, we get that 
  $$
 X(z) \,, X'(z) \,, \ \Big( \frac{X^{|z}(\cdot)}{\widehat \sigma_{|z}}, \frac{R(z)}{\widehat \sigma_{|z}}  \Big) \mbox{ and } \ \widehat \sigma_{|z} \mbox{ are mutually independent,}
 $$
and by consequence
$$
 X(z) \,, X'(z) \,,\ \Big( \frac{\lambda_2^z}{\widehat \sigma_{|z}}, \frac{R(z)}{\widehat \sigma_{|z}} \Big) \mbox{ and } \ \widehat \sigma_{|z} \mbox{ are mutually independent.}
 $$
We turn now to the Rice formula described previously and introduce the notation 
$$ 
T_{2,z}:=    \frac{\lambda_2^z}{\widehat \sigma_{|z}} \  \mathrm{and} \ \ T_2:= T_{2, \widehat z }.
$$
Denote $\mathrm{Leb}(\bbR^2)$ the Lebesgue measure on $\bbR^2$ and let $ \bar{\mu}_1$  be the joint law of the couple of random variables~$(T_{2,0}, {R(0)}/{\widehat \sigma_{|0}})$. 
Under $\mathbb H_0$, note that $X(0)$ is a centered Gaussian variable with variance~$\sigma^2$ and~${\widehat \sigma_{|0}}/{\sigma}$ is distributed as a $chi$-distribution with $m-3$ degrees of freedom, {\it i.e.}, the law of density
$$
f_{\chi_{m-3}}(s) = \frac{2^{1-\frac{m-3}{2}}}{\bar{\Gamma}\left(\frac{m-3}{2}\right)} s^{m-4} \exp(-s^2/2)
$$
where $\bar{\Gamma}$ is the Gamma function. Then the  quadruplet $(X(0), {\widehat \sigma_{|0}}/{\sigma},T_{2,0}, {R(0)}/{\widehat \sigma_{|0}})$ has a density with respect to $\mathrm{Leb}(\bbR^2)\otimes\bar\mu_1$ at point $(\ell_1,s,t_2,r)\in\bbR^3\times\bbS$
 equal to
 $$
 \mathtt{(cst)}\, s^{m-4} \exp\left(-\frac{s^2(m-3)}{2}\right)  \sigma^{-1} \phi(\sigma^{-1}\ell_1).
 $$ 
 Using the same method as for the proof of Proposition 
 \ref{prop:Rice_density_known_variance} we have the following proposition.

\begin{proposition}
\label{prop:Rice_density_unknown_variance}
Assume that $Z$ satisfies \eqref{e:Normalization}, \eqref{e:NonDegenerated}, \eqref{a:KLN} and \eqref{a:nonDegeneratefiniteComplex}, and set $m=2N$. Then, under~$\bbH_0$, the joint distribution of  $ \big(\lambda_1, { \widehat{ \sigma}_{|}}/{\sigma},T_2, {R(\widehat z)}/{ \widehat \sigma}_{|} \big)$  has a density with respect to $\mathrm{Leb}(\bbR^2)\otimes \bar \mu_1$ at point $(\ell_1,s,t_2,r)\in\bbR^3\times\bbS^+$
 equal to
 $$
\mathtt{(cst)} \det(- \widetilde \Lambda \ell_1+ \sigma s r)\, s^{m-4} \exp\left(-\frac{s^2(m-3)}{2}\right) \phi(\sigma^{-1} \ell_1) \1_{\{0<\sigma s t_2 < \ell_1\}},
 $$
 where $\mathtt{(cst)}$ is a positive constant that may depend on $m$ and $\sigma$.
\end{proposition}

\noindent
Consequently, we derive the following result.
 \begin{theorem}
 \label{thm:rice_unknown_variance}
 Assume that $Z$ satisfies \eqref{e:Normalization}, \eqref{e:NonDegenerated}, \eqref{a:KLN} and \eqref{a:nonDegeneratefiniteComplex}, and set $m=2N$. For all $r\in\mathbb S^+$, define $ \overline H_r(\cdot)$ as 
  $$
  \forall \ell>0,\quad
  \overline H_r(\ell) := \int_{\ell} ^{+\infty}\! \det\big( - \widetilde \Lambda t_1 + r \big)  f_{m-1}\left(t_1 \sqrt{\frac{m-1}{m-3}} \right)\, \mathrm dt_1,
  $$
  where $f_{m-1}$ is the density of the Student distribution with $m-1$ degrees of freedom. Under the null $\mathds{H}_0$, the test statistic  
 $$
   T^{\mathrm{Rice}} :=\frac{\overline  H_{R(\widehat  z)} (T_1)}{  \overline  H_{R(\widehat z)} (T_2)} \sim \mathcal U([0,1]),
  $$
  where $T_1 := \lambda_1/\widehat\sigma_{|}$, $T_2 = \lambda_2/\widehat\sigma_{|}$ and $\widehat\sigma_{|}$ is defined by \eqref{eq:sighatz2}.

    \end{theorem}
\begin{proof}
First, using Proposition \ref{prop:Rice_density_unknown_variance} and the change of variable $t_1 = \frac{\ell_1}{\sigma s}$, the joint distribution of the quadruplet $(T_1, { \widehat{ \sigma}_{|}}/{\sigma},T_2, {R(\widehat z)}/{ \widehat \sigma}_{|}))$ at point $(t_1,s,t_2,r)$ is given by
\begin{align*}
& \mathtt{(cst)} \det(\sigma s (- \widetilde \Lambda t_1 + r)) s^{m-3} \exp\left(-\frac{s^2(m-3)}{2}\right) \phi(s t_1) \1_{\{0<t_2 <t_1\}} \\
&= \mathtt{(cst)} \det(- \widetilde \Lambda t_1 + r) s^{m-1} \exp\left(-\left(s \sqrt{\frac{m-3}{m-1}}\right)^2 \frac{m-1}{2}\right) \phi(s t_1) \1_{\{0<t_2 <t_1\}}.
\end{align*}
Second, note that if $X$ and $Y$ are two independent random variables of density $f_X$ and $f_Y$ then the density of $X/Y$ satisfies
$$ f_{X/Y}(z) = \int_{\mathbb{R}} f_X(z y) y f_Y(y) \mathrm{dy}.
$$
In our case, integrating over $s$ and with the change of variable $s\leftarrow s\sqrt{(m-1)/(m-3)}$, it holds
\begin{align*}
& \int_{\mathbb{R}^+} \phi(s t_1) s^{m-1} \exp\left[- \Bigg(s\sqrt{\frac{m-3}{m-1}}\Bigg)^2 \frac{m-1}{2}\right] \mathrm{d}s \\
&= \mathtt{(cst)} \int_{\mathbb{R}^+} \phi\left(s t_1 \sqrt{\frac{m-1}{m-3}} \right) s{ s^{m-2} \exp\left[-\frac{s^2 (m-1)}{2}\right]}\mathrm{d}s \\
&= \mathtt{(cst)} \int_{\mathbb{R}^+} \phi\left(s t_1 \sqrt{\frac{m-1}{m-3}} \right) s f_{\frac{\chi_{m-1}}{\sqrt{m-1}}}(s) \mathrm{d}s \\
&= f_{m-1}\left(t_1 \sqrt{\frac{m-1}{m-3}} \right).
\end{align*}
Putting together, the density of $(T_1,T_2,R(\widehat{z})/\hat{\sigma})$ at point $(t_1,t_2,r)$ is now given by
$$ \mathtt{(cst)} \det(-\widetilde \Lambda t_1 + r) f_{m-1}\left( t_1 \sqrt{\frac{m-1}{m-3}} \right) \1_{\{0<t_2 <t_1\}}\,,
$$
and we conclude using the same trick as the one of Theorem \ref{thm:rice_known_variance}.
\end{proof}
  
  
  \section{Applications to the Super-Resolution Theory }

\label{sec:SR}

\subsection{Framework and results}
Deconvolution over the space of complex-valued Radon measure has recently attracted a lot of attention in the ``Super-Resolution'' community\textemdash and its companion formulation in ``Line spectral estimation''. A standard aim is to recover fine scale details of an image from few low frequency measurements\textemdash ideally the observation is given by a low-pass filter. The novelty in this body of work relies on new theoretical guarantees of the $\ell_{1}$-minimization over the space of Radon measures with finite support. Some recent works on this topic can be found in the papers \cite{DeCastro_Gamboa_12,Bredis_Pikkarainen_13,Tang_Bhaskar_Shah_Recht_13,Candes_FernandezGranda_14,Azais_DeCastro_Gamboa_15,FernandezGranda_13,Bendory_Dekel_Feuer_14,Duval_Peyre_JFOCM_15} and references therein. 

An important example throughout this paper is given by the Super-Resolution problem which can be stated as follows. Let $\nu^0\in(\mathcal M([0,2\pi),\mathds C),\|\cdot \|_{1})$ a complex-valued Radon measure on the one dimensional torus identified to $[0,2\pi)$ equipped with the natural circle-wise metric. Note that $||\cdot||_1$ denotes the total variation norm on $\mathcal M([0,2\pi))$. The space $(\mathcal M([0,2\pi),\mathds C),\|\cdot\|_{1})$ can be defined as the topological dual space of continuous functions on $[0,2\pi)$ equipped with the $L^\infty$-norm. 

Let $N=2f_c+1$ where $f_c\geq1$ is referred to as the ‘‘frequency cut-off''. Denote by $\mathbf D_{N}$ the Dirichlet kernel defined by
\[
\forall t\in[0,2\pi),\quad\mathbf D_{N}(t):=\frac{\sin(Nt/2)}{\sin(t/2)}.
\]
Consider the linear operator $\mathcal F_N:\mathcal M([0,2\pi),\mathds C)\to\mathds C^N$ that maps any complex-valued Radon measure~$\nu$ onto its Fourier coefficients $c_k(\nu)$ where
\[
c_k(\nu):=\int_{\bbT}\exp(-\imath kx)\nu(\mathrm dx)
\] 
for integers $k$ such that $|k|\leqslant f_c$. Consider $\zeta = (\zeta_k)_k$ where $\zeta_k=\zeta_{1,k}+\imath \zeta_{2,k}$ and $\zeta_{\ell,k}$ are i.i.d. standard Gaussian random variables for $|k|\leqslant f_c$ and $\ell=1,2$. In the Super-Resolution frame, we observe a perturbed version of the Fourier coefficients, namely
\eq
\notag
y=\frac{1}{\sqrt N}\mathcal F_N(\nu^0)+\sigma\zeta\,.
\qe
Applying $\mathcal F^\star_N$\textemdash the dual operator of $\mathcal F_N$, remark that we observe the trigonometric polynomial 
\[
Z:=\frac{1}{\sqrt N}\mathcal F^\star_N(y)
\] 
which reads as
\eq
\label{eq:CorrZSR}
\forall t\in[0,2\pi),\quad Z(t)=\frac1N\int_{\bbT}\mathbf D_N(t-x)\nu^0(\mathrm dx)+\sigma\sum_{k=-f_c}^{f_c}\frac{1}{\sqrt N}\zeta_k\exp(\imath kt).
\qe
Hence, one observes $Z$ and infers on $\nu^0$ assuming that it has finite support. To this purpose, consider the process $X$ defined for all $(t,\theta)\in\bbT$ by
\eq
\label{eq:X_SR}
X(t,\theta):=\mathrm{Re}(e^{-\imath\theta}Z(t))=\cos(\theta)\,\mathrm{Re}(Z(t))+\sin(\theta)\,\mathrm{Im}(Z(t)),
\qe
where $\mathrm{Re}$ and $\mathrm{Im}$ denote the real and imaginary part of a complex number. When $\nu^0 \equiv 0$, remark that the processes $A_1=\mathrm{Re}(Z)$ and $A_2=\mathrm{Im}(Z)$ are two independent and identically distributed real-valued processes  with $\mathcal C^\infty$-paths. An elementary computation shows that $X$ has correlation function $\rho$ and $A_1$ has correlation function $\Gamma$ with
\begin{align*}
\rho(z-y)&=\cos(\theta-\alpha)\Gamma(t-s)\quad\mbox{where}\quad\Gamma(t-s)={\mathbf D_{N}(t-s)}/N
\end{align*}
for all $z=(t,\theta)$ and $y=(s,\alpha)$ in $\bbT$. Remark that \eqref{e:Normalization} holds true for $\Gamma$. In this case, we are testing 
\eq
\notag
\mathds H_0: \mbox{``}\mathcal F^\star_N(\mathcal F_N(\nu^0))\equiv 0\,\mbox{''} \quad\text{against}\quad  \mathds H_1: \mbox{``}\exists t \in [0,2\pi),~ \mathcal F^\star_N(\mathcal F_N(\nu^0))(t)\neq0\,\mbox{''}\,,
\qe
or equivalently
\eq
\notag
\mathds H_0: \mbox{``}\nu^0\equiv 0\,\mbox{''} \quad\text{against}\quad  \mathds H_1: \mbox{``}\exists t \in [0,2\pi),~ \nu^0(t)\neq0\,\mbox{''}\,.
\qe
Subtracting the known measure $\nu^0$, remark that this framework encompasses testing problem whose null hypothesis is any single hypothesis $\mathds H_0: \mbox{``}\nu^0\equiv\nu_0\,\mbox{''}$ against alternatives of the form $\mathds H_1: \mbox{``}\exists t \in [0,2\pi),~ \nu^0(t)\neq\nu_0(t)\,\mbox{''}$.

Furthermore, we have the following propositions. First, we check that we can apply our results to the Super-Resolution process.
    \begin{proposition} \label{l:jm:l1} 
    The process $X$ defined by \eqref{eq:X_SR} satisfies Condition~\eqref{a:Klm} and Condition~\eqref{a:nonDegeneratem} with $m=2N= 4f_c +2$.
   \end{proposition}
  
   
   \noindent
 Then, we derive a first result when the noise level $\sigma$ si known.  
   \noindent
\begin{proposition}
 \label{cor:rice_known_variance_blasso}
Under the null $\mathds{H}_0$, the test statistic
  $$
 S_{\mathrm{SR}}^{\mathrm{Rice}} =\frac{\sigma (\alpha_1 \lambda_1 + \alpha_2) \phi(\lambda_1/\sigma) + (\alpha_1 \sigma^2- \alpha_3^2) \overline{\Phi}(\lambda_1/\sigma)  }{\sigma (\alpha_1 \lambda_2 + \alpha_2) \phi(\lambda_2/\sigma) + (\alpha_1 \sigma^2- \alpha_3^2) \overline{\Phi}(\lambda_2/\sigma)} \sim \mathcal U([0,1]),
  $$
  where  ${\Phi}$ is the standard Gaussian cumulative distribution function , $ \overline{\Phi} = 1-\Phi$ its survival function, $\phi$ its density function, $(\lambda_1,\lambda_2)$ is defined by $($\eqref{def:hatz}, \eqref{eq:second_knot}$)$ and
  $$
  \left\{ \begin{array}{l}
  \alpha_1 = \frac{1}{3} f_c (f_c + 1), \\
  \alpha_2 = \frac{1}{\sqrt{N}} \sum\limits_{k=-f_c}^{f_c} (k^2 - \alpha_1) \times \mathrm{Re}(y_{k} e^{\imath (k \hat{t} - \hat{\theta})}),  \\
  \alpha_3 = \frac{1}{\sqrt{N}} \sum\limits_{k=-f_c}^{f_c} k \times \mathrm{Re}(y_{k} e^{\imath (k \hat{t} - \hat{\theta})}).
  \end{array} \right.
  $$
\end{proposition}

Finally, we have the following result when the noise level $\sigma$ is unknown.
\begin{proposition}
\label{cor:rice_unknown_variance_blasso}
Under the null $\mathds{H}_0$, the test statistic  
  $$
 T_{\mathrm{SR}}^{\mathrm{Rice}} =\frac{\alpha_1 \overline{F}_{m-3}(T_1) + (\alpha_1 T_1 + \alpha_2) f_{m-3}(T_1) - \gamma_m^{-1} \alpha_3^2 \overline{F}_{m-1}(T_1)  }{\alpha_1 \overline{F}_{m-3}(T_2) + (\alpha_1 T_2 + \alpha_2) f_{m-3}(T_2) - \gamma_m^{-1} \alpha_3^2 \overline{F}_{m-1}(T_2)} \sim \mathcal U([0,1]),
  $$
where  $F_{d}$ is the Student cumulative distribution function with $d$ degrees of freedom, $ \overline  F _{d}  = 1-F_{d} $ its survival function, $f_d$ its density function, $T_1 = \lambda_1/\widehat{\sigma}_{|}$, $T_2 = \lambda_2/\widehat{\sigma}_{|}$, $\widehat{\sigma}_{|}$ is defined by \eqref{eq:sighatz2} and $ \gamma_m = \frac{m-3}{m-2} \frac{\bar{\Gamma}\left(\frac{m}{2}\right)\bar{\Gamma}\left(\frac{m-3}{2}\right)}{\bar{\Gamma}\left(\frac{m-1}{2}\right)\bar{\Gamma}\left(\frac{m-2}{2}\right)} $.

\end{proposition}
 
 \noindent   
 A proof of these propositions can be found in Appendix~\ref{proof:PropSR}.

\subsection{A numerical study}
\label{sec:num}

A Python code (and a Jupyter notebook) illustrating the following numerical experiments can be found at: \textit{https://github.com/ydecastro/super-resolution-testing}.

\subsubsection{Computation of $\lambda_2$} \label{s:l2}
To build our test statistic $S^{\mathrm{Rice}}$ in the Super-Resolution context (namely $S_{\mathrm{SR}}^{\mathrm{Rice}}$), we need to compute three quantities. The first one is $\lambda_1$, the maximum of~$X(\cdot)$ over the torus $\bbT$. Its simple form allow us to use classical optimization routines, for instance \textbf{scipy.optimize.minimize} on \textbf{Python}, \textbf{fminsearch} on \textbf{MATLAB} or \textbf{optim} on~\textbf{R} both combined with global resolution options on $\bbT$. The second one is $R = R(\hat{z})$ which appears in the test statistic through the coefficients $\alpha_1,\alpha_2$ and $\alpha_3$ that are simple functions of the observation~$y$ and $\hat z$. 
Finally, the third one is
$$
\lambda_2 = \lambda_2^{\hat{z}} = \lambda_1 + \max _{y \in \bbT} \left\{\frac{X(y) - X(\hat{z})}{1 - \rho(\hat{z}-y)} \right\}.
$$
Contrary to $\lambda_1$, there is some indetermination problem when $y$ is close to $\hat{z}$. In particular, the approximation of $\hat{z}$ is by definition not exact and the radial limits of $X^{|}$ are not numerically achieved. A way to get around that is the integral form of the remainder in Taylor's theorem. In full generality, we compute
$$
\lambda_2^{\hat{z}} - \lambda_1 = \max_{y \in \bbT} \left\{\frac{\displaystyle\int_0^1 (1-h) ~ (y-\hat{z})^T  X''(\hat{z} + h(y-\hat{z})) (y-\hat{z}) ~ \mathrm{d}h}{\displaystyle\int_0^1 (1-h) ~ (y-\hat{z})^T  \rho''(\hat{z} + h(y-\hat{z})) (y-\hat{z}) ~ \mathrm{d}h} \right\}\,.
$$
Denote by $r = ||y - \hat{z} ||_2$ the distance between $y$ and $ \hat{z}$. The numerical indetermination occurs for small values of $r$. But remark that one can factorize $r^2$ in both the numerator and the denominator. This leads to the expression
$$
\lambda_2^{\hat{z}} - \lambda_1 = \max_{y \in \bbT} \left\{\frac{\displaystyle\int_0^1 (1-h) ~ (\frac{y-\hat{z}}r)^T  X''(\hat{z} + h(y-\hat{z})) (\frac{y-\hat{z}}r) ~ \mathrm{d}h}{\displaystyle\int_0^1 (1-h) ~ (\frac{y-\hat{z}}r)^T  \rho''(\hat{z} + h(y-\hat{z})) (\frac{y-\hat{z}}r) ~ \mathrm{d}h} \right\}\,,
$$
which is more robust in practice. In the Super-Resolution case, elementary trigonometry identities give the following simpler form of the denominator
$$
 \sum\limits_{k=-f_c}^{f_c} \left( k \cos(\alpha) - \sin(\alpha) \right)^2 \times \mathrm{sinc}\left( \frac{r(k \cos(\alpha) - \sin(\alpha))}{2} \right)^2
$$
where $y-\hat{z} = (r \cos(\alpha), r \sin(\alpha))$ and \textbf{sinc} denote the cardinal sine function, i.e.
$$
\mathrm{sinc}(x) = \left\{  \begin{array}{cl}
\displaystyle\frac{\sin x}{x}&\text{if }x \neq 0, \\
&\\
1 &\text{if } x=0,
\end{array} \right.
$$ 
which is a numerically robust function. We conclude the optimization using the same routine as the one of $\lambda_1$.

\subsubsection{Monte-Carlo experiment}

In this section we compare  the cumulative distribution  of several statistics of test in the case where the variance is known, namely
\begin{itemize}
\item The statistics of the Rice test $ S^{\mathrm{Rice}}  $, given by Theorem \ref{thm:rice_known_variance}, are displayed in blue.
\item The statistics of the Spacing test $ S^{\mathrm{ST}} $, given   by \eqref{e:naive}, are displayed  in  green.
\item The statistics  of the Spacing test on  grids $G_n$  given by $ \overline \Phi(\lambda_{1,n}) / \overline\Phi(\lambda_{2,n}) $ are displayed with a color  that take the respective values green, red, purple and cyan for sizes equal to $ 3^2,10^2,32^2,50^2$. 
 \item The grid test, based on $S^{\mathrm{Grid}}$ of Theorem \ref{thm:grid_known_variance}
  can be viewed as the limit of the  discrete grid tests above  as the size  growths to infinity. As one can see in the figures, there is some evidence  that this limit is numerically reached for a size $n=50^2$. 
 \end{itemize}
We complete each graph by the diagonal to the cumulative distribution function of the uniform law on $[0,1]$ displayed in black. All the figures are based on $2000$ simulations of the corresponding statistics.
  
 The first figure studies the distribution of $S^{Rice}$ and $S^{ST}$ under the  Null. This figure is displayed in the introduction (see Figure~\ref{unbiaised_naive}). The second figure deals with the grid statistic and $S^{Rice}$ under various alternatives defined by a single spike and compares the power of the Rice test with  the discrete grid tests, see Figure~\ref{Compa_1mean}. Finally, the third figure  performs  the same study but with an alternative defined by two atoms, see Figure~\ref{Compa_2mean}.

\begin{figure}[!h]
\includegraphics[width=4.5cm,height=4.5cm]{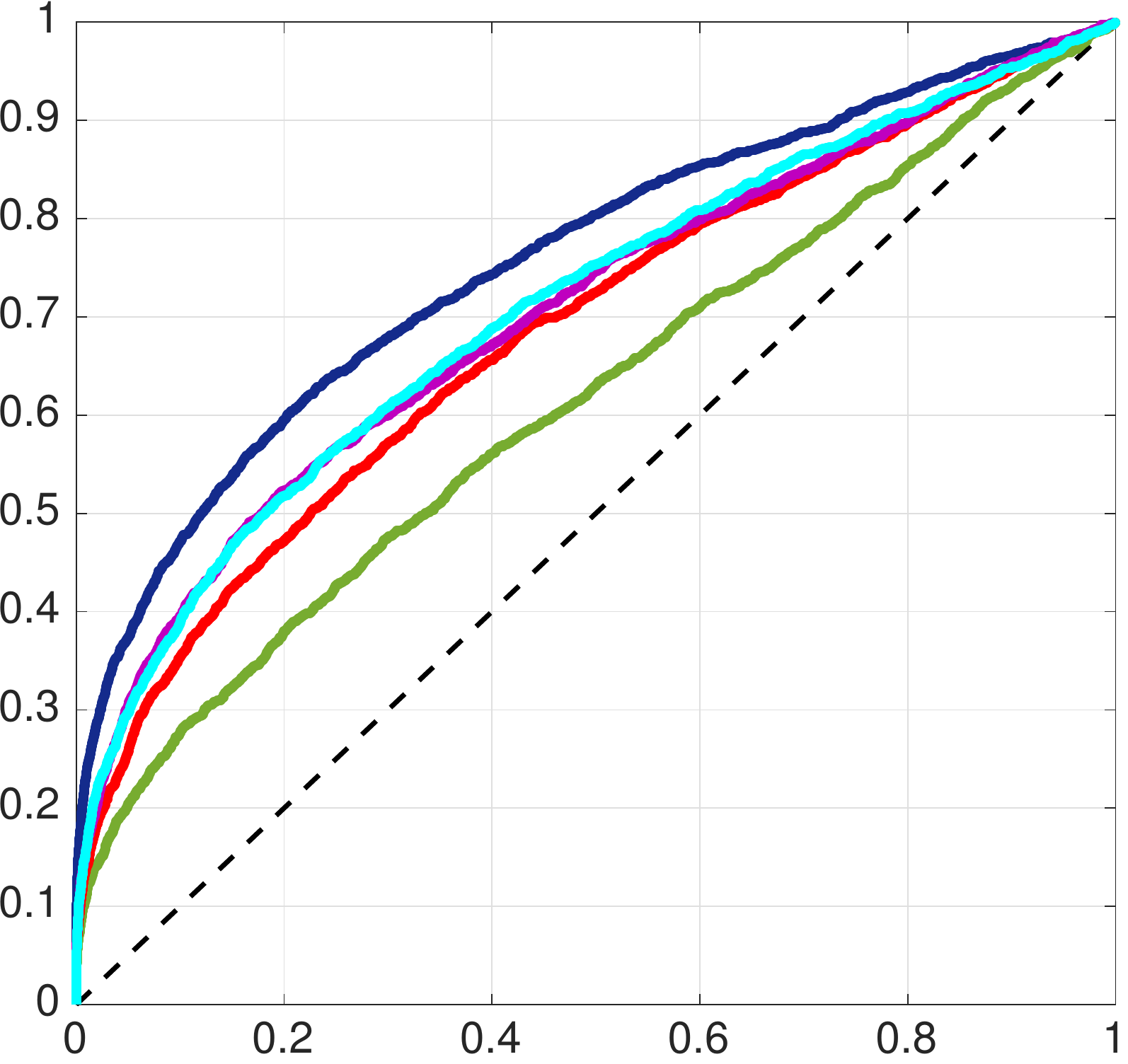}
\ \ \ \ \ 
\includegraphics[width=4.5cm,height=4.5cm]{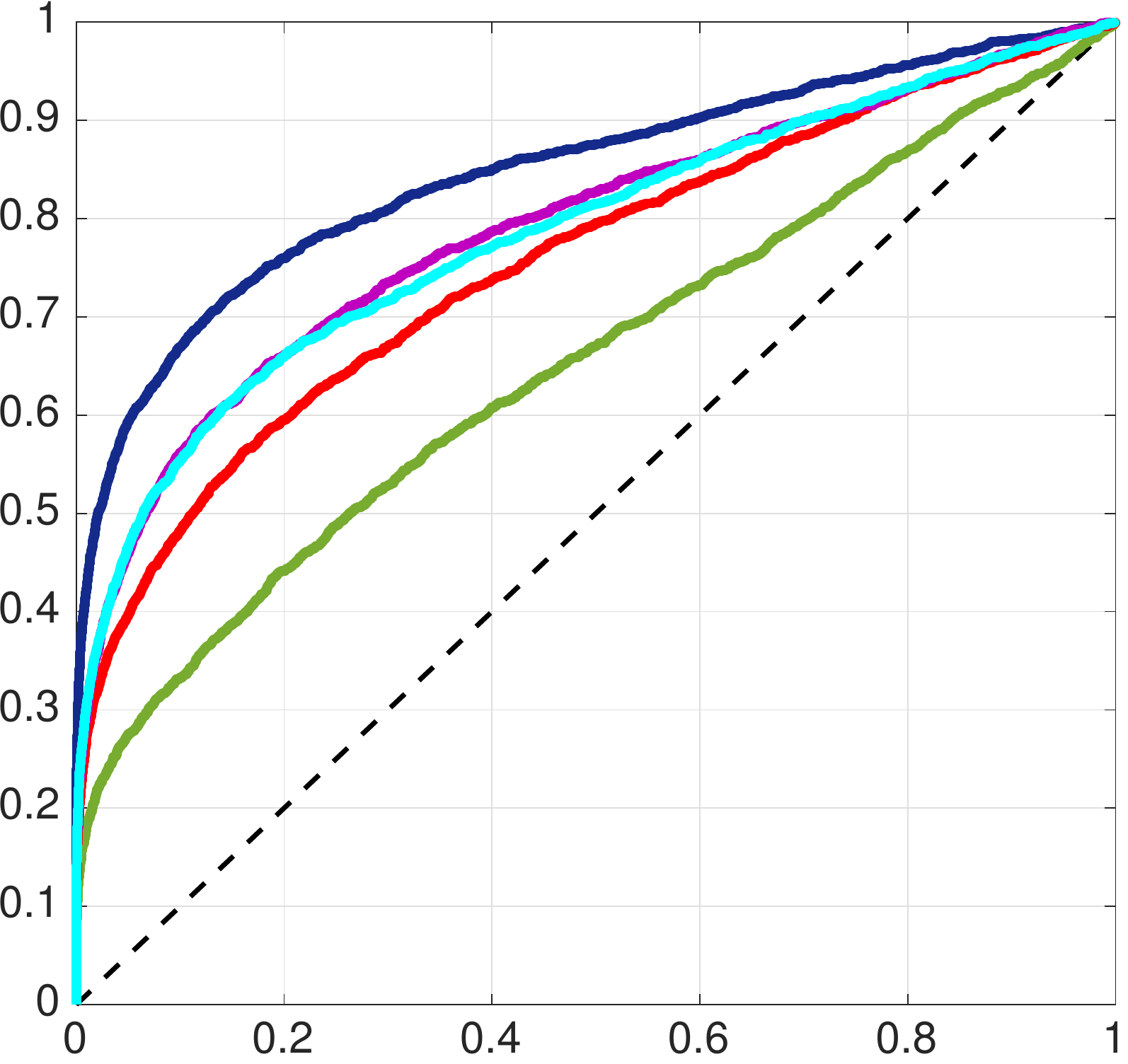}
\ \ \ \ \ 
\includegraphics[width=4.5cm,height=4.5cm]{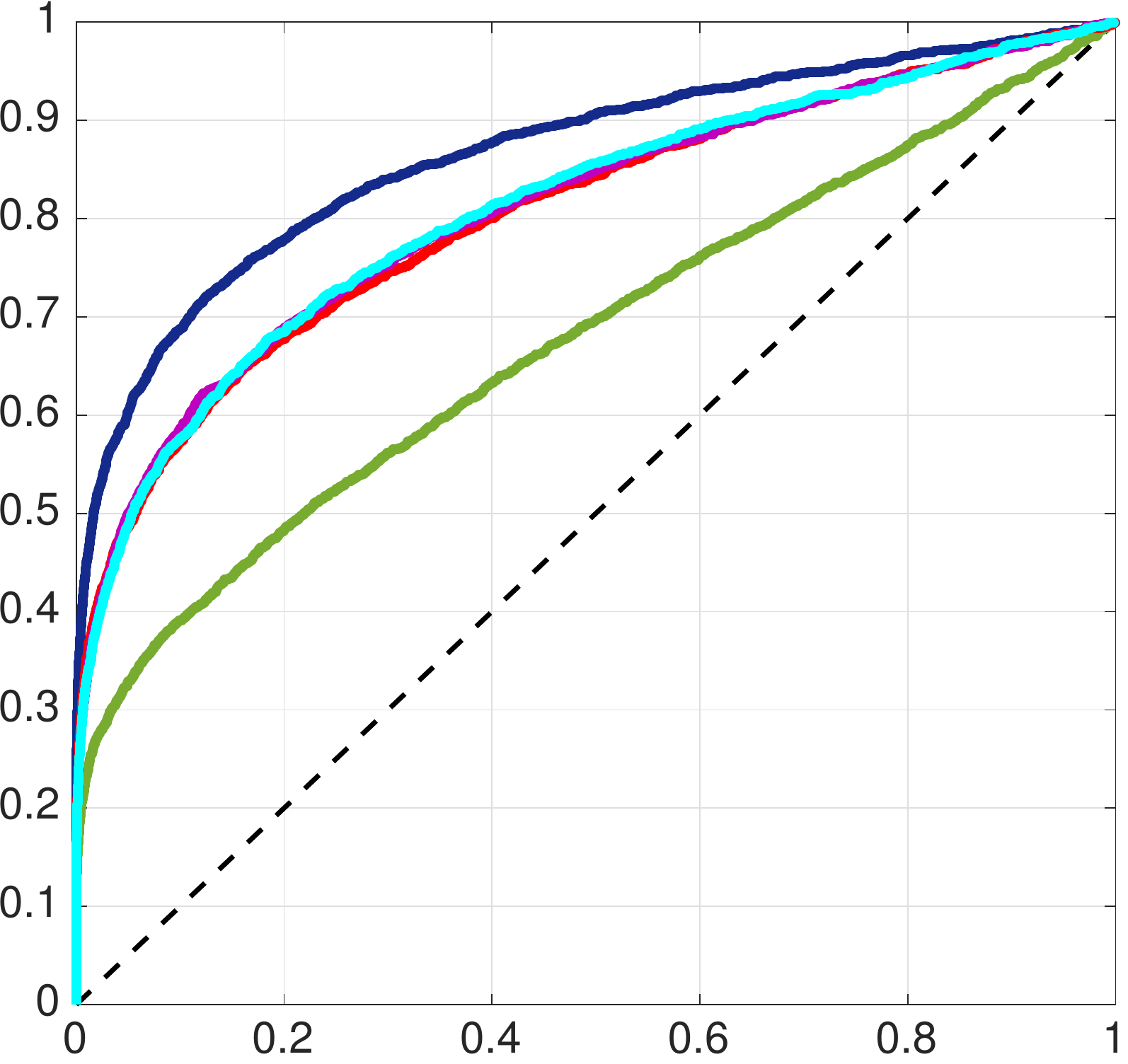}
\caption{ Same as Figure 2 except that  (a) $f_c=7$ , (b) the alternative is defined by two spikes at random locations (with a constraint of separation)  (c) the weights are now  from left to right $ (\log N,\log N)$; $ (\log N,\sqrt N)$; $ (\sqrt N ,\sqrt N)$.}
\label{Compa_2mean}
\end{figure}

A last set of experiments is devoted to the computation of the testing procedure when the noise level is unknown, see Figure~\ref{fig:student}. 

\begin{figure}[!h]
\includegraphics[width=0.48\textwidth]{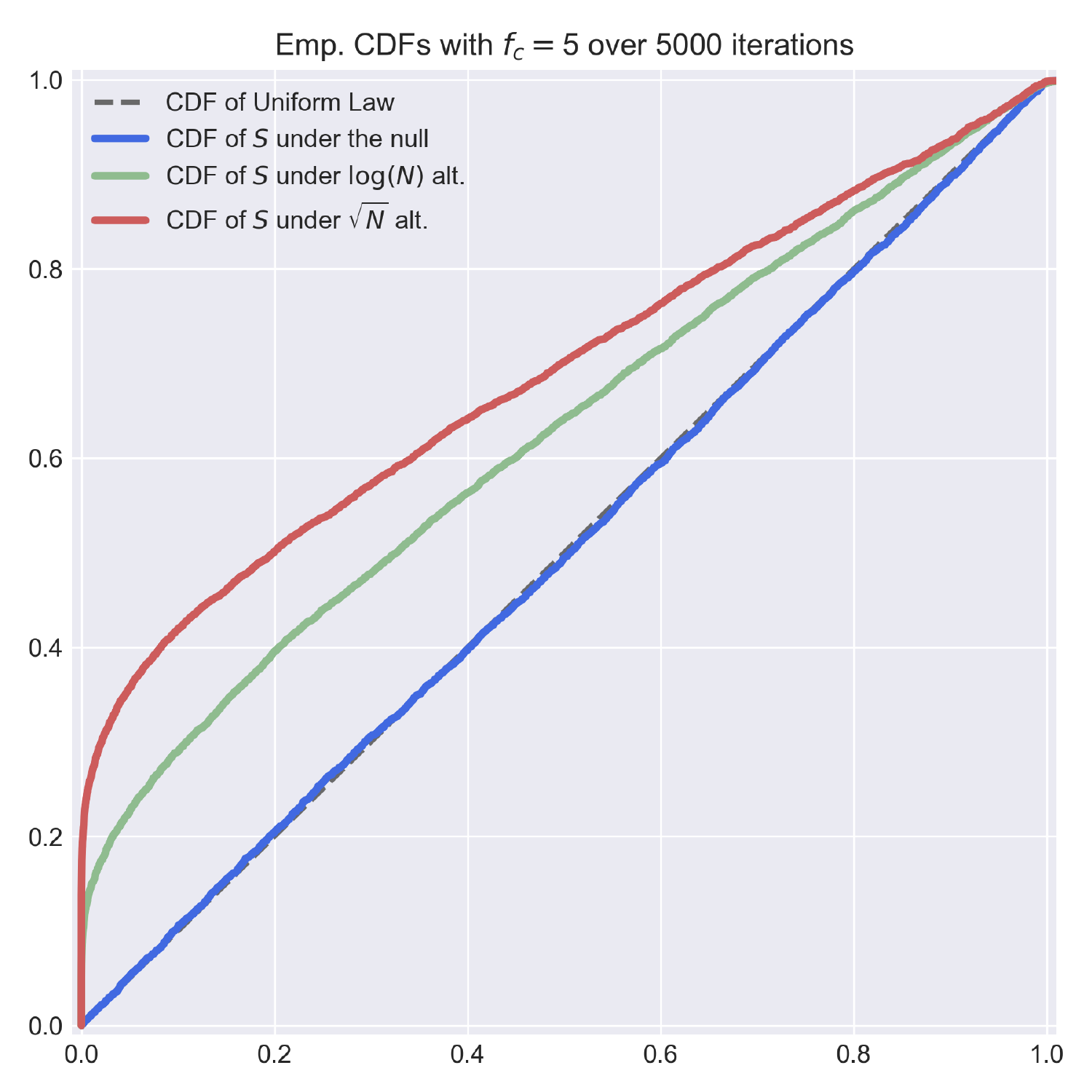}
\ \ \ \ \ 
\includegraphics[width=0.48\textwidth]{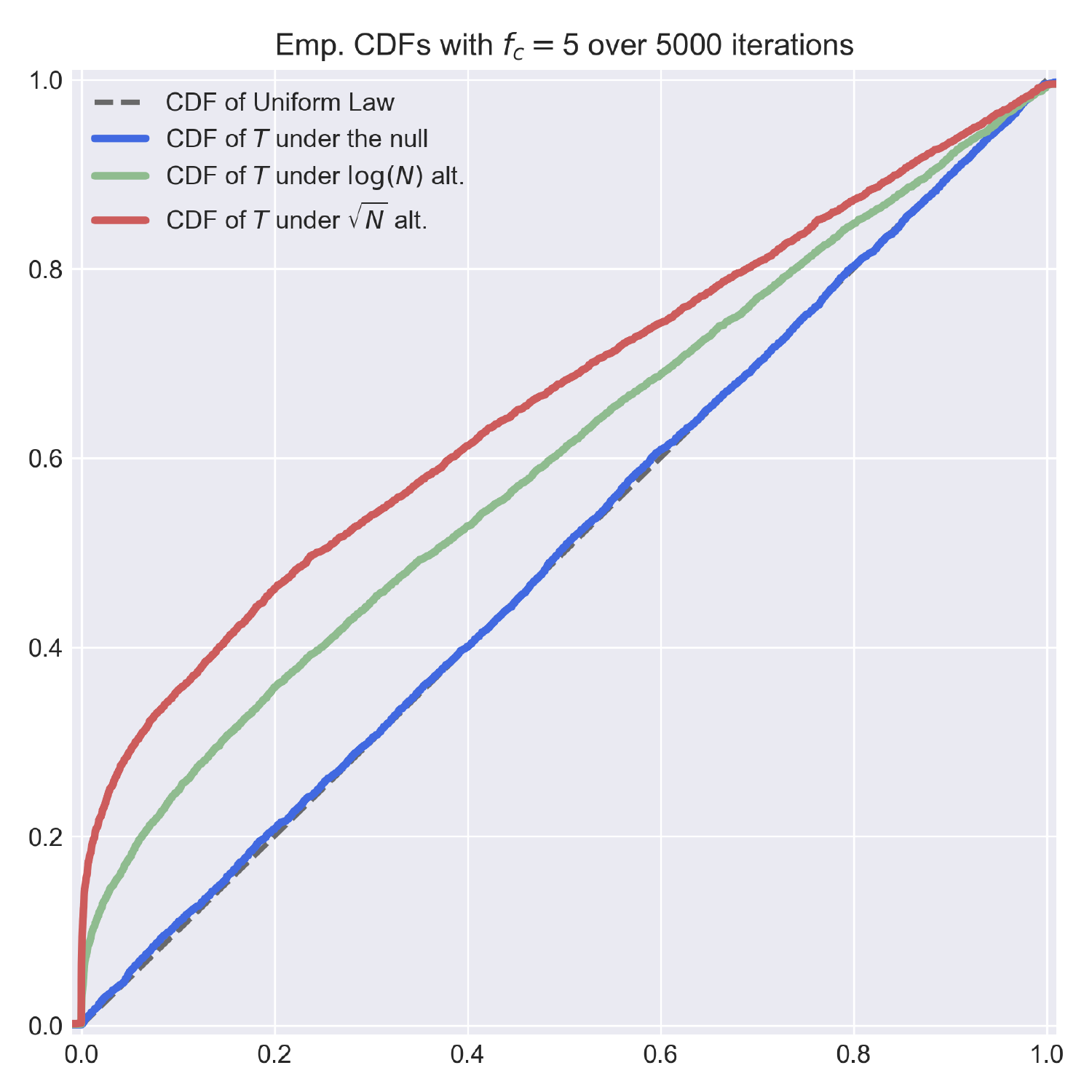}
\caption{ We compute three empirical cumulative distribution functions: under the null (dashed gray line), under one spike alternative of size $\log N$ (green line) and under one spike alternative of size $\sqrt N$. The left panel uses the statistic $S$ (known variance case) and the right panel the statistic $T$ (unknown variance case). We witness a slight loss of power in this later case.}
\label{fig:student}
\end{figure}

These latter numerical experiments were conducted using a Python code. The notebook \textbf{testing\_super\_resolution.ipynb} available at \textit{github.com/ydecastro/super-resolution-testing} allows to reproduce these experiments.
 \subsection*{Discussion}
 
  Figure \ref{unbiaised_naive}  suggests  that the Spacing test  is highly non-conservative  which is a major  drawback. For instance, when $f_c=7$, the empirical level  of the Spacing test  at a nominal  level  of 5\% is in fact 11,3\%, showing that this  test is very non-conservative. For its part, the Rice test is exact as predicted by the theory. This numerical agreement prove that the numerical algorithm described in Section~\ref{s:l2} is efficient.
  
   In Figure \ref{Compa_1mean} and \ref{Compa_2mean}  we see that  the power of the discrete grid tests  may seem  an increasing  function of the number of  points of the grid.   This power seems to converge since the curves associated  to  $32^2$ (purple)  and $ 50^2$ (cyan) are almost indistinguishable.  This suggests that the Rice  test (blue)  is always more powerful than the discrete  grid test or the limit grid test. Consequently, it seems unbiaised for any choice of alternative.
   
    In conclusion the Rice test seems to be the best choice  even if we are still not able to prove theoretically that it is unbiased.

%
\section*{Acknowledgement}
The authors would like to thank the referees for their useful comments and interesting remarks that have improved the presentation of this paper. 
\appendix

\section{Proofs}
We denote for random variables, $X_n=o_P(r_n)$ and $Y_n=O_P(r_n)$ (for $r_n\neq0$) means that $r_n^{-1}||X_n||$ converges to $0$ in probability and $r_n^{-1}||Y_n||$ is uniformly tight, respectively. 
Furthermore, we consider the following processes.
\begin{itemize}
\item The stationary process $X(z) = X(t,\theta)$ defined on $\bbT$ with covariance function given by $\Cov(X(y),X(z)) = \sigma^2\rho(z-y)$ where we recall the correlation function is given by $\rho(z-y)=\cos(\theta-\alpha)\Gamma(t-s)$,
\item For every $z \in \bbT$, recall the regressions with respect to $X(z)$
	     \begin{align*}
\forall y\in\bbT\setminus\{z\},\quad & X^{z}(y)=\frac{X(y)-X(z)\rho( z-y)}{1-\rho(z-y)}=X(z)+\frac{X(y)-X(z)}{1-\rho(z-y)}\,, \\
	    	    & X_{\mathrm{norm}}^z(y) = \frac{X(y) - \rho(z - y) X(z)}{\sqrt{1 - \rho^2(z-y)}}\,.\notag
	    	    \end{align*}
\item For every $z \in \bbT$, recall the regressions with respect to $(X(z),X'(z))$
\begin{align*}
\forall y\in\bbT\setminus\{z\},\quad 
 & X^{|z}(y) =\frac{ X(y) -\rho(z-y) X(z) +  \langle \rho'(z-y) ,\widetilde \Lambda^{-1} X'(z) \rangle }{1-\rho(z-y) }\,,\\
  & X_{\mathrm{norm}}^{|z}(y) = \frac{X(y) - \rho(z - y) X(z) + \langle \rho'(z-y) , \widetilde \Lambda^{-1} X'(z) \rangle}{\sqrt{1 - \rho^2(z-y) +  \langle \rho'(z-y), \widetilde \Lambda^{-1}\rho'(z-y) \rangle}}\,.
\end{align*}
\end{itemize}
\noindent
In particular, recall that $\widehat z$ is defined by~\eqref{def:hatz} so $X'(\widehat z)=0$ and it yields that $X^{\widehat z}=X^{|\widehat z}$.


\subsection{Proof of Theorem~\ref{thm:grid_known_variance}}
\label{proof:grid_known_variance}
Since the variance is known, we consider without loss of generality that~$\sigma^2 = 1$. Using the metric given by the quadratic form represented by $X''$, we can consider the closest point $\overline z_n$ of the grid~$G_n$ to $\widehat z$ by
 \[
 \overline z_n = \argmin_{u \in G_n} \| \widehat z - u\|_{X''} = \argmin_{u \in G_n} ~ \langle \widehat z - u, -X''(\widehat z) (\widehat z - u) \rangle\,.
 \]
 The main claim is that, while it holds $\lambda_{1,n} \to  \lambda_{1}$ a.s., we don't have the same result  for $\lambda_{2,n}$,  see Lemma~\ref{lem:z2}. We begin with the following preliminary result, which is related to the result of Aza\"is-Chassan \cite{AC2017}.

\begin{lemma}  \label{lem:z1}
Under~$\bbH_0$ and conditionally to $X''$, $\Delta_n^{-1}( \widehat z- \overline{z}_n)$ follows a uniform distribution on $V_{\mathrm 0}$ and this distribution is independent from $\lambda_1$ and $\lambda_2$. 
\end{lemma} 
\begin{proof} 
Remark that $\widehat z$ has uniform distribution on $\bbT$ by stationarity and this distribution is independent from $\lambda_1$ and $\lambda_2$. Let $B$ be a Borelian in $\bbR^2$. Remark that $\widehat z- \overline{z}_n\in\Delta_nV_\mathrm 0$ by definition of~$\overline{z}_n$ and note that $\overline{z}_n\in\Delta_n\bbZ^2$. Conditionally to~$X''$, it holds
 %
 \begin{align*}
 \P\{ \Delta_n^{-1}(\widehat z- \overline{z}_n) \in B \} &=\P\{ \Delta_n^{-1}(\widehat z- \overline{z}_n) \in B\cap V_\mathrm 0 \}\\
 &=\P\{ (\widehat z- \overline{z}_n) \in \Delta_n(B\cap V_\mathrm 0) \}\\
  &=\P\{ \widehat z \in \Delta_n(B\cap V_\mathrm 0+\bbZ^2) \}\,.
 \end{align*}
 Since $\widehat z$ has uniform distribution on $\bbT$ and since $V_\mathrm 0+\bbZ^2$ is a partition of $\bbR^2$, it holds that 
 \begin{align*}
\P\{ \widehat z \in \Delta_n(B\cap V_\mathrm 0+\bbZ^2) \}=\P\{ \widehat z \in 2\pi (B\cap V_\mathrm 0+\bbZ^2) \}=\frac{\mathrm{Leb}(\bbR^2)(2\pi (B\cap V_\mathrm 0))}{\mathrm{Leb}(\bbR^2)(2\pi  V_\mathrm 0)}\,,
 \end{align*}
where $\mathrm{Leb}(\bbR^2)$ denotes the Lebesgue measure on $\bbR^2$.
  \end{proof}
    
  %
  \begin{lemma} 
  \label{lem:hat}
 Under~$\bbH_0$, it holds that 
\begin{itemize}
\item [$(a)$]  $ X(\widehat{z}_n) - X(\overline{z}_n) = o_P(\Delta_n^2)$.
\item[$(b)$]  $ \P\{\widehat{z}_n\neq \overline{z}_n\}  \to 0$ as $n$ goes to $\infty$.
\item[$(c)$]  Let $F$  be any measurable function, then $F(\widehat z_n) - F(\overline z_n)$ tends to zero in probability at arbitrary speed.
\item[$(d)$]  Almost surely, one has $\overline{z}_n\to\widehat z$ and $\widehat{z}_n\to\widehat z$ as $n$ goes to infinity.
\end{itemize}
\end{lemma}
  \begin{proof} 
  Let $\varepsilon>0$.  By definition of $\overline z_n$ and since $V_\mathrm 0\subset[-1,1]^2$, it holds that 
  \eq
  \label{e:asBarZ}
  ||\widehat z-\overline z_n||\leqslant \sqrt 2\Delta_n\,,
  \qe 
  almost surely. Since $X$ has $\mathcal C^2$-paths and by Taylor expansion, one has
\eq
\label{e:3}
X(\widehat z)-X(\overline{z}_n)=(1/2)||\widehat z-\overline z_n||^2_{X''}+o_P(\Delta_n^2)
\qe
Since $-X''$ is positive definite, there exists $M>0$ sufficiently large such that
\[
(1/M)\mathrm{Id_2}\preccurlyeq  -X'' \preccurlyeq M\mathrm{Id_2}
\]
where $\preccurlyeq $ denotes the Lowner ordering between symmetric matrices. Then, it holds 
\eq
\label{e:31}
\forall z\in\bbR^2,\quad (1/M)||z||^2\leq||z||^2_{X''}\leqslant M||z||^2\,.
\qe
From \eqref{e:asBarZ}, \eqref{e:3} and \eqref{e:31}, we deduce that 
\eq
\label {e:32}
0\leqslant  X(\widehat z) -X(\widehat z_n) \leqslant X(\widehat z)-X(\overline{z}_n)
=O_P(\Delta_n^2)\,,
\qe
using the optimality of $\widehat z$ and $\widehat z_n$.

By compactness of $\bbT$, uniqueness of optimum $\widehat z\in\bbR^2$ and $\mathcal C^2$-continuity of~$X$, there exists~$\eta>0$ and a neighborhood $N_0\subset \bbR^2$ of $\widehat z\in\bbR^2$ such that $X(\widehat z)-\eta\geqslant X(z)$ for any $z\notin N_0$ and 
\eq
\label {e:33}
\forall z\in N_0,\quad (1/4) ||\widehat z-z||^2_{X''}\leqslant X(\widehat z)-X(z)\leqslant ||\widehat z-z||^2_{X''}
\qe
using again a Taylor expansion as in \eqref{e:3}. Using~\eqref{e:32}, it holds that, on an event of probability at least $1-\varepsilon/4$ and for $n$ large enough, $0\leqslant  X(\widehat z) -X(\widehat z_n) \leq\eta/2$ implying that $\widehat z_n\in N_0$. Invoke \eqref{e:31}, \eqref{e:32} and~\eqref{e:33} to deduce that $\widehat z-\widehat z_n =O_P(\Delta_n)$.

Using Taylor formula again, we get that
   \begin{align}
 X(\widehat{z})-  X(\widehat{z}_n) &=  (1/2) ||\widehat{z} -\widehat{z}_n||^2_{X''}  + o_P(\Delta_n^2)\,.  \label{e:2}
     \end{align}
 By optimality of $\widehat{z}_n$ and $\overline{z}_n$ and using \eqref{e:3} and \eqref{e:2}, one gets
 \eq
 \label{e:4}
 0\leqslant X(\widehat{z}_n)-X(\overline{z}_n)\leq(1/2)(||\widehat z-\overline z_n||^2_{X''}-||\widehat z-\widehat z_n||^2_{X''})+o_P(\Delta_n^2)\,.
 \qe
 Observing that $||\widehat z-\overline z_n||^2_{X''}-||\widehat z-\widehat z_n||^2_{X''}\leq0$, we get $(a)$.
  
Conditionally to $X''$ and in the metric defined by $||\cdot||_{X''}$, there exists $\eta'>0$, such that the $\eta'$-neighborhood, denoted by $N_{\eta'}$, of the boundary $\partial V_\mathrm 0$ of $V_\mathrm 0$ has relative volume (for the Lebesgue measure) less than~$\varepsilon/8$.  More precisely, $N_{\eta'}$ denotes the set of points in $V_{\mathrm 0}\subset\bbR^2$ with $||\cdot||_{X''}$-distance less than $\eta'$ to the boundary of $V_\mathrm 0$. In particular, 
\[
\forall k\in\bbZ^2\setminus \{0\},\ \forall z\in V_\mathrm 0\setminus N_{\eta'},\quad ||z||_{X''}+\eta'\leqslant ||z-k||_{X''}\,,
\]
by Cauchy-Schwarz inequality. Using Lemma~\ref{lem:z1} and by homogeneity, we deduce that it holds  
\[
\forall g\in\Delta_n\bbZ^2\setminus \{0\},\quad ||\widehat z-\overline z_n||_{X''}+\eta'\Delta_n\leqslant ||\widehat z-g||_{X''}\,,
\]
with probability at least $1-\varepsilon/8$. It follows that
\eq
\notag
\forall g\in\Delta_n\bbZ^2\setminus \{0\},\quad ||\widehat z-\overline z_n||_{X''}^2+(\eta')^2\Delta_n^2\leqslant ||\widehat z-g||_{X''}^2\,,
\qe
using that $(a+b)^2\geqslant a^2+b^2$ for $a,b\geqslant 0$. Now, invoke \eqref{e:4} to get that
\[
0\leqslant -\frac{(\eta')^2}2\Delta_n^2\ \mathbf 1_{\{\overline z_n\neq\widehat z_n\}}+ o_P(\Delta_n^2)\,.
\]
On these events, we get that, for $n$ sufficiently large, $\widehat z_n$ and $\overline z_n$ must be equal except on an event of probability at most $\varepsilon/4+\varepsilon/8\leq\varepsilon$. Furthermore, this result holds unconditionally in~$X''$. We deduce that $\limsup  \P\{ \overline{z}_n  \neq \widehat{z}_n \} \leqslant \varepsilon$, proving~$(b)$. Note that  $(c)$ is a consequence of the fact that, for $n$ sufficiently large, $\widehat z_n$ and $\overline z_n$ must be equal except on an event of arbitrarily small size. In particular, it shows that $\sup_{k\geqslant n}||\overline z_k-\widehat z_k||$ converges towards zero in probability, which is equivalent to almost sure converge of $\overline z_n-\widehat z_n$ towards zero. Claim $(d)$ follows when remarking that~\eqref{e:asBarZ} proves a.s. convergence of $\overline z_n$ towards~$\widehat z$.
  \end{proof}
  
  \pagebreak[3]
  
  \begin{lemma}  \label{lem:z2}
As $n$ tends to infinity, $\lambda_{2,n}$ converges in distribution to $  \overline \lambda_2$.
  \end{lemma}
  
\begin{proof}
Let $\beta\in\bbR$ be such that $0<\beta< 1/2$, say $\beta=1/4$. Let $\varepsilon\in(0,1)$. We can write $ \lambda_{2,n}= {\lambda_{A,n}\vee\lambda_{B,n}}$ with
    \begin{align*}
    \lambda_{A,n} & :=  \max_{u \in G_n\setminus\{\widehat z_n\}\ \mathrm{s.t.}\ \| u-\widehat z_n\| \leqslant  \Delta_n^\beta}  X^{\widehat z_n}(u) =: \max_{u \in G_{n,A} }X^{\widehat z_n}(u)  ,
    \\
  \lambda_{B,n} &: =\max_{u \in G_n\ \mathrm{s.t.}\  \| u-\widehat z_n\| > \Delta_n^\beta}  X^{\widehat z_n}(u)  =: \max_{u \in G_{n,B}} X^{\widehat z_n}(u) .
    \end{align*}
We first prove  that  $ \lambda_{B,n} \to \lambda_2$ as $n$ tends to infinity in distribution. By compactness, remark that there exists a constant $C_r>0$ such that
   \begin{align}
    1-\rho(u)  & \geqslant  C_r \|u\|^2. \label{f:1}
    \end{align}
 It also holds that
    \begin{align}
     X^{\widehat z_n} (u)  &=  X(\widehat z_n) +  \frac{X(u) - X(\widehat z_n)}{ 1-\rho(u-\widehat z_n)}, \label{f:2}
    \\
      X^{\widehat z} (u) &=  X(\widehat z) +  \frac{X(u) - X(\widehat z)}{ 1-\rho(u-\widehat z)}.  \label{f:3}
    \end{align}
    Let us look to  the rhs of  \eqref{f:2} and  \eqref{f:3}. By Claim $(d)$ of Lemma~\ref{lem:hat} and the continuous mapping theorem, note that $X(\widehat z_n)$ converges toward  $\lambda_1=X(\widehat z)$ a.s. and we can omit these terms. It remains  to prove  that  on $G_{n,B}$  the second terms are equivalent. Because of Lemma~\ref{lem:hat},  $1-\rho(u-\widehat z_n)$ converges  to $ 1-\rho(u-\overline z_n)$  at arbitrary speed. Remember that \eqref{e:asBarZ} gives $ \widehat z - \overline z_n = O_P( \Delta_n) $   and it holds that 
    $||u- \widehat z|| > \Delta_n^\beta,$ on  $G_{n,B}$. It follows that there exists $C>0$ such that
    \eq \label{e:den}
    1-\rho(u-\widehat z_n) \geqslant  C \Delta_n^{2\beta}\quad\mathrm{and}\quad 1-\rho(u-\widehat z) \geqslant  C \Delta_n^{2\beta},
    \qe
 with probability greater than $1-\varepsilon/2$. As for the numerators, Eqs. \eqref{e:32} and \eqref{e:33} show that for all $u\in G_{n,B}$ 
 \[
 \Big|\frac{X(u) - X(\widehat z_n)}{X(u) - X(\widehat z)}-1\Big|=\Big|\frac{X(\widehat z) - X(\widehat z_n)}{X(u) - X(\widehat z)}\Big|=(\mathrm{cst})\frac{|X(\widehat z) - X(\widehat z_n)|}{||u - \widehat z||^2}=O_P(\Delta_n^{2-2\beta})
 \]
 In this sense, we say that $X(u) - X(\widehat z_n)$ is uniformly equivalent to $X(u) - X(\widehat z)$ on the grid~$G_{n,B}$ in probability. Using \eqref{e:den} and 
 noticing that for any $u\in\bbT$
 \[
 |\rho(u-\widehat z) - \rho(u-\widehat z_n)|\leqslant ||\rho'||_{\infty}||\widehat z - \widehat z_n||=O_P(\Delta_n)\,,
 \]
 the same result holds for the denominators, namely $1-\rho(u-\widehat z_n)$ is uniformly equivalent to $1-\rho(u-\widehat z)$ on the grid $G_{n,B}$ in probability. We deduce that $ X^{\widehat z_n} (u)$ is uniformly equivalent to $ X^{\widehat z_n} (u)$ on the grid $G_{n,B}$  in probability and, passing tho their maximum, one can deduce that $\lambda_{B,n}$ converges to~$\lambda_2$ in probability.
 
 \medskip
    
    We turn now to  the study of the local part   $ \lambda_{A,n} $.  Again, by Claim~$(c)$ of Lemma~\ref{lem:hat} we can replace~$\widehat z_n$  by~$\overline z_n$ in the numerator of the r.h.s in~\eqref{f:2} and we forget the first term which limit is clearly $\lambda_1$ almost surely. We perform a Taylor expansion at $ \widehat z$, it gives that
   \[
    X(u) -X(\widehat z)  =  (1/2)(u- \widehat z) ^\top X'' (u- \widehat z) (1+ o_P(1)).
   \]
for any $u\in G_{n,A}$. Since $ \overline z_n-\widehat z = O_P(\Delta_n)$, we also get that
\[
   X(\overline z_n)  -X(\widehat z)  =  (1/2)( \overline z_n- \widehat z) ^\top X'' (\overline z_n- \widehat z)   + o_P(\Delta_n^2).
\]     
As for the denominator, invoke  \eqref{e:asBarZ}, \eqref{f:1} and Claim~$(c)$ of Lemma~\ref{lem:hat} to get that 
     \begin{align*}
     1-\rho(u-\overline z_n) &\geqslant 2 C_r \Delta_n^2,\\
     1-\rho(u-\widehat z_n) & =1-\rho(u-\overline z_n) +o_P(\Delta_n^2),
     \\
    \mathrm{and}\quad 
    1-\rho(u-\widehat z_n) &= (1/2)\big((u-\overline z_n)^\top \tilde \Lambda( u-\overline z_n) \big) (1+ o_P(1)),
    \end{align*}
  where $ -\tilde \Lambda$ denotes the Hessian at point $0$ of $\rho$.
  Putting all together yields 
  $$
    \frac{X(u) - X(\widehat  z_n)}{ 1-\rho(u-\widehat z_n)} =  \frac{( u -\overline z_n) ^\top X'' ( u +\overline z_n -  2 \widehat z) }
    { (u-\overline z_n)^\top \tilde \Lambda( u-\overline z_n)  } (1 + o_P(1)),
    $$ 
 for any $u\in G_{n,A}$. Now we know that, in distribution, $ \widehat z - \overline z_n = \Delta_n \mathcal U$ and we know that  $u - \overline z_n = k \Delta_n $  with $k$ belonging to a certain  growing subset of
  $\bbZ ^2$  which limit is  $\bbZ ^2$ . Finally, conditionally to $X''$, we obtain that 
  $$
  \max_{u\in G_{n,A}\setminus\{\widehat z_n\}} X^{\widehat z_n} (u)  \longrightarrow  \lambda_1 + \sup_{k\in \bbZ ^2 \setminus \{0\}}
 \frac{ k^\top}{\| {\tilde \Lambda}^{\frac12}k\|} X'' \frac{(k-2 \mathcal U)}{ \| {\tilde \Lambda}^{\frac12} k\|},
 $$
 in distribution.
  \end{proof}

Eventually, consider the test statistic $ S_n := {\overline \Phi(\lambda_{1,n})}/{  \overline \Phi ( \lambda_{2,n}) } $  and keep in mind  that  $X \big(u +(0,\pi)\big) = -X(u)$ and that if $u$ belongs to $G_n$, $ \big(u +(0,\pi)\big) $  also belongs.  So Theorem~$1$  of~\cite{ADCM16} applies showing that, under the alternative,  $ \bbP \{S_n \leqslant  \alpha\} \geqslant \alpha$. It suffices to pass to the limit to get the desired result.

%
\subsection{Proof of Theorem \ref{t:t}}
\label{proof:t:t}
We use the same grid argument as for the proof of Theorem~\ref{thm:grid_known_variance}.

Let $t_1, t_2, \ldots, t_{N}$ be pairwise distinct points of $[0,2\pi)$, $\theta_1 \in [0,2\pi)$, $m = 2N$ and set
  $$ z_1= (t_1,\theta_1),~~\dots,~~ z_N = (t_N,\theta_1),~~z_{N+1} = ( t_{1}, \theta_1 + \pi/2), ~~\dots, ~~z_{m} = (t_N, \theta_1 + \pi/2).$$
  
  Because of the first assumption of \eqref{a:nonDegeneratefiniteComplex}, the distribution of $(X(z_1),\dots,X(z_m))$ is non degenerated. Consequently, following the proof of Proposition \ref{prop:am}, we know that $X^{z_1}$  satisfies~$\mathrm{KL(m-1)}$ and $\mathrm{ND(m-1)}$. Denote $g_1,\ldots,g_{m-1}$ the eigenfunctions of the Karhunen-Lo\`eve~(KL) representation of $X^{(0,0)}$. Note that $X^{z_1}(\cdot)$ has the same distribution as $X^{(0,0)}(.-z_1)$ (stationarity) and that both are defined on the same space so the KL-eigenfunctions of $X^{z_1}$ are $g_1(.-z_1),\ldots,g_{m-1}(.-z_1)$.
   
    Now consider $A^{z_1} = (A^{z_1}_{i,j})_{1\leq i,j\leq m-1}$ the matrix with entries $A_{i,j}^{z_1} = g_{i}(z_{j+1} - z_1)$ which is invertible thanks to $\mathrm{KL(m-1)}$ and $\mathrm{ND(m-1)}$ and build so that 
    $$
    \left(\begin{array}{c} X^{z_1}(z_2)\\ \vdots  \\  X^{z_1}(z_{m})\end{array}\right) = A^{z_1} \left(\begin{array}{c} \zeta_1 \\ \vdots  \\  \zeta_{m-1} \end{array}\right).$$
    One possible explicit expression, among many others, of $ \widehat\sigma^2_{\mathrm{KL}} ( X_{ \rm{norm}}^{\hat{z}_n}(G_n)) $, the estimator of $\sigma^2$ on the grid $G_n$, is 
    \beq 
   \notag
     \widehat  \sigma^2_n := \widehat{\sigma}^2_{\mathrm{KL}} ( X_{\rm{norm}}^{\widehat z_n}(G_n))= \frac{1}{m-1} \left\|  \left(A^{\widehat z_n}\right)^{-1}\left(\begin{array}{c} X^{(0,0)}(z_2 - \widehat z_n)\\ \vdots  \\  X^{(0,0)}(z_{m} - \widehat z_n)\end{array}\right)\right\|_2^2,
  \eeq
which is a composition of continuous functions of $\widehat{z}_n$. In particular, as $ \widehat z_n$ converges a.s. to  $\widehat z$ (see Lemma~\ref{lem:hat}, Claim $(b)$), we deduce that $\widehat \sigma ^2_n$ converges a.s. to $\widehat {\sigma}^2_{ \widehat z}$ as $n$ goes to infinity.

 Finally, since the KL  estimator is unique, this estimator coincide with the estimator $\hat{\sigma}_2^2$ of~\cite{ADCM16} and Theorem 3  of  \cite{ADCM16} implies that 
 $$ 
  \frac{ \overline F_{m-1} \big( \lambda_{1,n}/  \widehat \sigma_n \big)}{ \overline F_{m-1} \big( \lambda_{2,n}/  \widehat \sigma_n \big)  } \sim \mathcal U([0,1]).
  $$
Note that $\lambda_1^n$ converges almost surely to $\lambda_1$ and $\lambda_{2,n}$ converges in distribution to $\overline{\lambda}_2$ (see Lemma~\ref{lem:z2}) to complete the proof.

\subsection{Proof of Proposition~\ref{l:jm:l1}}
\label{proof:PropSRConditions}

(a). We can assume that $Z$ defined by~\eqref{eq:CorrZSR} is centered and, in this case, it holds
\begin{equation}\label{e:jm:1} 
\forall t \in [0,2\pi)\quad 
 Z(t)  = \frac \sigma{\sqrt N}  \sum_{k=-f_c}^{f_c} \zeta_{k} \exp(\imath k t),
 \end{equation}
 where we recall that $N= 2fc+1$ and $\zeta_{k} = \zeta_{k,1} + \imath \zeta_{k,2}$ for $k=-f_c,\ldots,f_c$ are independent standard complex Gaussian variables. Formula \eqref{e:jm:1} shows that $Z$ satisfies~\eqref{a:KLN}.   
 
 \noindent 
(b). Let $(t_1,\dots,t_N) \in [0,2\pi)$ be pairwise differents, $\theta \in [0,2\pi)$ and set
$$
\left(\begin{array}{c} Z(t_1 )\\ \vdots  \\  Z(t_N)\end{array}\right) =  \left(\begin{array}{c c c} \exp(-\imath f_c t_1) & \dots &  \exp(\imath f_c t_1) \\ \vdots & & \vdots  \\  \exp(-\imath f_c t_N) & \dots & \exp(\imath f_c t_N)\end{array}\right) \left(\begin{array}{c} \zeta_1\\ \vdots  \\  \zeta_N\end{array}\right) =: A_{t_1,\dots,t_N} \zeta,
$$
where $A_{t_1,\dots,t_N}$ is a Vandermonde matrix, invertible as soon as $t_i \neq t_j$ for all $i \neq j$. This prove the first point of \ref{a:nonDegeneratefiniteComplex}. For the second assertion, consider $h > 0$ such that $h < \min_{1\leq i<j\leq N-1}(t_i - t_j)$ and the Gaussian vector
$$
\left(Z(t_1),\dots, Z(t_{N-1}),Z(t_1 + h)\right)^T =: A_{t_1,\dots,t_{N-1},t_{1}+h} \zeta,
$$
where the covariance matrix $A_{t_1,\dots,t_{N-1},t_{1}+h}$ satisfies
\begin{align*}
\det(A_{t_1,\dots,t_1+h}^{*} A_{t_1,\dots,t_1+h} ) &= \prod_{1\leq i<j \leq N-1} |\exp(\imath t_i) -  \exp(\imath t_j)|^2 \prod_{j=1}^{N-1} |\exp(\imath (t_1+h)) - \exp(\imath t_j)|^2 \\
&= 4^{N(N-1)/2}  \prod_{1\leq i<j \leq N-1} \sin^2\left(\frac{t_i - t_j}{2}\right) \prod_{j=1}^{N-1} \sin^2\left(\frac{t_j - (t_1 + h)}{2}\right) \\
&= \sin^2(h/2) \times g_{t_1,\dots,t_{N-1}}(h) 
\end{align*}
where $g_{t_1,\dots,t_{N-1}}(0) \neq 0$ if $(t_i)_{1\leq i \leq N-1}$ are pairwise distincts. Finally, denote by $R_h$ the linear transformation involving the first and the last coordinate such that
$$
\left(\begin{array}{c} Z(t_1) \\\vdots \\ Z(t_{N-1}) \\ \frac{Z(t_1 + h) - Z(t_1)}{h}\end{array}\right) = R_h \left(\begin{array}{c} Z(t_1)\\\vdots \\ Z(t_{N-1}) \\ Z(t_1+h)\end{array}\right)
$$
and remark that
$$
\lim_{h\to 0} \det(A_{t_1,\dots,t_1+h}^{*} R_h^* R_h A_{t_1,\dots,t_1+h}) = g_{t_1,\dots,t_{N-1}}(0) \times \lim_{h \to 0} \frac{\sin^2(h/2)}{h^2} = g_{t_1,\dots,t_{N-1}}(0) \times \frac14 \neq 0 
$$
giving the desired  non degeneracy condition. 

\subsection{Proof of Proposition \ref{cor:rice_known_variance_blasso} and Proposition \ref{cor:rice_unknown_variance_blasso}}
\label{proof:PropSR}
Easy computations give the following results for $\phi(\cdot)$,
\[ \int_{\ell}^{+\infty} \phi(t) \mathrm{d}t = \overline{\Phi}(\ell), ~~ \int_{\ell}^{+\infty} t \phi(t) \mathrm{d}t = \phi(\ell), ~~ \int_{\ell}^{+\infty} t^2 \phi(t) \mathrm{d}t = \ell \phi(\ell) + \overline{\Phi}(\ell), \]
for $f_{m-1}(\cdot)$,
\begin{align*}
\displaystyle \int_{\ell}^{+\infty} & f_{m-1}\left( t \sqrt{\frac{m-1}{m-3}} \right) \mathrm{d}t = \sqrt{\frac{m-3}{m-1}}  ~\overline{F}_{m-1}\left( \ell \sqrt{\frac{m-1}{m-3}}\right), \\
\displaystyle \int_{\ell}^{+\infty} t & f_{m-1}\left( t \sqrt{\frac{m-1}{m-3}} \right) \mathrm{d}t = \frac{(m-3) \sqrt{m-3}}{(m-2)\sqrt{m-1}} ~\frac{\bar{\Gamma}\left(\frac{m}{2}\right) \bar{\Gamma}\left(\frac{m-3}{2}\right)}{\bar{\Gamma}\left(\frac{m-1}{2}\right) \bar{\Gamma}\left(\frac{m-2}{2}\right)} ~f_{m-3}(\ell),\\
\displaystyle \int_{\ell}^{+\infty} t^2 & f_{m-1}\left( t \sqrt{\frac{m-1}{m-3}} \right) \mathrm{d}t = \frac{(m-3) \sqrt{m-3}}{(m-2)\sqrt{m-1}} ~\frac{\bar{\Gamma}\left(\frac{m}{2}\right) \bar{\Gamma}\left(\frac{m-3}{2}\right)}{\bar{\Gamma}\left(\frac{m-1}{2}\right) \bar{\Gamma}\left(\frac{m-2}{2}\right)} \times \left( \ell f_{m-3}(\ell) + \overline{F}_{m-3}(\ell) \right),
\end{align*}
and for $R$,
\begin{align*}
X''(\hat{z}) &= -\tilde{\Lambda} X(\hat{z}) + R(\hat{z}), \\
&= - \left( \begin{array}{c c} 
\alpha_1 & 0\\
0 & 1\\
\end{array} \right) X(\hat{z}) 
+ 
\left( \begin{array}{c c} 
-\alpha_2 & \alpha_3\\ 
\alpha_3 & 0\\
\end{array} \right),
\end{align*}
where
  $$
  \left\{ \begin{array}{l}
  \alpha_1 = \frac{1}{3} f_c (f_c + 1), \\
  \alpha_2 = \frac{1}{\sqrt{N}} \sum\limits_{k=-f_c}^{f_c} (k^2 - \alpha_1) \times \mathrm{Re}(y_{k} e^{\imath (k \hat{t} - \hat{\theta})}),  \\
  \alpha_3 = \frac{1}{\sqrt{N}} \sum\limits_{k=-f_c}^{f_c} k \times \mathrm{Re}(y_{k} e^{\imath (k \hat{t} - \hat{\theta})}).
  \end{array} \right.
  $$
To conclude, use Proposition \ref{l:jm:l1} to apply Theorem \ref{thm:rice_known_variance} and Theorem \ref{thm:rice_unknown_variance}.

\section{Auxiliary results}

\subsection{Regularity of $X^{|z}$ and new expression of $R(z)$}
\begin{lemma} \label{l:jm1}

$X^{|z}(y) $ admits radials limits as $y\to z$. More precisely  for all $\lambda $ in the unit sphere 
$$
 \lim_{u\to 0} X^{|z}(z+u \lambda) = \frac{\lambda^\top R(z) \lambda
}{  \lambda^\top \widetilde \Lambda \lambda}.$$
 \end{lemma}
 
  \begin{proof}
  As $u$  tends to zero 
  $$
   1 - \rho(u \lambda)   =  \frac{u^2}{2} ( \lambda^\top \widetilde\Lambda \lambda + o(1)).
  $$
 Moreover, a Taylor expansion gives
  $$
  X(z+u\lambda)  = X(z) + u X_{\lambda}'(z) + \frac{u^2}{2} X_{\lambda}''(z) + o_p(u^2),
$$
and
$$ \rho_\lambda'(u\lambda) = u \rho_\lambda''(0) + o_p(u^2)  = -u \widetilde \Lambda + o_p(u^2), $$
  where $(X_{\lambda}',\rho_\lambda') $ and $(X_{\lambda}'',\rho_\lambda'') $ are directional derivative and directional Hessian. By consequence,
  \begin{align*}
X^{|z}(z+u \lambda) &= \frac{\frac{u^2}{2} X(z) \lambda^\top \widetilde\Lambda \lambda + \langle \rho'(u\lambda) ,\widetilde \Lambda^{-1} X'(z) \rangle + u X_{\lambda}'(z) + \frac{u^2}{2} X_{\lambda}''(z) + o_p(u^2)}{\frac{u^2}{2}( \lambda^\top \widetilde\Lambda \lambda + o(1))} \\
&= \frac{\frac{u^2}{2} \left( X(z) \lambda^\top \widetilde\Lambda \lambda + X_\lambda''(z) + o_p(1)\right)}{\frac{u^2}{2} (\lambda^\top \widetilde\Lambda \lambda + o(1))}
  \end{align*}
which tends to
$$\frac{ \lambda^\top \left(\widetilde\Lambda X(z) + X''(z) \right) \lambda}{\lambda^\top \widetilde\Lambda \lambda}
$$
as $u$ tends to $0$ since $X_{\lambda}''(z) = \lambda^\top X''(z)\lambda$. The result follows from $X''( z) = -\widetilde\Lambda X( z) +   R(z)$.
  \end{proof}

\subsection{Maximum of a continuous process}
The following result is borrowed from \cite[Theorem 3]{lifshits1983absolute} and \cite{tsirel1976density}.
\begin{proposition}
\label{prop:Unicity}
Let $ \{Y(t)\, ;\  t\in T\}$ be a Gaussian process with continuous sample paths defined on a compact metric space $T$. Suppose in addition that: 
\begin{equation} 
\label{e:lif}
\mathrm{There\ is\ no\ two\ different\ points\ s,t\ }\in T\ \mathrm{such\ that}\ X(s) = X(t)\ a.s.
\end{equation}
Then almost  surely  the maximum of $X$ on $T$  is attained at a single point.
\end{proposition}
\noindent
Observe that \eqref{e:Normalization} implies \eqref{e:lif}.
 
 \begin{remark}
 \label{rem:Pumping}
 Proposition \ref{prop:Unicity} can be applied to the process $X^{| \widehat z} $ which is not continuous on a compact set. We use  the ‘‘pumping method'' as follows. Use 
 \begin{itemize}
 \item[$(a)$] a parameterization  of $\bbT$  as $[0,2\pi)^2$,
 \item[$(b)$] polar coordinates for $y \in \bbT  \setminus\{\widehat z\}$ with origin at $\widehat z$, 
 \item[$(c)$] the change of parameter 
 $$
y=( \rho, \theta)  \mapsto  ((\rho+1) , \theta)
$$
that  transforms the non-compact set  $\bbT \setminus \{\widehat z\}$ into a compact set $($we have inflated the ‘‘hole'' $\{\widehat z\}$  into a ball centered around $\widehat z$ with radius one$)$ on which the process $X^{| \widehat z} $  is continuous thanks to Lemma~\ref{l:jm1}.
 \end{itemize} 
\end{remark}



 \bibliographystyle{abbrv}
 \bibliography{references_all}

\end{document}